\documentclass[reqno,a4paper,11pt]{amsart}

\usepackage[english]{babel}
\usepackage{latexsym,amsmath,amsthm, mathtools}
\usepackage{amsfonts,xcolor}
\usepackage{graphicx}
\usepackage[colorlinks=true, allcolors=blue]{hyperref}
\usepackage{comment}
\usepackage{fullpage,hyperref,cleveref,amssymb,mathtools,ifthen,thmtools,tikz,bbm}
\usepackage{thm-restate}
\usepackage[shortlabels]{enumitem}

\usepackage[backend=biber,giveninits=true,doi=false,url=false,isbn=false,style=trad-plain,hyperref,date=year]{biblatex}
\bibliography{references}

\newtheorem{theorem}{Theorem}[section]
\newtheorem{definition}[theorem]{Definition}

\newtheorem{corollary}[theorem]{Corollary}
\newtheorem{lemma}[theorem]{Lemma}
\newtheorem{conjecture}[theorem]{Conjecture}
\newtheorem{claim}[theorem]{Claim}
\newtheorem{question}[theorem]{Question}

\newtheorem{remark}[theorem]{Remark}

\newcommand{\cA}{\mathcal{A}}
\newcommand{\cB}{\mathcal{B}}
\newcommand{\cC}{\mathcal{C}}
\newcommand{\cD}{\mathcal{D}}
\newcommand{\cE}{\mathcal{E}}
\newcommand{\cF}{\mathcal{F}}
\newcommand{\cG}{\mathcal{G}}

\newcommand{\cI}{\mathcal{I}}
\newcommand{\cJ}{\mathcal{J}}

\newcommand{\cL}{\mathcal{L}}
\newcommand{\cM}{\mathcal{M}}

\newcommand{\cO}{\mathcal{O}}
\newcommand{\cP}{\mathcal{P}}

\newcommand{\cS}{\mathcal{S}}

\newcommand{\ta}{\tilde{\alpha}}

\newcommand{\lam}{\lambda}

\newcommand{\om}{\omega}
\newcommand{\lambexp}{C_{\ast}}

\newcommand{\Gam}{\Gamma}
\newcommand{\Om}{\Omega}

\renewcommand{\Pr}{\mathbb{P}}
\newcommand{\E}{\mathbb{E}}
\newcommand{\Bin}{\operatorname{Bin}}

\providecommand\given{\nonscript\:\ifthenelse{\equal{\delimsize}{}}{\big\vert}{\delimsize\vert}\nonscript\:\mathopen{}}
\let\Pr\undefined
\DeclarePairedDelimiterXPP\Pr[1]{\mathbb{P}}(){}{#1}
\DeclarePairedDelimiterXPP\Ex[1]{\mathbb{E}}{[}{]}{}{#1}
\DeclarePairedDelimiter{\norm}{\lVert}{\rVert}

\title{Counting independent sets in percolated graphs via the Ising model}
\author{Anna Geisler, Mihyun Kang, Michail Sarantis, Ronen Wdowinski}

\address{Institute of Discrete Mathematics, Graz University of Technology, Steyrergasse 30, 8010 Graz, Austria}

\email{\{geisler,kang,sarantis,wdowinski\}@math.tugraz.at}

\begin{document}

\begin{abstract}
Given a graph $G$, we form a random subgraph $G_p$ by including each edge of $G$ independently with probability $p$. We provide an asymptotic expansion of the expected number of independent sets in random subgraphs of regular bipartite graphs satisfying certain vertex-isoperimetric properties, extending the work of Kronenberg and Spinka on the percolated hypercube. Combining graph containers with the cluster expansion from statistical physics, we give an expansion of the partition function of the Ising model in a certain range of the parameters. Among other applications, we obtain results for even tori of growing side-length. As a tool, we prove a refined container lemma for the Ising model, which mildly improves recent bounds of Jenssen, Malekshahian, and Park.

\vspace{0.3cm}
\noindent \textbf{Keywords.} independent sets, Ising model, container method, cluster expansion
\end{abstract}

\maketitle

\section{Introduction}
\subsection{Motivation and background}
An independent set in a graph $G=(V(G), E(G))$ is a vertex subset $I\subseteq V(G)$ such that no two vertices in $I$ are connected by an edge. We denote by $\cI(G)$ the collection of independent sets of $G$ and set $i(G)\coloneqq|\cI(G)|$. 
Independent sets arise naturally in various combinatorial and algorithmic problems, and thus their study has been a long-standing topic inspiring a variety of methods and results.
This includes a wealth of results on problems concerning their enumeration, sampling, algorithmic hardness, as well as structural and extremal properties. Since the problem of exact enumeration is known to be hard \cite{bezakova2018inapproximability, sly2014counting}, it is natural to pursue asymptotic results. One of the main techniques in this direction is via a probabilistic analysis of the hard-core model.

Let $G$ be a finite graph. The hard-core model on $G$ with \textit{fugacity} $\lam> 0$ samples independent sets $I\in\cI(G)$ with probability
$$\mathbb{P}(\boldsymbol{I}=I)=\frac{\lam^{|I|}}{Z_G(\lam)},$$
where
$$Z_G(\lam)=\sum_{I\in\cI(G)}\lam^{|I|}$$
is  the \textit{partition function} of the model, also known as the \textit{independence polynomial} of $G$. Since $Z_G(1)=i(G)$, studying the partition function in a range including $\lam=1$ implies a corresponding enumerative result. 

A graph that has drawn particular interest is the $d$-dimensional hypercube $Q^d$, whose vertex set is $\{0,1\}^d$ and two vertices are adjacent if they differ in exactly one coordinate. A celebrated and influential result of Korshunov and Sapozhenko \cite{korshunov1983number} states that
$$i(Q^d)=(1+o(1))2\sqrt{e}2^{2^{d-1}}.$$
Galvin \cite{galvin2011threshold} extended their technique to the hard-core model on $Q^d$ and described the typical structure of independent sets, as well as approximations of the partition function, for various values of the fugacity $\lam$. Later, Jenssen and Perkins \cite{jenssen2020independent} proved that
\begin{equation*}\label{eqn:hypercube}
    i(Q^d)=2\sqrt{e} 2^{2^{d-1}} \left(1+ \frac{3d^2-3d-2}{8 \cdot 2^d}+O\left(d^4 2^{-2d}\right)\right),
\end{equation*}
and furthermore they provide an algorithm to compute any number of terms of the cluster expansion.

Questions about independent sets are also an active area of research on various random graphs models \cite{bezakova2024fast, CoEf15, frieze1990independence, krivelevich2024greedy, krivelevich2003probability, nikoletseas2008large}. 
The percolated graph $G_p$ is a random graph obtained from the base graph $G$ after independently retaining each edge with probability $p$.
Percolation is a broad and deep subject on its own (see e.g., \cite{bollobas2006percolation, grimmett}), and in particular the percolated hypercube $Q^d_p$ has been extensively studied for its various structural and combinatorial properties \cite{ajtai1982largest, bollobas1983evolution, BKL92, erde2023expansion, erdos1979evolution, hulshof2020slightly, van2017hypercube}. 
It is natural to extend the result of Korshunov and Sapozhenko \cite{korshunov1983number} to its percolated setting, and
Kronenberg and Spinka \cite{kronenberg2022independent} initiated this line of work and studied the random variable $i(Q^d_p)$ (see also \cite{chowdhury2024gaussian} for a more recent work on the distribution of $i(Q^d_p)$). The first key observation of Kronenberg and Spinka was the equivalence, in expectation, of the hard-core model on percolated graphs to the antiferromagnetic Ising model in the unpercolated setting, thus eliminating the randomization of the graph. 

The antiferromagnetic Ising model on a graph $G$ has two parameters, the fugacity $\lam> 0$ and the \textit{inverse temperature} $\beta\in (0,\infty)$, and it samples any subset $I \subseteq V(G)$ of vertices with probability
$$\mathbb{P}(\boldsymbol{I}=I)=\frac{\lam^{|I|}e^{-\beta|E(I)|}}{Z_G(\lam,\beta)},$$
where 
$$Z_G(\lam,\beta)=\sum_{I\subseteq V(G)}\lam^{|I|}e^{-\beta |E(I)|}$$
is the partition function. Note that configurations $I$ with neighboring sites occupied are allowed, but are exponentially penalized by $\beta |E(I)|$ according to the number of such instances. The hard-core model may informally be viewed as the special case when $\beta=\infty$.
An elementary, yet important observation of Kronenberg and Spinka \cite[Proposition 1.7]{kronenberg2022independent} is that for $p=1-e^{-\beta}$,
\begin{equation}\label{eqn:hc_to_ising}
    \E[Z_{G_p}(\lam)]=Z_G(\lam,\beta).
\end{equation}
 All the theorems in \cite{kronenberg2022independent}, as well as in our current work, are then derived from the study of the Ising model on the base graph $G$. The common strategy for deriving the asymptotics of the partition functions above is verifying the convergence of the cluster expansion of the partition function's logarithm, and this also involves finding an appropriate graph container lemma.

\subsection{Main themes and results.}

The goals of this paper are two-fold. Our first goal concerns asymptotic enumeration of independent sets and explores how far beyond we can go from the hypercube $Q^d$, so that we can derive similar asymptotic results for more general regular bipartite graphs with similar mild vertex-expansion properties. We describe a set of technical conditions for a graph $G$ under which a container lemma and the convergence of the cluster expansion for the Ising and the hard-core model are guaranteed, generalizing the work of Kronenberg and Spinka \cite{kronenberg2022independent} on the hypercube. The parameterized isoperimetric conditions we impose allow us to apply our results to graphs with weaker vertex-expansion than previously considered examples.

Our second goal is to improve the graph container lemma applicable to the Ising model. Such a graph container lemma is also of interest for other purposes, but here it allows us to show that our asymptotic results hold for values of the fugacity $\lambda$ beyond previous works. Previous results in the hard-core model had to assume that $\lambda = \tilde{\Omega}(d^{-1/3})$ in order to apply a container lemma of Galvin \cite{galvin2011threshold} (where $\tilde{\Omega}$ hides a polylogarithmic factor). Jenssen, Malekshahian, and Park \cite{jenssen2024refined} recently proved a refined container lemma allowing them to improve the range of validity to $\lambda = \tilde{\Omega}(d^{-1/2})$. We give an alternative proof of this result that also works in the Ising model while also slightly improving the polylogarithmic factor. The main ingredients of our proof are entropy tools by Peled and Spinka \cite{peled2020long}.

\subsubsection{Vertex-isoperimetry}\label{sec:isoperimetry}
As noted in the past, e.g., in \cite{galvin2019independent, jenssen2020independent}, the key properties of the hypercube relevant to the container lemma and the cluster expansion are its isoperimetric inequalities and the fact that it is regular bipartite. Apart from the hypercube, isoperimetric inequalities were used in similar settings for the middle layers of the Boolean lattice \cite{balogh2021independent}, even tori \cite[Lemma 6.1]{jenssen2023homomorphisms}, and, more generally, Cartesian product graphs \cite{CEGK25}.
The similarities among these results are evident, and it is therefore natural to expect the method of \cite{jenssen2020independent,kronenberg2022independent} to work for a broad class of regular bipartite graphs with vertex-isoperimetric properties \textit{akin} to the hypercube. As stated before, one of our goals in this paper is to precisely pin down and quantify this statement, as well as relate it to the strength of the derived results. 
Furthermore, we uncover a hidden constraint on the size of the base graph $G$, which can be at most exponential in terms of its regularity.

\begin{definition}[Property I]\label{def:property-i}
Fix positive real constants $C_1, C_2, C_3, C_4, C_5$ satisfying $C_5< 2$ and $C_3>C_5+2$. Let $G$ be  a  $d$-regular $n$-vertex bipartite graph with bipartition $(\cO, \cE)$. We say $G$ satisfies {\bf Property I} if the following hold.
\begin{enumerate}
\item[\upshape{\textbf{Ia.}}]\label{expansionPI} For every subset $X \subseteq \cO$ or $X \subseteq \cE$,
        \begin{enumerate}
        \item[\upshape{(1)}]\label{iso-gen} $|N(X)| \geq (d-C_1|X|)|X|$,
          \item[\upshape{(2)}]\label{iso-mid} $|N(X)| \geq \frac{d}{C_2}|X|$ if $|X| \leq d^{C_3}$,
          \item[\upshape{(3)}]\label{iso-large} $|N(X)| \geq (1+\frac{C_4}{d^{C_5}})|X|$ if $|X| \leq \frac{3}{8}n$.
      \end{enumerate}
\item[\upshape{\textbf{Ib.}}]\label{sizePI} The number of vertices satisfies $n=\om(d^{C_5+5})$ and $\log n=O(d)$.
\end{enumerate}
\end{definition}

We note that, although the size constraint $\log n=O(d)$ was  embedded in the above mentioned examples and has not been a barrier for previous applications, it does pose an obstacle in the study of even tori $\mathbb{Z}_m^t$ for growing (even) $m$. The motivation behind studying $\mathbb{Z}_m^t$ for growing $m$ is a conjecture of Jenssen and Keevash \cite[Conjecture 19.1]{jenssen2023homomorphisms},  essentially stating that, for a fixed $t$ but $m\rightarrow\infty$, the first term of the cluster expansion should give a good approximation for the partition function of (weighted) \textit{homomorphisms} from the torus to any fixed graph. This conjecture was indeed proved by Peled and Spinka \cite{peled2020long}, for any graph with fixed vertex and edge weightings under certain symmetry conditions. On the other hand, we will prove that we may allow the fugacity and temperature to diminish as $t$ grows, while $m=\Omega_t(1)$. In this direction, we use a modified isoperimetric condition from a recent work \cite{CEGK25} that allows us to dispose of the size constraint, and leaves as an open direction the discovery of more families of graphs which fall in this category -- see \Cref{sec:tori} for further discussions. 

\subsubsection{Main results}

Our results will differ depending on whether $0<C_5 \leq 1$ or $C_5>1$, so it will be useful to set
  \[
  \lambexp \coloneqq \min\left\{\frac{1}{2}, 1-\frac{C_5}{2}\right\}.
   \]

\noindent Our first result gives an estimate of the expectation $\E[Z_{G_p}(\lam)]$ of the hard-core partition function of the percolated graph $G_p$.
\begin{theorem} \label{thm:expansion}
    Let $G$ be a $d$-regular $n$-vertex bipartite graph with bipartition $(\cO, \cE)$ satisfying \hyperref[def:property-i]{Property I}.
    For any fixed constant $\lambda_0>0$ there exists a constant $C_0>0$ such that the following holds. Suppose $\lam \le \lam_0$ and $\lam p \geq \frac{C_0 \log^{3/2} d}{d^{\lambexp}}$. Then for every $j \in \mathbb{N}$ and $\mathcal{D} \in \{\mathcal{O}, \mathcal{E}\}$, there exist functions $L_{\mathcal{D}, j}=L_{\mathcal{D}, j}(n,d,\lambda)$ and $\varepsilon_{\mathcal{D}, j}=\varepsilon_{\mathcal{D}, j}(n,d,\lambda)$ 
such that for every $k \in \mathbb{N}$, as $d \to \infty$,
  \[
    \Ex{Z_{G_p}(\lam)}  =  (1+\lam)^{n/2} \sum_{\cD \in \{\cO, \cE\}} \exp\left(\frac{n \lam}{2} \left(1-\frac{\lam p}{1+\lam}\right)^d + \mathbbm{1}_{k\geq 2} \cdot \sum_{j=2}^{k} L_{\cD, j} + \varepsilon_{\cD, k} \right),
    \]
where   
    \[
    |L_{\cD, j}|=O\left(n d^{2(j-1)} \lam^{j} \left(1-\frac{\lam p}{1+\lam}\right)^{d j} \right) \hspace{10pt} \textit{and} \hspace{10pt}
    |\varepsilon_{\cD, j}| =O\left(nd^{2j} \lam^{j+1} \left(1-\frac{\lam p}{1+\lam}\right)^{d(j+1)}\right).
  \]
  
\end{theorem}

By taking $G=Q^d$ in \Cref{thm:expansion}, we recover the results on the Ising partition function from \cite{kronenberg2022independent} and, by taking $p=1$, the results on the hard-core partition function for {\em unpercolated} graphs from \cite{CEGK25}.  The expansion of the Ising partition function can be translated to a counting result {regarding the expected number of independent sets in a percolated graph $G_p$} by taking $\lam=1$ and $k=1$ in \Cref{thm:expansion}, as follows.

\begin{theorem} \label{thm:numberindsets}
  Let $G$ be a $d$-regular $n$-vertex bipartite graph with bipartition $(\cO, \cE)$ satisfying \hyperref[def:property-i]{Property I}.
  There exists a constant $C_0>0$ such that if $p \geq \frac{C_0 \log^{3/2} d}{d^{\lambexp}}$, then as $d \to \infty$,
  \[
    \Ex{i(G_p)} = 2 \cdot 2^{n/2} \exp\left(\frac{n}{2^{d+1}}\left(2-p\right)^d + O\left(\frac{nd^2}{2^{2d}}(2-p)^{2d}\right)\right).
    \]
\end{theorem}

Note that \Cref{thm:numberindsets} provides the asymptotics for the expected number of independent sets on a logarithmic scale. As in many instances, if a $(1+o(1))$-factor approximation is desired, a finite number of terms is sufficient in order to get such an approximation for $\Ex{i(G_p)}$ for any $d$-regular $n$-vertex bipartite $G$ graph satisfying \hyperref[def:property-i]{Property I} and any fixed $p$.
\begin{remark}\label{rem:sharp}
    Let $k\geq 1$ be a fixed integer and let $\ell\geq 0$ a constant such that $\log n\leq \ell d$. Then, for any fixed $p$ satisfying $1\geq p> 2-2^{1-\frac{\ell}{k+1}}$, it suffices to consider the first $k$ terms of the expansion. In particular, as $d \to \infty$,
    \[
    \Ex{i(G_p)}=(1+o(1)) 2^{n/2} \sum_{\cD \in \{\cO, \cE\}} \exp\left(\frac{n}{2^{d+1}}\left(2-p\right)^d + \mathbbm{1}_{k\geq 2}\cdot \sum_{j=2}^{k} L_{\cD, j}\right),
    \]
    where \[|L_{\cD, j}|=O\left(n d^{2(j-1)} \left(1-\frac{p}{2}\right)^{d j} \right).\]
\end{remark}

\subsubsection{Graph containers and the cluster expansion}

From a structural and formal points of view, the proof of our results can be traced back to the vertex-isoperimetric properties of a (regular bipartite) graph as discussed in \Cref{sec:isoperimetry}. Another essential ingredient, indeed the heart of the proof, is \textit{Sapozhenko's container method}, which has been widely applied in enumerative problems on independent sets in graphs \cite{balogh2025sharp, galvin2006slow, kahn2022number, park2022note, potukuchi2021enumerating}. The main idea is that bounds on the number of independent sets can be obtained via bounds on the number of sets of a given size with a neighborhood of a given size. This idea was initiated by Sapozhenko, who introduced the method for the purpose of asymptotically counting the number of antichains of the Boolean lattice \cite{sapozhenko1991number}, and the number of independent sets in the hypercube together with Korshunov \cite{korshunov1983number}. For an excellent exposition of the latter result, and the container method in general, see the exposition by Galvin \cite{galvin2019independent}. We note that container methods have been extended to hypergraphs independently by Balogh, Morris, and Samotij \cite{BALOGH_2019} and by Saxton and Thomason \cite{Saxton_2015} and have since found a wealth of applications \cite{Balogh_2023, Balogh_2019, Mattheus_2025, mubayi2023random}, while improvements, variations and extensions of hypergraph containers remain an active area of research \cite{campos2024towards, Nenadov_2023, nenadov2024short, nenadov2024containers}.

Galvin \cite{galvin2011threshold} generalized Sapozhenko's approach to the hard-core model and obtained asymptotics for the partition function of the hypercube $Q^d$ with fugacity $\lam=\Om\left(\frac{\log d}{d^{1/3}}\right)$: Among others, he proved that as $d \to \infty$,
\begin{equation}\label{eqn:approx_galvin}
    Z_{Q^d}(\lam)=2(1+\lam)^{2^{d-1}}\exp\left(\frac{\lam}{2}\left(\frac{2}{1+\lam}\right)^d(1+o(1))\right).
\end{equation}
Jenssen and Perkins \cite{jenssen2020independent} pioneered a new, fruitful line of work, where the framework of graph containers is used in conjunction with \textit{abstract polymer models} and the \textit{cluster expansion} from statistical physics.
It has since been used to answer several enumerative and algorithmic questions in combinatorics, see, e.g., \cite{jenssen2023homomorphisms, jenssen2024dedekind, jenssen2022independent, li2023number}, or \cite{davies2021proof, jenssen2023evolution, jenssen2024sampling} for more applications without the use of the container method.
In their work, Jenssen and Perkins  extended the results of Galvin \cite{galvin2011threshold} to obtain asymptotics for the hard-core partition function on $Q^d$ in the range $\lam=\Om\left(\frac{\log d}{d^{1/3}}\right)$.
Subsequent applications of the method have yielded analogous results \cite{balogh2021independent, jenssen2023approximately}. In all cases, the bound $\lam=\Om\left(\frac{\log d}{d^{1/3}}\right)$ appears as the barrier of the graph container approach following Galvin \cite{galvin2011threshold}, which stems from the optimization of the container parameters in the resulting upper bound of the partition function.
In the same work, Galvin conjectures the derived asymptotics for $Z_G(\lam)$ should hold for the extended range $\lam=\tilde{\Om}\left(\frac{1}{d}\right)$, which would be optimal for the hypercube  $Q^d$ (see \cite{galvin2011threshold}, specifically the paragraph after Theorem 1.6). The first result which broke the barrier of $d^{-1/3}$ was recently obtained by Jenssen, Malekshahian, and Park \cite{jenssen2024refined}, who showed that $\lam=\tilde{\Om}(d^{-1/2})$ is enough for so-called \textit{approximately biregular} graphs that satisfy certain vertex-expansion conditions akin to our \hyperref[def:property-i]{Property I}. Their argument is based on a new, iterative $\psi$-approximation to produce containers \cite[Lemma 2.4]{jenssen2024refined}. 

In light of these exciting recent developments, on one hand we similarly improve the range of the validity for the Ising and hard-core container lemmas in previously considered examples, such as the hypercube or the middle layer graph (see \Cref{lem:container_main} and Corollaries \ref{cor:hard-core-corollary-1}, \ref{cor:hard-core-corollary-2}). On the other hand, we clarify the dependence between the parameters of \hyperref[def:property-i]{Property I} and the valid range of the fugacity $\lambda>0$ and inverse temperature $\beta>0$.
For example, our main container lemma (\Cref{lem:container_main}) implies that if $\lam$ is bounded and $\lam(1-e^{-\beta})=\Om\left(\frac{\log^{3/2}d}{d^{1/2}}\right)$, then as $d \to \infty$,
$$Z_G(\lam,\beta)=2(1+\lam)^{n/2}\exp\left(\frac{n\lam}{2}\left(\frac{1+\lam e^{-\beta}}{1+\lam}\right)^d(1+o_d(1))\right)$$
for any $d$-regular $n$-vertex graph $G$ satisfying \hyperref[def:property-i]{Property I} with $0<C_5\leq 1$. If $1 < C_5 < 2$, then the same estimate on $Z_G(\lam,\beta)$ holds when $\lam(1 - e^{-\beta}) = \Omega\left(\frac{\log^{3/2}d}{d^{1-C_5/2}}\right)$. Both of our Ising and hard-core container results provide the best-known lower bound on $\lam$, namely $$\Om\left(\frac{\log^{3/2}d}{d^{1/2}}\right),$$  for $d$-regular bipartite graphs in the range of interest of previous work, i.e., $0 < C_5 \le 1$.
In particular, this improves the current best bound of \cite{jenssen2024refined} by a $\log^{1/2}d$-factor. The detailed statements of our container lemma can be found in \Cref{sec:containers}. Our approach follows the same lines as Kronenberg and Spinka \cite{kronenberg2022independent}, but more carefully optimizes their proof. This approach is different from that of \cite{jenssen2024refined}, in that we still use the classical $\psi$-approximation for containers, but combine it with entropy-derived lemmas by Peled and Spinka \cite{peled2020long} (see \Cref{sec:entropy_tools}).

\subsection{Organization}
After a preliminary discussion and definitions in \Cref{sec:prelim}, \Cref{sec:proof_main} sets up the polymer model for our purposes and states the main properties relating its partition function and its derived measure to the originals from the hard-core model. We then present a self-contained proof of our main theorem using these properties. \Cref{sec:containers} is dedicated to containers, and includes our main container lemma for the Ising model and the corollaries for the hard-core model (for unpercolated graphs). In \Cref{sec:entropy_tools} we present the entropy-derived bounds, which will be used to prove our key lemma (\Cref{lemma:total-weight-given-approximation}) in \Cref{sec:bound_given_approx}. \Cref{sec:nonpolymer} deals with the contribution of configurations not captured by the polymer model. All these bounds are then used in \Cref{sec:kotecky} to prove the convergence of the cluster expansion by verifying the Koteck\'y--Preiss condition, while the main properties of the derived measure are proved in \Cref{sec:propertiesmu}. \Cref{sec:examples} presents some concrete examples and, finally, \Cref{sec:tori} discusses the case of {even tori of growing side-length} and the modifications required to treat them. We conclude with a short discussion of the results and possible avenues for future research.

\section{Preliminaries}\label{sec:prelim}

\subsection{Notation and basic facts}

Throughout the paper, whenever we consider a $d$-regular $n$-vertex graph satisfying \hyperref[def:property-i]{Property I}, we in fact consider a sequence $(d_{\ell}, n_{\ell})_{\ell \in\mathbb{N}}$ of pairs $(d_{\ell},n_{\ell})\in \mathbb{N}^2$ giving rise to a sequence $(G_{\ell})_{\ell \in\mathbb{N}}$ of $d_{\ell}$-regular $n_{\ell}$-vertex graphs such that there exists a $\ell_0 \in\mathbb{N}$ such that for every $\ell \geq \ell_0$ the graphs $G_{\ell}$ satisfy \hyperref[def:property-i]{Property I}. All asymptotics are with respect to $d_{\ell} \to \infty$ and we use standard Landau-notation for asymptotics.

For ease of presentation, we will omit floor/ceiling signs and assume that $n$ is even whenever necessary throughout the paper. Unless explicitly stated otherwise, all logarithms have the natural base.  A tilde over any asymptotic notation signifies the presence of polylogarithmic factors.

Let $G=(V(G),E(G))$ be a graph. Given a set $A \subseteq V(G)$ we denote the \emph{(external) neighborhood} of $A$ by $N_G(A) \coloneqq \{v \in V(G) \setminus A : \text{ there exists } w \in A \text{ with } vw \in E(G)\}$, and we write $N_G(v)$ for $N_G(\{v\})$. For a vertex $v \in V(G)$, we denote by $d_G(v)$ the \emph{degree} of $v$, that is, $d_G(v)\coloneqq|N_G(v)|$. Given two distinct vertices $u,v\in V(G)$, we denote by $\operatorname{codeg}(u,v)$ their codegree, i.e., $\operatorname{codeg}(u,v)\coloneqq |N_G(u)\cap N_G(v)|$.  
The \textit{maximum codegree} of $G$ is defined as $\max\{\operatorname{codeg}(u,v):\ u,v\in V(G), u\neq v\}$. For two disjoint subsets $A, B\subseteq V(G)$ we denote by $E_G(A, B)$ the set of edges with one endpoint in $A$ and one endpoint in $B$, while $E_G(A)$ is the sets of edges with both endpoints in $A$.
Whenever the graph $G$ is clear from context we may suppress the index $G$ in this notation.  
A vertex subset $A \subseteq V(G)$ is \textit{$2$-linked} if $A$ induces a connected subgraph of $G^2$, where the graph $G^2$ has the same vertex set as $G$ and $uv$ is an edge if the distance between $u$ and $v$ in $G$ is at most $2$. The following lemma is well known and can be found, e.g., in \cite{GaKa2004}.

\begin{lemma}\label{l:counting2linked}
  Let $\ell \in \mathbb{N}$ and $G$ be a $d$-regular graph. Fix a vertex $v \in V(G)$. The number of $2$-linked subsets of size $\ell$ which contain $v$ is at most $(ed^2)^{\ell-1}$.
\end{lemma}

We will need some basic tools from probability. Given a non-negative real random variable $X$ and $r \geq 0$, the \emph{cumulant generating function} $K_X(r)$ is defined as
\[
  K_X(r)\coloneqq \log \mathbb{E}\left[e^{rX}\right].
\]
For such a random variable, a simple application of Markov's inequality yields that for any $a>0$ and $r \geq 0$,
\begin{equation}\label{eq:Markov}
  \mathbb{P}(X \geq a) = \mathbb{P}\left(e^{rX} \geq e^{ra}\right) \leq e^{-ra+K_X(r)}.
\end{equation}

Given two probability measures $\mu$ and $\hat{\mu}$ on a sample space $\Omega$, the \emph{total variation distance} between them is defined by
\begin{equation}\label{eq:totalvariationdist}
  \norm{\hat{\mu}-\mu}_{TV} \coloneqq \frac{1}{2} \sum_{\om \in \Om} \left|\hat{\mu}(\om) - \mu(\om)\right|.
\end{equation}

We will also use a standard version of the Chernoff bound for binomial distributions, see \cite[Appendix A]{alon2016probabilistic}.
\begin{lemma}\label{l:Chernoff}
  Let $N \in \mathbb{N}$, $0 < p < 1$, and $X \sim \Bin(N,p)$. Then for every $r \geq 0$,
  \[
    \Pr[\big]{|X - Np| \geq r} \; \le \; 2\exp\left( -\frac{2r^2}{N} \right).
  \]
\end{lemma}

\subsection{Abstract polymer models and the cluster expansion}
Let $\cP$ be a finite set of objects, which we call \textit{polymers}, equipped with a \textit{weight function} $w:\cP\rightarrow [0,\infty)$. We also equip $\cP$ with a symmetric, anti-reflexive relation $\sim$, and we say that two polymers $A,A'\in \cP$ are \textit{compatible} if $A\sim A'$ and \textit{incompatible} otherwise, denoting the latter by $A\nsim A'$. We call a set of pairwise compatible polymers a \textit{polymer configuration} and let $\Om$ denote the collection of polymer configurations. The triple $(\cP,w,\sim)$ is called a \textit{polymer model} and we define its \textit{partition function} as
$$\Xi_\cP\coloneqq\sum_{\Theta\in\Om}\prod_{A\in \Theta} \om(A).$$

For an ordered  tuple $\Gamma$ of polymers, the \emph{incompatibility graph} $H(\Gamma)$ is the graph with vertex set $\Gamma$, in which there is an edge between any two elements of $\Gamma$ that are incompatible.
A \textit{cluster} is an ordered tuple of polymers for which the incompatibility graph is connected. We denote by $\cC$ the collection of all clusters. We extend the weight function to clusters by setting
$$\om(\Gam)\coloneqq\phi(H(\Gam))\prod_{A\in \Gam}\om(A),$$
where $\phi$ is the Ursell function of graphs, defined as
$$\phi(G)\coloneqq\frac{1}{|V(G)|!}\sum_{\substack{E\subseteq E(G)\\ \text{spanning, connected}}}(-1)^{|E|}.$$
The \textit{cluster expansion} is the formal power series for $\log\Xi(\cP)$ given by
\begin{equation}\label{eq:cluster_exp}
    \log\Xi_\cP=\sum_{\Gam\in\cC}\om(\Gam).
\end{equation}
Note that this is indeed an infinite formal series: Although the set of polymers is finite, the set of clusters is infinite as there exist arbitrary long sequences of incompatible polymers (for example, by repeating the same polymer any number of times). In fact, this is the multivariate Taylor expansion of $\log \Xi_\cP$. The equality \eqref{eq:cluster_exp} was first observed by Dobrushin \cite{dobrushin1996estimates}.

Our goal is to use the cluster expansion to get quantitative bounds on $\Xi_\cP$, which in turns requires the convergence of (\ref{eq:cluster_exp}), as well as bounds on its tails. A very useful theorem, which provides both, is the verification of the so-called \textit{Koteck\'y--Preiss condition} \cite{kotecky1986cluster}. For any function $g:\cP\rightarrow[0,\infty)$ and cluster $\Gam\in \cC$, define $g(\Gam) \coloneqq \sum_{A\in\Gam}g(A)$. We say that a polymer $A\in \cP$ is incompatible with a cluster $\Gam\in \cC$ if there exists some polymer $A'\in \Gam$ such that $A\nsim A'$, and we denote this by $\Gam\nsim A$.

\begin{lemma}[Koteck\'y--Preiss condition \cite{kotecky1986cluster}]\label{lem:KP}
    Let $(\cP,\om,\sim)$ be a polymer model and $f,g:\cP\rightarrow[0,\infty)$ be two functions such that for every polymer $A\in \cP$,
    \begin{equation}\label{eq:KP_cond}
        \sum_{A'\nsim A}|\om(A')|e^{f(A')+g(A')}\leq f(A).
    \end{equation}
    Then the cluster expansion (\ref{eq:cluster_exp}) converges absolutely. Furthermore, 
    \begin{equation}\label{eq:KP_tail}
        \sum_{\substack{\Gam\in \cC,\\ \Gam\nsim A}}\om(\Gam)e^{g(\Gam)}\leq f(A).
    \end{equation}
\end{lemma}

Note that if (\ref{eq:KP_cond}) holds for some function $g$, then it also holds for $g=0$. The benefit of a larger $g$ is the improved bounds via (\ref{eq:KP_tail}), which will allow us to get a better control of the tails of the cluster expansion and hence better approximations of the partition function.

\section{Our polymer model: Proof of \Cref{thm:expansion}}\label{sec:proof_main}

We let $G$ be a $d$-regular $n$-vertex bipartite graph with bipartition $(\cO, \cE)$ satisfying \hyperref[def:property-i]{Property I}, unless stated otherwise. For any statement that refers to a partition function or that contains parameters and objects related to a graph without the graph being specified, it will always refer to such a graph $G$. In the statements of our lemmas, we will always implicitly assume that $d$ is sufficiently large.

First, we set up the polymer models specific to our setting, following Kronenberg and Spinka \cite{kronenberg2022independent}. 
We refer to the bipartition classes of $G$ as the \textit{odd}  and \textit{even} sides of $G$, denoted by $\cO, \cE$. We will use $\cD$ to denote one of the sides of $G$, i.e., $\cD\in\{\cE,\cO\}$, and we use $\overline{\cD} \coloneqq V(G) \setminus \cD$ to denote its complement. The (bipartite) \textit{closure} of a set $A\subset \cD$ is defined as
$$[A]\coloneqq\{v\in \cD:N(v)\subseteq N(A)\}.$$ 
Fix parameters $\lam,\beta > 0$. We define the \textit{even} polymer model as 
\begin{align*}
   \cP_\cE\coloneqq \left\{A\subseteq\cE: \quad A\ \text{is $2$-linked\quad and} \quad |[A]|\leq \frac{3}{4}\cdot |\cE|\right\}.
\end{align*}
Each element $A \in \cP_\cE$ is called an \textit{even polymer}. Two even polymers are \textit{compatible} if their union is not $2$-linked. The weight of an even polymer $A\in\cP_\cE$ is defined as
$$\om(A)\coloneqq\sum_{B\subseteq N(A)}\frac{\lam^{|A|+|B|}}{(1+\lam)^{|N(A)|}}e^{-\beta|E(A,B)|}.$$
The \textit{odd} polymer model $\cP_\cO$ is defined analogously to $\cP_\cE$. Note that in the definition of the polymer models the constant $\frac{3}{4}$ in the upper bound of the size of $[A]$ is not special, and any value in $(\frac{1}{2},1)$ would suffice. Unlike the deterministic setting, not every set under consideration will be captured by either the even or odd polymer model, but the choice of a constant in $(\frac{1}{2},1)$ allows us to show that the fraction of these uncaptured sets is negligible (\Cref{lem:nonpolymer}).
The following upper bound on the weight function follows from a simple calculation (see \Cref{proof:boundingweight}).

\begin{lemma} \label{l:boundingweight}
    For any even polymer $A \in \cP_\cE$, we have
\begin{align*}
    \om(A) \le \frac{\lam^{|A|}}{(1+\lam)^{|N(A)|}} (1+\lam e^{-\beta})^{|N(A)|}.
\end{align*}
\end{lemma}

A set of pairwise compatible even polymers is called an \textit{even polymer configuration}.
Let $\Om_\cE$ denote the collection of even polymer configurations and let $\Xi_\cE$ denote its corresponding partition function. Also let $\Om_{\cO}$ and $\Xi_\cO$ denote their analogues on the odd side $\cO$. The goal is to use these even and odd polymer models to define a measure on subsets $I \subseteq V(G)$ that resembles the Ising model, the latter given by
\begin{align} \label{Ising-mu}
\mathbb{P}_{\mu}(\boldsymbol{I}=I)=\frac{\lam^{|I|}e^{-\beta|E(I)|}}{Z_G(\lam,\beta)}.
\end{align}
However, there is an obstacle here, as usually even and odd polymer models only choose independent sets, whereas here we also want to consider subsets of vertices that are not necessarily independent.
The sum over subsets of the neighborhood merely changes the weight of the polymers, while we want to take into account the possibility of these vertices being in $I$. Thus,  for $\cD \in \{\cO, \cE\}$, we consider the auxiliary polymer models defined as
\begin{align*}
\hat{\cP}_\cD \coloneqq \left\{(A,B): \quad A\subseteq \cD\ \text{is $2$-linked},\quad |[A]|\leq \frac{3}{4}\cdot |\cD|, \quad \text{and} \quad B \subseteq N(A)\right\}.
\end{align*}
Two polymers  are {\em compatible} if $A\cup A'$ is not $2$-linked, denoted by $(A,B)\sim(A',B')$, and {\em incompatible} otherwise, denoted by $(A,B)\nsim(A',B')$. To make clear the distinction from even and odd polymers,  we call $(A,B) \in \hat{\cP}_{\cD}$\ a \textit{decorated polymer} and its weight is defined as
$$\om(A,B)\coloneqq\frac{\lam^{|A|+|B|}}{(1+\lam)^{|N(A)|}}e^{-\beta|E(A,B)|}.$$
A set of pairwise compatible decorated polymers is called a \textit{decorated polymer configuration}. We denote by $\hat{\Om}_{\cD}$ the collection of decorated polymer configurations. The weight of a polymer configuration $\hat{\Theta} \in \hat{\Om}_{\cD}$ is given by
$$\om(\hat{\Theta}) \coloneqq \prod_{(A,B)\in \hat{\Theta}} \om(A,B).$$

Now we may use the new decorated polymer model to decompose $\om(A)$ to its summands $\om(A,B)$ and distinguish $B\subseteq N(A)$. That is,
\begin{equation}\label{eq:weight_decomp}
    \om(A)=\sum_{B\subseteq N(A)}\om(A,B).
\end{equation}
It follows that the partition functions of $\hat{\cP}_\cD$ and $\cP_\cD$ are identical:
$$\Xi_\cD \coloneqq \Xi_{\hat{\cP}_\cD}=\Xi_{\cP_\cD}=\sum_{\Theta \in \Om_\cD}\prod_{A\in\Theta}\om(A).$$ 
We now define the measure $\hat{\mu}^*$ on pairs $(I, \cD)$ as follows:
\begin{enumerate}
    \item Choose a \textit{defect side} $\cD\in\{\cO,\cE\}$ with probability proportional to $\Xi_{\cD}$.
    \item Sample a decorated polymer configuration $\hat{\Theta}$ from $\hat{\Om}_{\cD}$ according to the distribution
    $$\mathbb{P}(\boldsymbol{\hat{\Theta}}=\hat{\Theta})=\frac{\om(\hat{\Theta})}{\Xi_\cD}.$$
    \item If $\hat{\Theta}=\{(A_1,B_1),\dots, (A_s,B_s)\}$, then let $D \coloneqq \bigcup_i(A_i\cup B_i)$, and put 
    \begin{itemize}
        \item every vertex of $D$, and
        \item every vertex $v\in\overline{\cD} \setminus N(D)$ independently with probability $\frac{\lam}{1+\lam}$
    \end{itemize}
    into vertex set $I$.
\end{enumerate} 
Then $\hat{\mu}^*$ is a measure on the set of pairs $(I, \cD)$ with $\cD \in \{\cO, \cE\}$ a defect side and $I \subseteq V(G)$ a vertex subset. Given $(I, \cD)$, we may recover the decorated polymer configuration $\hat{\Theta} \in \hat{\Om}_{\cD}$ by taking the maximal $2$-linked components of $\cD$ to form the $A_i$ and the vertices in $I \cap N(A_i)$ to form $B_i$. Let $\hat{\Theta}(I)$ denote this decorated polymer configuration. In practice, we will identify the vertex set $D$ with the polymer configuration $\hat\Theta$. 

Setting $q \coloneqq \frac{\lam}{1+\lam}$, we may explicitly write out the measure $\hat{\mu}^*$ as
\begin{align} \label{eq:defhatmu}
    \mathbb{P}_{\hat{\mu}^*}\left((\boldsymbol{I}, \boldsymbol{\cD})= (I, \cD)\right)= \frac{\Xi_{\cD}}{\Xi_{\cO}+\Xi_{\cE}} \mathbbm{1}_{\hat{\Theta}(I) \in \hat{\Om}_{\cD}} \frac{\om(\hat{\Theta}(I))}{\Xi_{\cD}} q^{|I \setminus \hat{\Theta}(I)|} (1-q)^{\frac{n}{2} - |N(\bigcup_{(A, B) \in \hat{\Theta}(I)} A)|- |I\setminus \hat{\Theta}(I)|}.
\end{align}
Note that the term $\frac{\Xi_{\cD}}{\Xi_{\cO}+\Xi_{\cE}}$ is the probability to pick $\cD$ as the defect side. The term $\mathbbm{1}_{\hat{\Theta}(I) \in \hat{\Om}_{\cD}} \frac{\om(\hat{\Theta}(I))}{\Xi_{\cD}}$ corresponds to the probability that $\hat{\Theta}(I)$ is chosen in step (2), and the term $q^{|I \setminus \hat{\Theta}(I)|} (1-q)^{n/2 - |N(\bigcup_{(A, B) \in \hat{\Theta}(I)} A)|- |I\setminus \hat{\Theta}(I)|}$ is the probability of picking exactly the vertices in $I \setminus (\hat{\Theta}(I) \cup N(\hat{\Theta}(I)))$
independently with probability $q$. Now, plugging $q = \frac{\lambda}{1+\lambda}$ and
\[
\om(\hat{\Theta}(I))=\prod_{(A, B) \in \hat{\Theta}(I)} \frac{\lam^{|A|+|B|}}{(1+\lam)^{|N(A)|}} e^{-\beta |E(A, B)|}
\]
into \eqref{eq:defhatmu} yields that the measure $\hat{\mu}^*$ is given by
\begin{align}
        \mathbb{P}_{\hat{\mu}^*}((\boldsymbol{I}, \boldsymbol{\cD})= (I, \cD))&= \frac{\mathbbm{1}_{\hat{\Theta}(I) \in \hat{\Om}_{\cD}}}{\Xi_{\cO}+\Xi_{\cE}} \left( \frac{\lam}{1+\lam}\right)^{|I \setminus \hat{\Theta}(I)|} \left( \frac{1}{1+\lam} \right)^{{n}/{2} - |N(\bigcup_{(A, B) \in \hat{\Theta}(I)} A)|- |I\setminus \hat{\Theta}(I)|}\nonumber\\
    &\hspace{145pt} \cdot\prod_{(A, B) \in \hat{\Theta}(I)} \frac{\lam^{|A|+|B|}}{(1+\lam)^{|N(A)|}} e^{-\beta |E(A, B)|}\nonumber\\
    &= \frac{\mathbbm{1}_{\hat{\Theta}(I) \in \Om_{\cD}}}{\Xi_{\cO}+\Xi_{\cE}} (1+\lam)^{-n/2} \lam^{|I|} e^{-\beta |E(I)|}.\label{eq:gettinghatZ}
    \end{align}

We also define
\begin{align} \label{eq:hatweight}
    \hat{\om}^*(I, \cD) \coloneqq \mathbbm{1}_{\hat{\Theta}(I) \in \Om_\cD} \lam^{|I|} e^{- \beta |E(I)|},
\end{align}
and
\begin{align*}
\hat{\om}(I) \coloneqq \hat{\om}^*(I, \cO)+\hat{\om}^*(I, \cE).
\end{align*}
To obtain the measure $\hat{\mu}$, we sample sets $I \subseteq V(G)$ according to their weight $\hat{\om}(I)$
\begin{align}\label{mu-hat}
\mathbb{P}_{\hat{\mu}}(\mathbf{I}=I)=\frac{\hat{\om}(I)}{\hat{Z}},
\end{align}
where $\hat{Z}=\sum_{I \subseteq V(G)} \hat{\om}(I)$.
By \eqref{eq:gettinghatZ}, we conclude that the partition function $\hat{Z}$ of $\hat{\mu}$ satisfies 
$$\hat{Z}=(1+\lam)^{n/2}(\Xi_{\cO}+\Xi_{\cE}).$$

Now, we aim to compare the measure $\hat{\mu}$ (defined in \eqref{mu-hat}) with the measure $\mu$ for the Ising model (defined in \eqref{Ising-mu}). Although not exactly the same as in the Ising model, the measure $\hat{\mu}$ allows us to utilize the cluster expansion and obtain precise information on the partition functions $\Xi_\cE$ and $\Xi_\cO$. By proving the measure $\hat{\mu}$ is close in total variation distance to the measure $\mu$ of the Ising model, we are able to pass these results to the original Ising partition function $Z_G(\lam,\beta)$. This approach is standard in the hard-core model, and was initiated in the work of Jenssen and Perkins \cite{jenssen2020independent}. As we mentioned, the reason why the decorated polymer model $\hat{\cP}_\cD$ is defined is to derive the measure $\hat{\mu}$. However, in practice, we will be working mainly with the (odd or even) polymer models $\cP_\cD$ for  $\cD \in \{\cO, \cE\}$, as the grouping of weights according to \eqref{eq:weight_decomp} is more convenient for bounding their sum over sets of (odd or even) polymer configurations.

Fix $\cD \in \{\cO, \cE\}$ and define the {\em size} of a cluster $\Gamma\in \cC_\cD$ to be $\norm{\Gamma} \coloneqq \sum_{A \in \Gamma} |A|$. For each $k \in \mathbb{N}$, define 
$\cC_{\cD,k}\coloneqq \{\Gamma\in \cC_\cD: \norm{\Gamma} =k\}$, $\cC_{\cD, <k}\coloneqq \bigcup_{j<k}\cC_{\cD,j}$, and  $\cC_{\cD, \geq k} \coloneqq \bigcup_{j\geq k}\cC_{\cD,j}$. Also we define  
\[
L_{\cD, k} \coloneqq \sum_{\Gamma \in \cC_{\cD,k}} \omega(\Gamma) \qquad \text{and}\qquad L_{\cD, <k} \coloneqq \sum_{\Gamma \in \cC_{\cD,<k}} \omega(\Gamma) 
\]
and define $L_{\cD, \geq k}$ analogously. The convergence of the cluster expansion will allow us to bound the difference $L_{\cD, \geq k}$ between the cluster expansion $$\log \Xi_\cD=\sum_{\Gam\in\cC_\cD}\om(\Gam)$$ and its truncation $L_{\cD, <k}$. The following lemma is proved in \Cref{sec:kotecky}.
\begin{restatable}{lemma}{lKotecky} \label{l:Kotecky}
    For any $\lambda_0 > 0$, there exists $C_0 > 0$ such that if $\lam \le \lam_0$ and $\lam p \geq \frac{C_0 \log^{3/2} d}{d^{\lambexp}}$, then for every $k \in \mathbb{N}$, 
    \[
    |L_{\cD, \geq k}| \leq nd^{(C_5+7)k-C_5-1} \left(\frac{1+\lam}{1+\lam e^{-\beta}}\right)^{-kd + C_1k^2}.
    \]
\end{restatable}

For a set $I \subseteq V(G)$, we denote by $\cM(I)$ the minority side of $I$, i.e., one of the sides $\cM(I) \in \{\cO, \cE\}$ such that $|\cM(I) \cap I| \leq |I\setminus \cM(I)|$.\footnote{In the case of equality, we may choose either side to be the minority side.} The key property of $\hat{\mu}$ is that it is close to $\mu$ in terms of total variation distance, as formalized in the next lemma. The proof, along with other properties of $\hat{\mu}$, is given in \Cref{sec:propertiesmu}.

\begin{restatable}{lemma}{lemapproximateZ} \label{lem:approximateZ}
For any $\lambda_0 > 0$, there exists $C_0 > 0$ such that if $\lam \le \lam_0$ and $\lam p \geq \frac{C_0 \log^{3/2} d}{d^{\lambexp}}$, then
    \[
    \left| \log Z_G(\lam, \beta) - \log\left((1+ \lam)^{n/2} (\Xi_\cO + \Xi_\cE)\right) \right| = O\left( \exp\left(-\frac{n}{d^{C_5+4}}\right)\right),
    \]
and this implies that
\[
\norm{\hat{\mu} - \mu}_{TV} = O\left( \exp\left(-\frac{n}{d^{C_5+4}}\right)\right).
\]
\end{restatable}
Having Lemmas \ref{l:Kotecky} and \ref{lem:approximateZ}, we now proceed to the proof of \Cref{thm:expansion}. For this, we define the parameters
\begin{align*}
    \alpha \coloneqq \lambda(1 - e^{-\beta}) = \lambda p, \hspace{40pt} \ta \coloneqq \left( 1 - \frac{\alpha}{1+\lambda} \right)^{-1} = \frac{1+\lam}{1+\lam e^{-\beta}},
\end{align*}
and we note that $\ta \le 1+\lambda$ and $\log \ta \ge \frac{\alpha}{1+\lambda}$. These parameters will also appear in later sections.

\begin{proof}[Proof of \Cref{thm:expansion}]
Recall that  for $p=1-e^{-\beta}$, we have
    \[
    \Ex{Z_{G_p}(\lam)} = Z_G(\lam, \beta).
    \]
 Thus, let us compute an expansion of $Z_G=Z_G(\lam, \beta)$.
    From the cluster expansion of $\Xi_{\cD}$ and the definition of $L_{\cD, \geq k+2}$ we have that
    \[
    \Xi_{\cD}=\exp\left( L_{\cD, \leq k} + L_{\cD, k+1} + L_{\cD, \geq k+2} \right).
    \]
    By \Cref{lem:approximateZ}, we may approximate $Z_G$ using $\Xi_{\cO}, \Xi_{\cE}$ to get that
    \[
    Z_G=(1+\lambda)^{n/2} \sum_{\cD \in \{\cO, \cE\}} \exp\left( L_{\cD, \leq k} + L_{\cD, k+1} + L_{\cD, \geq k+2} + \varepsilon' \right),
    \]
    where
    \[
    \varepsilon' \coloneqq \log Z_G - \log((1+\lambda)^{n/2} (\Xi_{\cO} + \Xi_{\cE})) = O\left( \exp\left(- \frac{n}{d^{C_5+4}} \right)\right).
    \]

To start, we compute the first term in this expansion.
Since $G$ is regular bipartite, there are $n/2$ polymers $A$ of size $1$ in $\cD$, each with weight 
\begin{align*}
    \om(A) &= \sum_{B \subseteq N(A)} \frac{\lam^{1+|B|}}{(1+\lam)^d} e^{-\beta |E(A, B)|} = \frac{\lam}{(1+\lam)^d} \left(1+\lam e^{-\beta}\right)^d = \lam \left(1-\frac{\lam p}{1+\lam}\right)^d.
\end{align*}
Since the Ursell function of a single vertex graph is $1$, we get that
\[
L_{\cD, 1} = \sum_{\Gam \in \cC_{\cD,1}} \om(\Gam) = \frac{n \lam}{2} \left(1-\frac{\lam p}{1+\lam}\right)^d. 
\]

Next, we claim that for any fixed $k\in \mathbb N$, we have
    \[
    |L_{\cD, k}|=O\left(n d^{2(k-1)} \lam^k \tilde{\alpha}^{-kd} \right).
    \]
    To prove this, let $\Gam \in \cC_{\cD, k}$ be a cluster of size $k$, that is, $\norm{\Gamma} \coloneqq \sum_{A \in \Gamma} |A| =k$.
    The vertex set of the cluster $\Gamma$, i.e., $V(\Gam) = \bigcup_{A \in \Gam} A$, is a $2$-linked set of size at most $k$. There are at most \[
    \sum_{\ell=1}^k n \cdot (ed^2)^{\ell-1} \leq k \cdot n \cdot (ed^2)^{k-1}\] possible choices for $V(\Gam)$, because the number of $2$-linked sets of size $\ell$ containing a given vertex $v$ is at most $(ed^2)^{\ell-1}$ by \Cref{l:counting2linked}, and there are $n$ choices for the vertex $v$.
    On the other hand, for every subset $X \subseteq V(G)$ of size at most $k$, there are only a constant number of clusters $\Gam$ such that $X=V(\Gam)$.
    Thus, in total there are $O(nd^{2(k-1)})$ clusters of size $k$.
    Now we bound the weight of each such cluster $\Gam$ using \Cref{l:boundingweight}. This gives us that
    \[
    |\om(\Gam)| = |\phi (\Gam)| \prod_{A \in \Gam} \om(A) \leq |\phi(\Gam)| \prod_{A \in \Gam} \lam^{|A|} \tilde{\alpha}^{-|N(A)|},
    \]
    where $\tilde{\alpha}$ was defined above.
    By \hyperref[def:property-i]{Property I}, we have $|N(A)| \geq |A|(d-C_1|A|)$, which implies that $\sum_{A\in \Gam}|N(A)|\geq k(d-C_1k)$. Hence,
    \begin{align*}
        |\om(\Gam)| = O\left(\lam^k \tilde{\alpha}^{-k(d-C_1k)}\right),
    \end{align*}
    as the Ursell function of a cluster of fixed size can be bounded from above by a constant.
    Since $C_1$ is constant, $k$ is fixed, and $\ta \le 1+\lambda_0$, we have that $\ta^{C_1k^2}$ is also bounded from above by a constant. Combining this upper bound on $|\omega(\Gamma)|$ with the upper bound on the total number of clusters yields the claimed upper bound on $|L_{\cD, k}|$. 
    
    Our next claim is that the error when approximating $\Xi_{\cD}$ by $L_{\cD, <k+2}$ is much smaller than the bound on $|L_{\cD,k+1}|$ given by the previous claim:
    \[
    L_{\cD, \geq k+2} = o\left(n d^{2k} \lam^{k+1} \tilde{\alpha}^{-(k+1)d}\right).
    \]
    To prove this, from \Cref{l:Kotecky}, we may bound $L_{\cD, \geq k+2}$ as
    \begin{align*}
        |L_{\cD, \geq k+2}| &\le n d^{(C_5+7)(k+2)-C_5-1} \tilde{\alpha}^{-(k+2)d+C_1(k+2)^2} \\
        &= n d^{2k} \lam^{k+1} \tilde{\alpha}^{-(k+1)d} \cdot \left( d^{(C_5+5)k+C_5+13} \lambda^{-(k+1)} \tilde{\alpha}^{-d+C_1(k+2)^2} \right).
    \end{align*}
    We observe that the last parenthetical term is indeed $o(1)$: from the assumptions $\lambda \le \lambda_0$ and $\alpha = \lambda p \ge \frac{C_0 \log^{3/2} d}{d}$, we have that
    \begin{align*}
        \ta^{-d} \le \exp\left(-d \cdot \frac{\alpha}{1+\lambda} \right) \le \exp \left( -\frac{C_0 \log^{3/2} d}{1+\lambda_0} \right),
    \end{align*}
    which shrinks faster than any polynomial in $d$. On the other hand, $\lambda^{-(k+1)} \le d^{k+2}$ is at most polynomial in $d$, and like before $\ta^{C_1(k+2)^2}$ is bounded from above by a constant. This proves the claim.
    
    We stress again that in general $\Xi_{\cO}$ and $\Xi_{\cE}$ are different. But we have shown above that $L_{\cD, 1}$ is the same for both models, and the bounds we obtained for $L_{\cD, k}$ when $k>1$ hold in both cases as well. Setting $\varepsilon_{\cD, k} \coloneqq L_{\cD, k+1} + L_{\cD, \geq k+2} + \varepsilon'$, we get that
    \[
    Z_G = \sum_{\cD \in \{\cO, \cE\}} (1+\lam)^{n/2} \exp\left(\sum_{j=1}^k L_{\cD, j}+ \varepsilon_{\cD, k}\right).
    \]
    Using $n = \om(d^{C_5+5})$, we moreover get that
    \[
    |\varepsilon_{\cD, k}| = O\left(nd^{2k} \lam^{k+1}{ \tilde{\alpha}^{-d(k+1)}}\right).
    \]
    This finishes the proof.
\end{proof}

\section{Graph containers}\label{sec:containers}

In this section we will state our main container lemma (\Cref{lem:container_main}) that we promised in the Introduction, but we will prove it in slightly more general settings than graphs satisfying \hyperref[def:property-i]{Property I}. 

Throughout this section we let $G$ be a $d$-regular $n$-vertex bipartite graph with bipartition $(\cO, \cE)$ (which does not necessarily satisfy \hyperref[def:property-i]{Property I}) and fix $\cD \in \{\cO, \cE\}$. A central concept in graph containers is the family of 2-linked sets whose closure and neighborhood have fixed sizes. Given $a, b \in \mathbb{N}$ we define
\begin{align*}
    \cG_\cD(a,b)\coloneqq \Big\{A\subseteq\cD:\ A\ \text{is $2$-linked}, \quad |[A]|=a, \quad \text{and}\quad |N(A)|=b\Big\}.
\end{align*}
Note that since $G$ is regular and bipartite, $G$ contains a perfect matching (e.g., due to Hall's condition) and therefore
$\cG_\cD(a,b)$ is empty whenever $b < a$. So we assume throughout this section that $b \ge a$. It is important to note that here we do not impose a \lq global' upper bound on $|[A]|$ for any set $A\in \cG_\cD(a,b)$, in contrast to our polymer model $\cP_\cD$ in which we require $|[A]| \leq  \frac{3}{4} \cdot |\cD|$ for every $A\in \cP_\cD$.

For each set $A \in \cG_\cD(a,b)$, we define its weight as
\begin{align}
    \om(A) \coloneqq \sum_{B \subseteq N(A)} \frac{\lam^{|A|+|B|}}{(1+\lam)^{|N(A)|}} e^{-\beta |E(A,B)|}.\label{weight:AinG}
\end{align}
In the results below, we will always implicitly assume that $d$ is sufficiently large compared to the fixed parameters. The fugacity $\lambda > 0$ and inverse temperature $\beta > 0$ may depend on $d$, and we again set
\begin{align*}
    \alpha \coloneqq \lambda(1 - e^{-\beta}) \qquad \text{and} \qquad \lambexp = \min \left\{\frac{1}{2}, 1 - \frac{C_5}{2}\right\}.
\end{align*}

The following definition is central for graph containers.
\begin{definition}\label{def:psi_approx}
Let $1 \le \psi \le d - 1$. A $\psi$-approximation $(F,H)$ is a pair of vertex sets $F \subseteq \overline{\cD}, H \subseteq \cD$ such that for some set $A \subseteq \cD$, 
    \begin{equation}\label{eqn:cont_1}
        F\subseteq N(A)  \ \ \text{and} \ \ H\supseteq [A],
    \end{equation}
    \begin{equation}\label{eqn:cont_2}
        d_F(u)\geq d(u)-\psi, \quad \forall u\in H,
    \end{equation}
    \begin{equation}\label{eqn:cont_3}
        d_{\cO\setminus H}(v)\geq d(v)-\psi, \quad \forall v\in \cE\setminus F.
    \end{equation}
    In this case, we say that $(F, H)$ approximates $A$, and we denote this relation by 
    \[A\approx (F,H).\]
\end{definition}

Our main container lemma, whose proof can be found at the end of this section, can now be stated.

\begin{lemma} \label{lem:container_main}
For any fixed constants $C_2, C_4, C_5, \lam_0 > 0$ with $C_5 < 2$, there exist constants $C_0, C > 0$ such that the following holds. Suppose that $\lam \le \lam_0$ and $\alpha \coloneqq \lam (1 - e^{-\beta}) \ge \frac{C_0 \log^{3/2} d}{d^{\lambexp}}$, that for every set $X \subseteq N(y)$ for some $y \in \overline{\cD}$ with $|X| > d/2$ it holds that $|N(X)| \ge \frac{d}{C_2}|X|$, and that $a, b \in \mathbb{N}$ satisfy $b \ge \left( 1 + \frac{C_4}{d^{C_5}} \right)a$. Then
\begin{equation}\label{eq:container}
    \sum_{A \in \cG_\cD(a,b)} \om(A) \le |\cD| \exp \left( -\frac{C(b - a)\alpha^2}{\log d} \right).
\end{equation}
\end{lemma}

In \cite{kronenberg2022independent}, Kronenberg and Spinka proved the same result for the hypercube in the range $\alpha= \lam (1-e^{-\beta})=\tilde{\Om}(1/d^{1/3})$, which can be viewed as a generalization to the Ising model of the corresponding bound by Galvin in \cite{galvin2011threshold}. Their proof uses two main tools. On one hand are graph containers, which provide control on $|\cG_\cD(a,b)|$ and are necessary ingredients present in all previous works. On the other hand, having to deal with the new weights in the summation arising from the Ising model, they required a novel use of entropy-derived lemmas by Peled and Spinka \cite[Lemma 7.3]{peled2020long}.

Specializing our result to the hard-core model, we recover the result of Jenssen, Malekshahian, and Park \cite[Lemma 1.2]{jenssen2024refined} for regular bipartite graphs, with a slight improvement on the lower bound on $\lam$ by a factor of $\log^{1/2}d$ (see \Cref{cor:hard-core-corollary-1}). Furthermore, we derive a similar bound for a wider class of regular graphs, allowing the exponent $C_5$ in the isoperimetric statement to be greater than $1$ (see \Cref{cor:hard-core-corollary-2}).

\begin{corollary} \label{cor:hard-core-corollary-1}
    Under the conditions of Lemma \ref{lem:container_main} and $C_5\leq 1$, for any $\lam \ge \frac{C_0\log^{3/2}d}{d^{1/2}}$ (not necessarily bounded) we have that
    \begin{align*}
    \sum_{A \in \cG_\cD(a,b)} \lam^{|A|} \le |\cD|(1+\lam)^b \exp \left( -\frac{C(b - a)\log^2d}{d} \right).
    \end{align*}
\end{corollary}

\begin{corollary} \label{cor:hard-core-corollary-2}
    Under the conditions of Lemma \ref{lem:container_main} and $1 \le C_5 < 2$, for any $\lam \ge \frac{C_0\log^{3/2}d}{d^{1-C_5/2}}$ (not necessarily bounded) we have that
    \begin{align*}
    \sum_{A \in \cG_\cD(a,b)} \lam^{|A|} \le |\cD|(1+\lam)^b \exp \left( -\frac{C(b - a)\log^2d}{d^{2-C_5}} \right).
    \end{align*}
\end{corollary}

We remark that in \cite{jenssen2024refined}, the authors also noted their container lemma can be extended to hold for slightly worse expansion. Rephrasing this to our notation their proof works for $C_5 \in (1, 3/2)$, under the revised assumption that $\lambda \ge \frac{C_0 \log^2 d}{d^{3/2 - C_5}}$. In addition, Corollaries \ref{cor:hard-core-corollary-1} and \ref{cor:hard-core-corollary-2} are not exactly immediate from \Cref{lem:container_main}, as taking the limit as $\beta\rightarrow\infty$ only yields the statements under the assumption that $\lam$ is bounded. However, if we work on the hard-core model from the start and revisit the steps of the proof of \Cref{lem:container_main}, this boundedness assumption can be removed. This is outlined in \Cref{sec:container-hard-core}.

Returning to the container method, it typically consists of two parts to derive inequalities of the form of \eqref{eq:container}. Instead of working with the sum directly, we first use a family of sets, called  \textit{containers}, in which every set $A\in \cG_\cD(a,b)$ has a $\psi$-approximation, and then the sum of weights of elements with a given approximation is estimated. Multiplying by the size of the family of containers yields a bound in the spirit of our main container lemma (Lemma \ref{lem:container_main}). The strength of the result depends on the optimization of the two parts of the computation. In \cite{jenssen2024refined}, the bound on the weighted sum was improved by considering a new container family with desirable properties. In this paper, we use the same containers as Galvin \cite{galvin2011threshold} and instead improve on the computation via entropy. The following is the container lemma of Galvin.

\begin{lemma}[\cite{galvin2011threshold}, Lemmas 5.1 and 5.2] \label{lem:galvin_approx_family}
For any fixed constants $C_2, C_4, C_5 > 0$, there exists a constant $C' > 0$ such that the following holds. Suppose that for every set $X \subseteq N(y)$ for some $y \in \overline{\cD}$ with $|X| > d/2$ it holds that $|N(X)| \ge \frac{d}{C_2}|X|$, that $a, b \in \mathbb{N}$ satisfy $b \ge \left( 1 + \frac{C_4}{d^{C_5}} \right)a$, and that $1 \le \psi \le d-1$. Then there exists a family $\cA \subseteq 2^{\cD} \times 2^{\overline{\cD}}$ of pairs $(F, H)$ with
\begin{equation}\label{eq:approx_family}
    |\cA| \le |\cD| \exp \left(  \frac{C' b \log^2 d}{d^2} + \frac{C' (b - a)\log^2 d}{d} + \frac{C' (b - a)\log d}{\psi} \right),
\end{equation}
such that every $A\in \cG_{\cD}(a,b)$ has a $\psi$-approximation in $\cA$.
\end{lemma}

We will combine Lemma \ref{lem:galvin_approx_family} with the following upper bound on the total weight of sets from $\cG_{\cD}(a,b)$ with a given approximation, improving Kronenberg and Spinka's result \cite[Lemma 4.8]{kronenberg2022independent}. Its proof is deferred to \Cref{sec:bound_given_approx}, after we introduce the necessary entropy tools in \Cref{sec:entropy_tools}.

\begin{lemma} \label{lemma:total-weight-given-approximation}
For any fixed constants $C_4, C_5, \lam_0 > 0$, there exist constants $C_0, C'' > 0$ such that the following holds. Suppose that $\lam \le \lam_0$ and $\lam (1 - e^{-\beta})^2 \ge \frac{C_0 \log d}{d}$, that $a,b \in \mathbb{N}$ with $b \ge \left(1 + \frac{C_4}{d^{C_5}}\right)a$, and that $1 \le \psi \le d/2$. Then for any $\psi$-approximation $(F,H)$, we have
\begin{align*}
    \sum_{A \in \cG_\cD(a,b) :\ A \approx (F,H)} \om(A) \le \exp\left( -\frac{C''(b-a)\alpha^2}{\log d} \right).
\end{align*}
\end{lemma}

Now we verify that Lemma \ref{lem:container_main} follows from Lemmas \ref{lem:galvin_approx_family} and \ref{lemma:total-weight-given-approximation}. Note that if $\lam \le \lam_0$ and $\alpha= \lam(1 - e^{-\beta}) \ge \frac{C_0 \log^{3/2} d}{d^{1/2}}$, then we have $\lam (1 - e^{-\beta})^2 \ge \frac{C_0^2 \log^3 d}{\lam_0 d}$, so that Lemma \ref{lemma:total-weight-given-approximation} applies.

\begin{proof}[Proof of Lemma \ref{lem:container_main}]
Choose $\psi = d/2$. By Lemmas \ref{lem:galvin_approx_family} and \ref{lemma:total-weight-given-approximation}, we have
\begin{align*}
    \sum_{A \in \cG_{\cD}(a,b)} \om(A) \le |\cD| \exp \left( \frac{C'b\log^2 d}{d^2} + \frac{C'(b - a)\log^2 d}{d} - \frac{C''(b-a)\alpha^2}{\log d} \right).
\end{align*}
The desired inequality will follow once we show that $\frac{C''(b-a)\alpha^2}{\log d}$ is larger than both $\frac{C'b\log^2 d}{d^2}$ and $\frac{C'(b - a)\log^2 d}{d}$. The latter inequality holds because $\alpha \ge \frac{C_0 \log^{3/2} d}{d^{1/2}}$. The former inequality rearranges to $\alpha > \frac{c \log^{3/2} d}{d} \left( \frac{b}{b - a} \right)^{1/2}$, and this holds because $\frac{b}{b - a} \le 1 + \frac{1}{C_4}d^{C_5}$ and $\alpha \ge \frac{C_0 \log^{3/2} d}{d^{1 - C_5/2}}$.
\end{proof}

\section{Bounds from entropy}\label{sec:entropy_tools}

The goal for this section is to give a brief overview of the tools/lemmas we will use to prove Lemma \ref{lemma:total-weight-given-approximation}, which bounds the total weights of sets from $\cG_{\cD}(a,b)$ with a given approximation. These tools are entropy bounds due to Peled and Spinka \cite[Lemma 7.3]{peled2020long}. The original bounds in \cite{peled2020long} are quite general and apply to a broad class of discrete spin systems on regular bipartite graphs, but we will only require special cases, as outlined by Kronenberg and Spinka \cite[Lemmas 4.11, 4.12, 4.13]{kronenberg2022independent}. We dedicate \Cref{sec:appendix} for more background on this topic.

Assume the general setup as in \Cref{sec:containers}.  
We say that a vertex subset $T \subseteq V(G)$ is \textit{boundary-odd} if $N(T \cap \cE) \subseteq T$. For a vertex subset $I \subseteq V(G)$, we define its weight as
\begin{align*}
    \tilde{\om}(I) \coloneqq \lam^{|I|} e^{-\beta |E(I)|}.
\end{align*}
This is not the same as the weight $\om(A)$ that we defined in \eqref{weight:AinG} for $A \in \cG_{\cD}(a,b)$, but the two weights are related, and this connection will be used in the proofs. For a family $\cF$ of vertex subsets $I \subseteq V(G)$, we set $\tilde{\om}(\cF) \coloneqq \sum_{I \in \cF} \tilde{\om}(I)$. For a collection $\Psi$ of subsets $\psi \subseteq [d] \coloneqq \{1, \ldots, d\}$, we define
\begin{align*}
    Z(\Psi) \coloneqq \sum_{\psi \in \Psi} \lam^{|\psi|} \left( 1 + \lam e^{-\beta |\psi|} \right)^d,
\end{align*}
which is a type of partition function on the collection $\Psi$.

The following lemmas from \cite[Lemma 4.14]{kronenberg2022independent} provide upper bounds on $\tilde{\om}(\cF)$ in terms of products of the functions $Z(\Psi)$ for associated collections $\Psi$ of subsets of neighborhoods. More details on how these statements are derived from \cite{peled2020long} can be found in \Cref{sec:appendix}. Similar to entropy conventions, we use the convention that $(1/p)^p$ equals $1$ whenever $p = 0$.

\begin{lemma} \label{lemma:entropy-1}
Let $T \subseteq V(G)$ be boundary-odd and let $\cF$ be a family of subsets of $T$. Then
\begin{align*}
    \tilde{\om}(\cF) \le \prod_{v \in T \cap \cO} Z(\Psi_v)^{\frac{1}{d}},
\end{align*}
where $\Psi_v = \{I \cap N(v) : I \in \cF\}$.
\end{lemma}

\begin{lemma} \label{lemma:entropy-2}
Let $T \subseteq V(G)$ be boundary-odd and let $\cF$ be a family of subsets of $T$ such that every $I \in \cF$ contains no isolated vertex in $I \cap \cO$. Then
\begin{align*}
    \tilde{\om}(\cF) \le \prod_{v \in T \cap \cO} Z(\Psi_v)^{\frac{p_v}{d}} \left( \frac{1}{p_v} \right)^{\frac{p_v}{d}} \left( \frac{1}{1 - p_v} \right)^{\frac{1-p_v}{d}},
\end{align*}
where $\Psi_v \coloneqq \{I \cap N(v) : I \in \cF, |I \cap N(v)| > 0\}$ and $p_v \coloneqq \mathbb{P}(|I \cap N(v)| > 0)$, when $I$ is randomly chosen from $\cF$ according to $\tilde{\om}$.
\end{lemma}

We will use Lemmas \ref{lemma:entropy-1} and \ref{lemma:entropy-2} to prove Lemma \ref{lemma:total-weight-given-approximation} in \Cref{sec:bound_given_approx}. We will also need the following result, in order to bound the total weight of {\em non-configurations} in \Cref{sec:nonpolymer}.
\begin{lemma} \label{lemma:entropy-3}
Let $T \subseteq V(G)$ be boundary-odd, let $\cF$ be a family of subsets of $T$ and let $s > 0$. Then
\begin{align*}
    \tilde{\om}(\cF) \le \prod_{v \in T \cap \cO} Z(\Psi_v)^{\frac{p_v}{d}} Z(\Psi_v')^{\frac{p_v'}{d}} (1 + \lam)^{1 - p_v - p_v'} \left( \frac{1}{p_v} \right)^{\frac{p_v}{d}} \left( \frac{1}{p_v'} \right)^{\frac{p_v'}{d}} \left( \frac{1}{1 - p_v - p_v'} \right)^{\frac{1-p_v-p_v'}{d}},
\end{align*}
where $\Psi_v \coloneqq \{I \cap N(v) : I \in \cF, 1 \le |I \cap N(v)| \le s\}$, $\Psi_v' \coloneqq \{I \cap N(v) : I \in \cF, |I \cap N(v)| > s\}$, $p_v \coloneqq \mathbb{P}(1 \le |I \cap N(v)| \le s)$, and $p_v' = \mathbb{P}(|I \cap N(v)| > s)$, when $I$ is randomly chosen from $\cF$ according to $\tilde{\om}$.
\end{lemma}

To effectively use the upper bounds on $\tilde{\om}(\cF)$ given by the above lemmas, we need good upper bounds on $Z(\Psi)$ for general families $\Psi \subseteq 2^{[d]} \setminus \{\emptyset\}$. Define
\begin{align*}
    \ell_{\Psi} \coloneqq \left| [d] \setminus \bigcup \Psi \right| = \Big|\{ i \in [d] : i \notin \psi \text{ for all } \psi \in \Psi \}\Big|
\end{align*}
and
\begin{align*}
    \overline{\alpha} \coloneqq -\log \left( 1 - \frac{\alpha}{1 + \lam} \right) = \log \left( \frac{1 + \lam}{1 + \lam e^{-\beta}} \right).
\end{align*}
Note that when $0 < \lam \le \lam_0$, we have $\frac{1}{1+\lam_0} \alpha \le \overline{\alpha} \le \alpha$. The following lemma provides us an effective upper bound on $Z(\Psi)$.

\begin{lemma} \label{lemma:bound-on-Z}
Fix $C, \lam_0 > 0$, and take $C_0$ to be sufficiently large. Suppose that $\lam \le \lam_0$ and $\lam(1 - e^{-\beta})^2 \ge \frac{C_0 \log d}{d}$. Then for any family $\Psi$ of subsets of $[d]$ with $\emptyset \notin \Psi$, we have
\begin{align*}
    Z(\Psi) \le (1 + \lam)^d \exp \left( -\frac{1}{2} \overline{\alpha} \ell_{\Psi} + d^{-C} \right).
\end{align*}
\end{lemma}

Lemma \ref{lemma:bound-on-Z} will follow from Lemma \ref{lemma:more-general-bound-on-Z} below by taking $\ell=\ell_{\Psi}/2$. The latter more general lemma will also be used when we bound the total weight of non-configurations  in \Cref{sec:nonpolymer}.

The following inequalities will be used in the proof of Lemma \ref{lemma:more-general-bound-on-Z}. Suppose that $\lam \le \lam_0$ and $\lam (1 - e^{-\beta})^2 \ge \frac{C_0 \log d}{d}$. Then for some constants $C, C' > 0$ (with $C \rightarrow \infty$ as $C_0 \rightarrow \infty$) we have
\begin{align} \label{convenient-inequalities}
    \frac{\lam}{1+\lam} \ge \frac{C'\overline{\alpha}}{32}+\frac{4 \log (\lam d^{C+1})}{\beta d}, \hspace{20pt} \overline{\alpha} \ge \frac{4C \log d}{d} + \frac{20 \log (2+\lam) \log (1 + \lam d)}{\beta d}.
\end{align}
The $\beta$ factors in the denominators come from the inequality $\beta \ge 1 - e^{-\beta}$. Indeed, the last fraction is the main reason that we require the stronger condition $\lam(1 - e^{-\beta})^2 \ge \frac{C_0 \log d}{d}$.

\begin{lemma} \label{lemma:more-general-bound-on-Z}
Suppose that the inequalities in (\ref{convenient-inequalities}) hold with $C > 0$ and $0 < C' < 1/4$.
Let $0 \le \ell \le \min \{ \ell_{\Psi}, d/2\}$ and $s = \frac{d - \ell}{2} \frac{\lam}{1+\lam}$. Then for any family $\Psi$ of subsets of $[d]$, we have
\begin{align*}
    Z(\Psi^{\le s}) \le (1 + \lam)^d \exp\left( -\overline{\alpha} \ell - C' \overline{\alpha} d\right)\qquad \textit{and} \qquad Z(\Psi^{>s}) \le (1 + \lam)^d \exp\left( -\overline{\alpha} \ell + d^{-C}\right),
\end{align*}
where $\Psi^{\le s} \coloneqq \{ \psi \in \Psi : 1 \le |\psi| \le s\}$ and $\Psi^{>s} \coloneqq \{ \psi \in \Psi : |\psi| > s\}$. 
\end{lemma}

\begin{proof}
First, an upper bound for $Z(\Psi^{> s})$ is given by
\begin{align*}
    \sum_{\psi \in \Psi : |\psi| > s} \lam^{|\psi|} (1 + \lam e^{-\beta |\psi|})^d &\le (1 + \lam e^{-\beta s})^d \sum_{\psi \in \Psi} \lam^{|\psi|} \\
    &\le (1 + \lam e^{-\beta s})^d (1 + \lam)^{d - \ell_{\Psi}} 
    \le (1 + \lam)^d e^{-\overline{\alpha} \ell} (1 + \lam e^{-\beta s})^d,
\end{align*}
where the last inequality uses that $\ell \le \ell_{\Psi}$ and $\overline{\alpha} \le \log(1+\lam)$. 
Then from $s \ge \frac{d}{4} \frac{\lam}{1+\lam} \ge \frac{\log(\lam d^{C+1})}{\beta}$, we have that $(1 + \lam e^{-\beta s})^d \le \exp \left( \lam d e^{-\beta s}\right) \le \exp(d^{-C})$, which gives the stated bound on $Z(\Psi^{>s})$.

Next, for $Z(\Psi^{\le s})$, we define $s'= \frac{5 \log(2+\lam)}{\beta}$ and split into two cases. For $s' < |\psi| \le s$, we have
\begin{align*}
    \sum_{\psi \in \Psi : s' < |\psi| \le s} \lam^{|\psi|} \left( 1 + \lam e^{-\beta |\psi|} \right)^d &\le (1 + \lam e^{-\beta s'})^d \sum_{\psi \in \Psi : |\psi| \le s} \lam^{|\psi|} \\ 
    &\le (1 + \lam e^{-\beta s'})^d (1 + \lam)^{d-\ell} \sum_{\psi \in 2^{[d-\ell]} : |\psi| \le s} \left( \frac{\lam}{1+\lam} \right)^{|\psi|} \\
    &\le (1 + \lam e^{-\beta s'})^d (1 + \lam)^{d-\ell} e^{-s/4} \\
    &\le (1 + \lam)^d e^{-\overline{\alpha} \ell} e^{-s/4} (1 + \lam e^{-\beta s'})^d,
\end{align*} 
where the second inequality uses that $s \le d - \ell$ and $\ell \le \ell_{\Psi}$, the third inequality uses a Chernoff bound (\Cref{l:Chernoff}), and the last inequality uses that $\overline{\alpha} \le \log(1+\lam)$. On the other hand, for $1 \le |\psi| \le s'$, we have
\begin{align*}
    \sum_{\psi \in \Psi : 1 \le |\psi| \le s'} \lam^{|\psi|} \left( 1 + \lam e^{-\beta |\psi|} \right)^d &\le (1+\lam e^{-\beta})^d \sum_{\psi \in \Psi : 1 \le |\psi| \le s'} \lam^{|\psi|} \\
    &\le (1+\lam)^d \left( 1 - \frac{\lam (1 - e^{-\beta})}{1+\lam} \right)^d (1 + \lam d)^{s'} \\
    &\le (1+\lam)^d e^{-\overline{\alpha}\ell} e^{-\frac{1}{2}\overline{\alpha}d} (1 + \lam d)^{s'},
\end{align*}
where the third inequality uses the definition of $\overline{\alpha}$ and $d \ge \ell + \frac{d}{2}$. Therefore,
\begin{align*}
    Z(\Psi^{\le s}) \le (1 + \lam)^d e^{-\overline{\alpha} \ell} \left[ e^{-s/4}(1 + \lam e^{-\beta s'})^d + e^{-\frac{1}{2} \overline{\alpha} d} (1 + \lam d)^{s'} \right].
\end{align*}
We now show that both terms in the brackets of the right-hand side are at most $e^{-C' \overline{\alpha} d}$, and then the stated upper bound on $Z(\Psi^{\le s})$ will follow. For the first term, from $s \ge \frac{d}{4}\frac{\lam}{1+\lam}$, the definition of $s'$, and $\frac{\lam}{1+\lam} \ge \frac{C'\overline{\alpha}}{32}$, we have
\begin{align*}
    e^{-s/4} (1 + \lam e^{-\beta s'})^d &\le \exp \left( -\frac{d}{16} \frac{\lam}{1+\lam} \right) \cdot \left(1 + \frac{\lam}{(2 + \lam)^{5}}\right)^d \\ 
    &\le \exp \left( -\frac{d}{16} \frac{\lam}{1+\lam} \right) \cdot \exp\left(d \cdot \frac{\lam}{32(1 + \lam)} \right) \\
    &\le \exp\left( -C'\overline{\alpha}d \right).
\end{align*}

For the second term, from the definition of $s'$ and the lower bound on $\overline{\alpha}$ in \eqref{convenient-inequalities}, we also have
\begin{align*} 
    e^{-\frac{1}{2} \overline{\alpha} d} (1 + \lam d)^{s'} = \exp \left( -\frac{1}{2} \overline{\alpha} d + \frac{5 \log(2+\lam) \log(1 + \lam d)}{\beta} \right) \le e^{-\overline{\alpha}d/4} \le e^{-C'\overline{\alpha}d}.
\end{align*}
This finishes the proof.
\end{proof}

\section{Bounding weights of polymers with a given approximation}\label{sec:bound_given_approx}

In this section, we prove Lemma \ref{lemma:total-weight-given-approximation}, assuming the general setup as in \Cref{sec:containers}.   To do that, we prove two auxiliary lemmas giving different upper bounds on the total weights of polymers (in fact, more generally  sets in $\cG_\cD(a,b)$) with a given approximation, each effective in different regimes. Recall from the previous section that we defined
\begin{align*}
    \overline{\alpha} \coloneqq -\log \left( 1 - \frac{\alpha}{1 + \lam} \right) = \log \left( \frac{1 + \lam}{1 + \lam e^{-\beta}} \right),
\end{align*}
and we observed that $\frac{1}{1+\lam_0} \alpha \le \overline{\alpha} \le \alpha$ whenever $0 < \lam \le \lam_0$.

The first lemma and its proof are essentially the same as Kronenberg and Spinka's \cite[Lemma 4.9]{kronenberg2022independent} with more explicit constants, and for completeness we include the relatively short proof below.

\begin{lemma} \label{lemma:b-f-small}
For any fixed constants $C, \lam_0 > 0$, there exists a constant $C_0 > 0$ such that the following holds. Suppose that $\lam \le \lam_0$ and $\lam(1 - e^{-\beta})^2 \ge \frac{C_0 \log d}{d}$, that $a, b \in \mathbb{N}$ with $b \ge a$, and that $1 \le \psi \le d/2$. Then for any $\psi$-approximation $(F,H)$,
\begin{align*}
    \sum_{A \in \cG_\cD(a,b) :\ A \approx (F,H)} \om(A) \le \binom{2bd}{b - |F|} \exp{\left( -\frac{1}{2}(b - a) \overline{\alpha} + \frac{b}{d^{C}} \right)}.
\end{align*}
\end{lemma}

The second lemma is a stronger version of Kronenberg and Spinka's \cite[Lemma 4.10]{kronenberg2022independent} that gives a better upper bound in the exponential, and it comes from more carefully optimizing their proof.

\begin{lemma} \label{lemma:b-f-large}
For any fixed constants $C, \lam_0 > 0$, there exists a constant $C_0 > 0$ such that the following holds. Suppose that $\lam \le \lam_0$ and $\lam(1 - e^{-\beta})^2 \ge \frac{C_0 \log d}{d}$, that $a, b \in \mathbb{N}$ with $b \ge a$, and that $1 \le \psi \le d/2$. Then for any $\psi$-approximation $(F,H)$,
\begin{align*}
    \sum_{A \in \cG_\cD(a,b) :\ A \approx (F,H)} \om(A) \le \exp{\left(-\frac{1}{8} (b - |F|) \overline{\alpha} + \frac{b}{d^{C}} \right)}.
\end{align*}
\end{lemma}

In addition to the entropy-derived results from the previous section, the proofs of Lemmas \ref{lemma:b-f-small} and \ref{lemma:b-f-large} will involve the following facts: If $(F,H)$ is a $\psi$-approximation of a set $A \in \cG_\cD(a,b)$, then
\begin{align} \label{approximation-facts}
    |H| \le |F| + \frac{(b - a)\psi}{d - \psi} \hspace{20pt} \text{and} \hspace{20pt} |E(H, N(A) \setminus F)| \le (b - |F|)\psi.
\end{align}
For a proof of the first inequality, see \cite[Lemma 5.3]{galvin2019independent}, while the second one follows from the fact that each vertex $N(A)\setminus F$ can have at most $\psi$ neighbors in $H$ by \eqref{eqn:cont_3}. Note that the second inequality in \eqref{approximation-facts}, with the appropriate choice of $\psi$, is what leads us to our improved container lemma. The following upper bound on $|N(H)|$ will also be useful in the proofs below (for $\psi \le d/2$):
\begin{align} \label{NH-upper-bound}
    |N(H)| \le d|H| \le d\left(|F| + \frac{(b-a)\psi}{d - \psi}\right) \le d(b + b - a) \le 2bd.
\end{align}

\begin{proof}[Proof of Lemma \ref{lemma:b-f-small}]
We bound the sum of weights of $2$-linked sets with a given closure, and then combine it with a union bound to obtain an upper bound on the total weight of sets in $\cG_\cD(a,b)$ with a given approximation. We assume that $\cD = \cE$ without loss of generality. Fix a $2$-linked set $A' \subseteq \cE$ with $|A'| = a$ and $|N(A')| = b$. Let $T = A' \cup N(A')$ and let
\begin{align*}
    \cF \coloneqq \{A \cup B :\ [A] = A'\quad \text{and} \quad  B \subseteq N(A) \},
\end{align*}
which is a family of subsets of $T$. Observe that
\begin{align*}
    \tilde{\om}(\cF) = \sum_{A :\ [A] = A'}\ \sum_{B \subseteq N(A)} \lam^{|A|+|B|} e^{-\beta |E(A,B)|} = (1+\lam)^b \sum_{A \in \cG_\cE(a,b) : [A]=A'} \om(A).
\end{align*}
By Lemmas \ref{lemma:entropy-1} and \ref{lemma:bound-on-Z}, we have
\begin{align*}
    \tilde{\om}(\cF) \le \prod_{v \in N(A')} Z(\Psi_v)^{\frac{1}{d}} \le (1 + \lam)^b \exp{\left( \frac{b}{d^{C}} - \frac{\overline{\alpha}}{2d} \sum_{v \in N(A')} \ell_{\Psi_v} \right)},
\end{align*}
where $\Psi_v = \{A \cap N(v) :\ [A] = A'\}$ and $\ell_{\Psi_v} = |N(v) \setminus A'|$. (Note that Lemma \ref{lemma:bound-on-Z} applies because any $A$ with $[A] = A'$ satisfies $A \cap N(v) \neq \emptyset$ for all $v \in N(A')$, so that $\emptyset \notin \Psi_v$.) We calculate that
\begin{align*}
    \sum_{v \in N(A')} \ell_{\Psi_v} = \sum_{v \in N(A')} |N(v) \setminus A'| = \sum_{v \in N(A')} (d - |N(v) \cap A'|) = d|N(A')| - d|A'| = d(b - a).
\end{align*}
Therefore,
\begin{align*}
    \sum_{A \in \cG_\cE(a,b) :\ [A]=A'} \om(A) \le \exp{\left( \frac{b}{d^{C}} - \frac{1}{2}(b-a) \overline{\alpha} \right)}.
\end{align*}
The lemma will follow by combining this inequality with the following upper bound on the number of possible closures of sets under consideration:
\begin{align*}
    \biggl|\left\{[A] :\ A \in \cG_{\cE}(a,b) \text{ and } A \approx (F,H)\right\}\biggr| \le \binom{2bd}{b - |F|}.
\end{align*}
To derive this last inequality, note that for any $A \approx (F,H)$, we have $F \subseteq N(A) \subseteq N(H)$, and since $N(A)$ determines $[A]$, so does $N(A) \setminus F$. Thus, we can determine $[A]$ by choosing a subset of size $b - |F|$ from $N(H)$, and by (\ref{NH-upper-bound}) we have $|N(H)| \le 2bd$. The claimed inequality follows.
\end{proof}

\begin{proof}[Proof of Lemma \ref{lemma:b-f-large}]
Again assume that $\cD = \cE$ without loss of generality. This time, let $T = H \cup N(H)$ and let
\begin{align*}
    \cF \coloneqq \{A \cup B : A \in \cG_\cE(a,b),\quad B \subseteq N(A), \quad \text{and}\quad A \approx (F,H)\}.
\end{align*}
Observe that
\begin{align*}
    \tilde{\om}(\cF) = (1 + \lam)^b \sum_{A \in \cG_\cE(a,b) :\ A \approx (F,H)} \om(A).
\end{align*}
Note that for any $A \cup B \in \cF$, no vertex in $B$ is isolated in $A \cup B$ because $B \subseteq N(A)$. Hence, Lemma \ref{lemma:entropy-2} applies and states that
\begin{align*}
    \tilde{\om}(\cF) \le \prod_{v \in N(H)} Z(\Psi_v)^{\frac{p_v}{d}} \left( \frac{1}{p_v} \right)^{\frac{p_v}{d}} \left( \frac{1}{1 - p_v} \right)^{\frac{1 - p_v}{d}},
\end{align*}
where $\Psi_v \coloneqq \{A \cap N(v) : A \in \cG_\cE(a,b), A \approx (F,H), |A \cap N(v)| > 0\}$ and $p_v \coloneqq \mathbb{P}(|A \cap N(v)| > 0)$ when $A \cup B \in \cF$ is randomly chosen from $\cF$ according to $\tilde{\om}$. By Lemma \ref{lemma:bound-on-Z},
\begin{align*}
    Z(\Psi_v)^{\frac{1}{d}} \le (1 + \lam) \exp{\left( \frac{1}{d^{C+1}} - \frac{\overline{\alpha}}{2d} \ell_{\Psi_v} \right)},
\end{align*}
where $\ell_{\Psi_v} \coloneqq |N(v) \setminus \bigcup \Psi_v|$. We split $\exp{\left(-\frac{\overline{\alpha}}{2d} \ell_{\Psi_v} \right)}$ into two copies of $\exp{\left(-\frac{\overline{\alpha}}{4d} \ell_{\Psi_v} \right)}$ and write
\begin{align*} 
    \hspace{-50pt} \tilde{\om}(\cF) \le \left[ \prod_{v \in N(H)} \left( (1 + \lam) \exp{\left(\frac{1}{d^{C+1}}\right)} \right)^{p_v} \exp{\left( -\frac{\overline{\alpha}}{4d} p_v \ell_{\Psi_v} \right)}  \right] \\
    \cdot \left[\prod_{v \in N(H)} \exp{\left( -\frac{\overline{\alpha}}{4d} p_v \ell_{\Psi_v} \right)} \left( \frac{1}{p_v} \right)^{\frac{p_v}{d}} \left( \frac{1}{1 - p_v} \right)^{\frac{1 - p_v}{d}} \right]. \hspace{-50pt}
\end{align*}
We show that, for the right-hand side, the first term is at most $(1 + \lam)^b \exp{\left( \frac{b}{d^{C+1}} - \frac{1}{8}(b - |F|) \overline{\alpha} \right)}$ and the second term is at most $\exp\left(\frac{1}{4}bd\overline{\alpha} \cdot \exp\left(-\frac{1}{16}\overline{\alpha}d \right)\right)$. Then we will observe that this second term is negligible, and the lemma will follow.

For the first term, we notice that 
\begin{align*}
    \sum_{v \in N(H)} p_v = \sum_{v \in N(H)} \mathbb{P}(v \in N(A)) = \mathbb{E} |N(A)| = b,
\end{align*}
where the last equality holds because any $A \in \cG_\cE(a,b)$ satisfies $|N(A)|=b$. Next, we notice that
\begin{align*}
    \sum_{v \in N(H)} p_v \ell_{\Psi_v} = \sum_{v \in N(H)} \mathbb{P}(v \in N(A)) \ell_{\Psi_v} = \sum_{v \in N(H)} \mathbb{E}[\mathbbm{1}_{\{v \in N(A)\}} \ell_{\Psi_v}] = \mathbb{E} \sum_{v \in N(A)} \ell_{\Psi_v}.
\end{align*}
Since $\bigcup \Psi_v \subseteq H$, we have $\ell_{\Psi_v} \ge |N(v) \setminus H|$, so that 
\begin{align*}
    \sum_{v \in N(A)} \ell_{\Psi_v} &\ge |E(N(A), \cE \setminus H)| = |E(N(A), \cE)| - |E(F,H)| - |E(N(A) \setminus F, H)| \\
    &\ge bd - |F|d - (b - |F|)\psi = (b - |F|)(d - \psi) \ge (b - |F|)d/2,
\end{align*}
where the second inequality uses that $|E(N(A) \setminus F, H)| \le (b - |F|)\psi$ from (\ref{approximation-facts}). Hence, we have $\sum_{v \in N(H)} p_v \ell_{\Psi_v} \ge (b - |F|)d/2$. This proves the claimed upper bound on the first term.

For the second product, since $|N(H)| \le 2bd$ by (\ref{NH-upper-bound}), it suffices to show that each term in the product is at most $\exp\left(\frac{1}{8}\overline{\alpha} \cdot \exp(-\frac{1}{16}\overline{\alpha} d)\right)$. Taking logarithms, this is the same as showing that
\begin{align*}
    -\frac{1}{4d} \overline{\alpha} p_v \ell_{\Psi_v} + \frac{1}{d} H(p_v) \le \frac{1}{8}\overline{\alpha} e^{-\overline{\alpha}d /16},
\end{align*}
where $H(x) := x \log \left( \frac{1}{x} \right) + (1-x) \log \left( \frac{1}{1 - x} \right)$ and $H(0) = H(1) = 0$. This inequality holds when $p_v = 1$, so we now assume that $p_v < 1$. This means that $v \notin N(A)$ for some $A \in \cG_\cE(a,b)$ with $A \approx (F,H)$, so in particular $v \notin F$. By the definition of $\psi$-approximation, at most $\psi$ neighbors of $v$ belong to $H$, so that $\ell_{\Psi_v} \ge |N(v) \setminus H| \ge d - \psi \ge d/2$. Thus, we wish to show that
\begin{align*}
    -\frac{1}{8}\overline{\alpha}p_v + \frac{1}{d} H(p_v) \le \frac{1}{8}\overline{\alpha} e^{-\overline{\alpha}d /16}.
\end{align*}
If $0 \le p_v \le e^{-\overline{\alpha} d/16}$, then using that $H(p_v) \le 2p_v \log \left(\frac{1}{p_v}\right)$, that $x \log \left(\frac{1}{x}\right)$ is increasing on $(0,1/e)$, and that $\overline{\alpha} \ge \frac{C_0}{(1+\lam_0)d}$ (so that $p_v \in (0, 1/e)$), we have that
\begin{align*}
    -\frac{1}{8}\overline{\alpha} p_v + \frac{1}{d} H(p_v) \le \frac{1}{d}H(p_v) \le \frac{2}{d} p_v \log \left(\frac{1}{p_v}\right) \le \frac{1}{8} \overline{\alpha} e^{- \overline{\alpha} d/16}.
\end{align*}
If $e^{-\overline{\alpha} d/16} \le p_v \le \frac{1}{2}$, then again using that $H(p_v) \le 2p_v \log \left(\frac{1}{p_v}\right)$, we have that
\begin{align*}
    -\frac{1}{8} \overline{\alpha} p_v + \frac{1}{d} H(p_v) \le -\frac{1}{8}\overline{\alpha}p_v + \frac{2}{d} p_v \log \left(\frac{1}{p_v}\right) \le p_v \left( -\frac{1}{8}\overline{\alpha} + \frac{2}{d} \cdot \frac{\overline{\alpha} d}{16} \right) = 0.
\end{align*}
Finally, if $\frac{1}{2} \le p_v < 1$, then using that $H(p_v) \le H(\frac{1}{2}) = \log 2$ and that $\overline{\alpha} \ge \frac{C_0}{(1+\lam_0)d}$, we have that
\begin{align*}
    -\frac{1}{8}\overline{\alpha}p_v + \frac{1}{d} H(p_v) \le -\frac{1}{16}\overline{\alpha} + \frac{\log 2}{d} \le 0.
\end{align*}
This proves the claimed upper bound on the second term.

In summary, we have shown that
\begin{align*}
    \tilde{\om}(\cF) \le (1 + \lam)^b \exp{\left(-\frac{1}{4}(b - |F|)\overline{\alpha} + \frac{b}{d^{C+1}} + \frac{1}{8}bd \overline{\alpha} \cdot \exp\left(-\frac{1}{16}\overline{\alpha} d\right) \right)}.
\end{align*}
Using that $\overline{\alpha} d \ge \frac{C_0}{1+\lam_0} \log d$ and that $x \cdot \exp(-x)$ is decreasing for $x \ge 1$, we find that
\begin{align*}
    \frac{1}{8}b \cdot \overline{\alpha} d \cdot \exp\left(-\frac{1}{16}\overline{\alpha} d\right) \le 2b \cdot \frac{C_0}{1+\lam_0} \log d  \cdot d^{-C_0/(16(1+\lam_0))} \le \frac{b}{d^{C+1}}.
\end{align*}
Therefore,
\begin{align*}
    \sum_{A \in \cG_\cE(a,b) : A \approx (F,H)} \om(A) = (1+\lam)^{-b} \cdot \tilde{\om}(\cF) \le \exp{\left(-\frac{1}{4}(b - |F|) \overline{\alpha} + \frac{b}{d^C} \right)}.
\end{align*}
\end{proof}

Finally, we prove Lemma \ref{lemma:total-weight-given-approximation} by combining Lemmas \ref{lemma:b-f-small} and \ref{lemma:b-f-large}.

\begin{proof}[Proof of Lemma \ref{lemma:total-weight-given-approximation}]
We combine our previous upper bounds on total weights of sets in $\cG_\cD(a,b)$ with a given approximation, as in Lemmas \ref{lemma:b-f-small} and \ref{lemma:b-f-large}, into a single upper bound when we assume that $b \ge \left(1 + \frac{C_4}{d^{C_5}}\right)a$, and recall that $\frac{\alpha}{1+\lam_0} \le \overline{\alpha} \le \alpha$. Let $c \coloneqq \frac{1}{3(C_5+1)}$. We split into two cases. 

If $b - |F| \le \frac{c(b - a)\overline{\alpha}}{\log d}$, then we apply Lemma \ref{lemma:b-f-small} to get that
\begin{align*}
    \sum_{A \in \cG_\cD(a,b) :\ A \approx (F,H)} \om(A) &\le \binom{2bd}{\frac{c(b - a)\overline{\alpha}}{\log d}} \exp\left( -\frac{1}{2}(b - a)\overline{\alpha} + \frac{b}{d^{C_5+1}} \right). 
\end{align*}
By the bound $\binom{n}{m} \le \left(\frac{en}{m}\right)^m$, we have
\begin{align*}
    \binom{2bd}{\frac{c(b - a)\overline{\alpha}}{\log d}} \le \exp \left( \frac{c(b - a)\overline{\alpha}}{\log d} \log \left( \frac{2bd \log d}{c(b - a)\overline{\alpha}} \right) + \frac{(b - a)\overline{\alpha}}{\log d} \right).
\end{align*}
Since $\frac{b}{b-a} \le 1 + \frac{1}{C_4}d^{C_5}$, the leading term in the above exponential is $\frac{c(b - a)\overline{\alpha}}{\log d} \cdot (C_5 + 1)\log d = \frac{1}{3}(b-a)\overline{\alpha}$. In addition, we have that $\frac{b}{d^{C_5+1}} \le \frac{2(b - a)}{C_4 d}$, which is negligible compared to $(b - a)\overline{\alpha}$ because $\overline{\alpha} \ge \frac{C_0 \log d}{(1+\lam_0)d}$. We deduce that for $c' = \frac{1}{6(C_5+1)}$, in this case we have
\begin{align} \label{eq:unbounded-case-1}
    \sum_{A \in \cG_\cD(a,b) :\ A \approx (F,H)} \om(A) \le \exp \left( -c'(b - a)\overline{\alpha} \right),
\end{align}
and this is at most the claimed upper bound since $\overline{\alpha}$ is bounded (in particular, it is less than $\frac{c'}{C''} \log d$).

Otherwise, if $b - |F| > \frac{c(b - a)\overline{\alpha}}{\log d}$, then we apply Lemma \ref{lemma:b-f-large} to get that
\begin{align} \label{eq:unbounded-case-2}
    \sum_{A \in \cG_\cD(a,b) :\ A \approx (F,H)} \om(A) &\le \exp\left( -\frac{c(b - a)\overline{\alpha}^2}{8\log d} + \frac{b}{d^{C_5+2}} \right).
\end{align}
Since $\frac{b}{b-a} \le 1 + \frac{1}{C_4}d^{C_5}$ and $\overline{\alpha} \ge \frac{C_0 \log d}{(1+\lam_0)d}$, we see that $\frac{b}{d^{C_5+2}} \le \frac{2(b-a)}{C_4 d^2}$ is negligible compared to $\frac{(b - a)\overline{\alpha}^2}{\log d}$. This establishes Lemma \ref{lemma:total-weight-given-approximation}.
\end{proof}

\section{Bounding weights of non-configurations} \label{sec:nonpolymer}

In this section, we will show that the total weight of vertex sets that do not arise from the measure $\hat{\mu}$ is negligible. We let $G$ be a $d$-regular $n$-vertex bipartite graph with bipartition $(\cO, \cE)$. Note that $|\cO|=|\cE|=n/2$. Certain sets $I \subseteq V(G)$ do not arise via the measure $\hat{\mu}$ because both sides contain $2$-linked components that are too large, and thus they are neither captured by polymer configurations from $\hat{\Theta}_\cO$ nor from $\hat{\Theta}_\cE$. 

We define the family of {\em non-configurations} as consisting of all subsets $I$ of $V(G)$ such that $I \cap \cO$ and $I \cap \cE$ both contain $2$-linked components whose closure is strictly larger than $\frac{3}{8}n$.
Note that any non-configuration $I$ in particular satisfies $|[I \cap \cO]| > \frac{3}{8}n$ and $|[I \cap \cE]| > \frac{3}{8}n$. Thus the family of non-configurations is a subset of the following set:
\begin{align} \label{eq:nonpolymer}
    \cJ \coloneqq \left\{I \subseteq V(G) : \quad |[I \cap \cO]| >\frac{3}{8}n \quad \text{and} \quad |[I \cap \cE]| >  \frac{3}{8}n\right\}. 
\end{align}
Recall that the weight of a set $I \subseteq V(G)$ is defined as
\begin{align*}
    \tilde{\om}(I) \coloneqq \lam^{|I|} e^{-\beta |E(I)|}.
\end{align*}
Using the entropy tools from \Cref{sec:entropy_tools} (Lemmas \ref{lemma:entropy-3} and \ref{lemma:more-general-bound-on-Z}), we show that the contribution of  the total weight of sets in $\cJ$ is negligible compared to $Z_G=Z_G(\lam, \beta)$. This readily implies that the total contribution of non-configurations is negligible.

\begin{lemma} \label{lem:nonpolymer}
For any fixed $\lam_0 > 0$, there exists a constant $C_0 > 0$ sufficiently large such that the following holds. If $\lam \le \lam_0$ and $\lam (1 - e^{-\beta}) \ge \frac{C_0 \log d}{d^{1/2}}$, then
\begin{align*}
    \sum_{I \in \cJ} \lam^{|I|} e^{-\beta |E(I)|} \leq Z_G \cdot\exp(-\Om(n/d)).
\end{align*}
\end{lemma}

\begin{proof}
Let $m \coloneqq n d^{-1/2}$ and $s \coloneqq \frac{d}{2} \cdot \frac{\lam}{1+\lam}$. For $\cD \in \{\cE, \cO\}$, let
\[
    \cI_{\cD} \coloneqq\{I \subseteq V(G) : \exists \text{ at least } m \text{ vertices } v \in \cD \text{ such that } 1 \leq |N(v) \cap I| \leq s\},
\]
and let
\[
    \cI' \coloneqq \cJ \setminus (\cI_{\cE} \cup \cI_{\cO}). 
\]
Then $\cJ \subseteq \cI_{\cO} \cup \cI_{\cE} \cup \cI'$. We prove that $\tilde{\om}(\cI_{\cD}) \le Z_G \exp(-\Om(n/d))$ for $\cD \in \{\cE, \cO\}$ and that $\tilde{\om}(\cI') \le \exp(-\Om(n \log d))$, from which the lemma follows. 

First we prove the claimed upper bound on $\tilde{\om}(\cI_{\cD})$ for $\cD \in \{\cE, \cO\}$. To do this, we apply Lemma \ref{lemma:entropy-3} with the boundary-odd vertex set $T = V(G)$, the family $\cF = \cI_{\cD}$, and the above chosen $s$. The lemma yields that
\begin{align*}
    \tilde{\om}(\cI_{\cO}) \le \prod_{v \in \cD} Z(\Psi_v)^{\frac{p_v}{d}} Z(\Psi_v')^{\frac{p_v'}{d}} (1 + \lam)^{1 - p_v - p_v'} \left( \frac{1}{p_v} \right)^{\frac{p_v}{d}} \left( \frac{1}{p_v'} \right)^{\frac{p_v'}{d}} \left( \frac{1}{1 - p_v - p_v'} \right)^{\frac{1-p_v-p_v'}{d}},
\end{align*}
where $\Psi_v \coloneqq \{I \cap N(v) : I \in \cF, 1 \le |I \cap N(v)| \le s\}$, $\Psi_v' \coloneqq \{I \cap N(v) : I \in \cF, |I \cap N(v)| > s\}$, $p_v \coloneqq \mathbb{P}(1 \le |I \cap N(v)| \le s)$, and $p_v' \coloneqq \mathbb{P}(|I \cap N(v)| > s)$, when $I$ is randomly chosen from $\cI_{\cD}$ according to $\tilde{\om}$. We apply Lemma \ref{lemma:more-general-bound-on-Z} with $\ell = 0$ to get that 
\begin{align*}
    Z(\Psi_v) \le (1 + \lam)^d e^{-C' \alpha d} \hspace{30pt} \text{and} \hspace{30pt} Z(\Psi_v') \le (1 + \lam)^d e^{d^{-C}}
\end{align*}
where $C, C' > 0$ are suitable constants (recalling that $\frac{1}{1+\lam_0} \alpha \le \overline{\alpha} \le \alpha$).
From these upper bounds on $Z(\Psi_v)$ and $Z(\Psi_v')$, we get that
\begin{align*}
    \tilde{\om}(\cI_{\cO}) \le (1 + \lam)^{n/2} \prod_{v \in \cD} e^{-C' \alpha p_v} e^{p_v'd^{-C-1}} \left( \frac{1}{p_v} \right)^{\frac{p_v}{d}} \left( \frac{1}{p_v'} \right)^{\frac{p_v'}{d}} \left( \frac{1}{1 - p_v - p_v'} \right)^{\frac{1-p_v-p_v'}{d}}.
\end{align*}
We observe that $Z_G \ge (1 + \lam)^{n/2}$, that $\sum_{v \in \cD} p_v \ge m$ (by definition of $\cI_{\cD}$), and that $\sum_{v \in \cD} p_v' \le n/2$ (since $p_v' \le 1$ for all $v$). Thus, the bound $\tilde{\om}(\cI_{\cO}) \le Z_G \exp(-\Om(n/d))$ will follow once we show that
\begin{align*}
    e^{-C' \alpha m} e^{n/(2d^{C+1})} \left[ \prod_{v \in \cD} \left( \frac{1}{p_v} \right)^{p_v} \left( \frac{1}{p_v'} \right)^{p_v'} \left( \frac{1}{1 - p_v - p_v'} \right)^{1-p_v-p_v'} \right]^{\frac{1}{d}} \le \exp( -\Om(n/d)).
\end{align*}
This holds because $m = nd^{-1/2}$, $\alpha \geq \frac{C_0 \log d}{d^{1/2}}$, and each term in the product is at most $3$ (being the exponential of the entropy of a random variable that takes at most $3$ values). 

To prove the claimed upper bound on $\tilde{\om}(\cI')$, we show that every $I \in \cI'$ satisfies $|E(I)| \ge ms$, which would then imply that
\begin{align*}
    \tilde{\om}(\cI') &\le (1 + \lam)^n e^{-\beta m s} \le \exp\left( n \log (1 + \lam_0) - \frac{1}{2} \beta n d^{1/2} \frac{\lam}{1+\lam_0} \right) \\
    &\le \exp\left( n \log (1 + \lam_0) - n \frac{C_0}{2(1+\lam_0)} \log d \right) = \exp \left( -\Om(n \log d) \right),
\end{align*}
as required. Thus, fix $I \in \cI'$. Since $|[I \cap \cO]| > \frac{3}{8}n$ and $|[I \cap \cE]| > \frac{3}{8}n$, we deduce that $|E([I])| \ge \frac{1}{2}dn - 2 \cdot \frac{1}{8}n \cdot d = \frac{1}{4}dn$. Since the subgraph spanned by $E([I])$ has maximum degree at most $d$, has its vertex set contained in $[I] \cap N([I])$, and $N([I]) = N(I)$, we have
\begin{align*}
    |[I] \cap N(I) \cap \cO| \ge \frac{1}{d} |E([I])| \geq \frac{n}{4}.
\end{align*}
Since $I \notin \cI_{\cO}$, all but at most $m$ vertices in $\cO$ have at least $s$ neighbors in $I$, so that
\begin{align*}
    |I \cap N(I) \cap \cE| \ge (|[I] \cap N(I) \cap \cO| - m)\frac{s}{d} \ge n\left( \frac{1}{4} - \frac{1}{d^{1/2}} \right) \frac{\lam}{2(1+\lam_0)} \ge \frac{\lam n}{9(1+\lam_0)} \ge 2m,
\end{align*}
where the last inequality uses that $\lam \ge C_0 d^{-1/2}$. Because $I \notin \cI_{\cE}$, all but at most $m$ vertices in $\cE$ have at least $s$ neighbors in $I$, so that $|E(I)| \ge (|I \cap N(I) \cap \cE| - m)s \ge ms$. This finishes the proof.
\end{proof}

\section{Verifying the Koteck\'y--Preiss condition} \label{sec:kotecky}
In this section we let $G$ be a $d$-regular $n$-vertex bipartite graph with bipartition $(\cO, \cE)$ satisfying \hyperref[def:property-i]{Property I}, and we fix $\cD \in \{\cO, \cE\}$.  We will show that the cluster expansion converges and will prove \Cref{l:Kotecky}, stating that for any fixed $\lambda_0>0$ there exists $C_0>0$ such that whenever $\lambda \le \lambda_0$ and $\lambda p \geq \frac{C_0 \log^{3/2} d}{d^{\lambexp}}$ and $k \in \mathbb{N}$, we have
\[
|L_{\cD, \geq k}| \leq nd^{(C_5+7)k-C_5-1} \left( \frac{1+\lambda}{1+\lambda e^{-\beta}}\right)^{-kd+C_1 k^2}.
\]

Recall that we defined 
\begin{align*}
  \alpha \coloneqq \lambda p = \lambda(1 - e^{-\beta}) \qquad \text{and} \qquad  \ta \coloneqq \left( 1 - \frac{\alpha}{1+\lambda} \right)^{-1} = \frac{1+\lam}{1+\lam e^{-\beta}},
\end{align*}
and we have the inequalities $\ta \le 1+\lambda$ and $\log \ta \ge \frac{\alpha}{1+\lambda}$. Now we define the following functions. For each $ \ell \in \mathbb{N}$,  define
\begin{align*}
\tilde{g}(\ell) &\coloneqq     \left\{\begin{array} {c@{\quad \textup{if} \quad}l}
       (d \ell -C_1\ell^2)\log\ta - (C_5+7) \ell \log d & \ell < \sqrt{d}, \\[+1ex]
      \frac{d \ell}{2C_2} \log\ta & \sqrt{d} \leq \ell \le d^{C_3}, \\[+1ex]
      \frac{\ell}{d^{C_5+1}} & \ell > d^{C_3}.
    \end{array}\right.
\end{align*}
Recall that we defined our polymer model as 
\begin{align*}
   \cP_\cD\coloneqq \left\{A\subseteq\cD: \quad A\ \text{is $2$-linked\quad and} \quad |[A]|\leq \frac{3}{4} \cdot |\cD|=  \frac{3}{8}n\right\}.
\end{align*}
For each polymer $A \in \cP_{\cD}$, we also define \[f(A) \coloneqq d^{-(C_5+1)} |A| \hspace{20pt} \;\text{ and } \; \hspace{20pt} g(A) \coloneqq f(A) + \tilde{g}(|A|).\]
Since $f(A)$ and $g(A)$ only depend on $|A|$, we will slightly abuse notation and write $f(\ell), g(\ell)$ when the corresponding polymer has size $\ell$.
For a vertex $v \in V(G)$, we let $\cP^v_{\cD}$ denote the set of all polymers $A\in \cP_{\cD}$ that contain $v$, that is,
\[
  \cP^v_{\cD}\coloneqq\{A \in \cP_{\cD} : v \in A\}.
\]

\begin{claim}\label{c:KoteckyPreiss}
    For every $v \in \cD$,
    \[
    \sum_{A \in \cP^v_{\cD}} \om(A) \exp(f(A) + g(A)) \leq \frac{1}{d^{C_5+3}}.
    \]
\end{claim}

Before proving \Cref{c:KoteckyPreiss} we first show it implies \Cref{l:Kotecky}.

\begin{proof}[Proof of \Cref{l:Kotecky}]
    Given two polymers $A,A' \in \cP_{\cD}$, recall that $A' \nsim A$ (i.e., they are incompatible) if and only if there exists $v \in N^2(A)$ with $v \in A'$.
Thus for a fixed polymer $A \in \cP_{\cD}$, using that $|N^2(A)| \le d^2|A|$, we have that
\begin{align*}
    \sum_{A' \nsim A} \om(A') \exp(f(A') + g(A')) & \leq \sum_{v \in N^2(A)}\  \sum_{A' \in \cP^v_{\cD}} \om(A') \exp(f(A') + g(A')) \\
    &\le d^2 |A| \sum_{A' \in \cP^v_{\cD}} \om(A') \exp(f(A') + g(A')) \\
    &\le d^2 |A| \frac{1}{d^{C_5+3}} = d^{-(C_5+1)} |A| = f(A),
\end{align*}
where the last inequality is from \Cref{c:KoteckyPreiss}.
By the Koteck\'y--Preiss condition (\Cref{lem:KP}) this implies that
  \begin{equation}\label{eq:Kotecky}
    \sum_{\substack{\Gam \in \cC_{\cD}\\ \Gam \nsim A}} |\om(\Gam)| \exp(g(\Gam)) \leq f(A),
  \end{equation}
  where $g(\Gam)\coloneqq \sum_{A' \in \Gam} g(A')$.
  For each $v \in V(G)$, if we take $A=\{v\}$ then \eqref{eq:Kotecky} becomes
  \[
    \sum_{\substack{\Gam \in \cC_{\cD}\\ \Gam \nsim v}} |\om(\Gam)| \exp(g(\Gam)) \leq d^{-(C_5+1)}.
  \]
  Summing over all $v \in V(G)$, we get that
  \[
    \sum_{\Gam \in \cC_{\cD}} |\om(\Gam)| \exp(g(\Gam)) \leq n d^{-(C_5+1)}.
  \]
  In particular, omitting the clusters with $\norm{\Gam} < k$, it follows that
  \begin{equation} \label{eq:boundbigclusters}
    \sum_{\Gam \in \cC_{\cD, \geq k}} |\om(\Gam)| \exp(g(\Gam)) \leq n d^{-(C_5+1)}.
  \end{equation}

  Next, one may check that the function $\ell \mapsto \tilde{g}(\ell)/\ell$ is non-increasing (since $C_2\geq 1$, as a set in a $d$-regular graph expands by at most a factor of $d$), so that $\tilde{g}(\ell) \geq \frac{\ell}{x} \tilde{g}(x)$ for all $x \geq \ell$. Hence,
  \begin{equation} \label{eq:gandgamma}
  \begin{split}
      g(\Gam) &= \sum_{A' \in \Gam} (f(A')+\tilde{g}(|A'|)) \geq \sum_{A' \in \Gam} |A'|d^{-(C_5+1)} + \sum_{A' \in \Gam} \frac{|A'|}{\norm{\Gam}} \tilde{g} (\norm{\Gam}) \\
      &= f(\norm{\Gam}) + \tilde{g}(\norm{\Gam}) = g(\norm{\Gam}).
  \end{split}
  \end{equation}
  Combining \eqref{eq:boundbigclusters}, \eqref{eq:gandgamma} and the fact that $g(\|\Gamma\|)\geq g(k)$ when $\Gamma \in \cC_{\cD, \geq k}$, yields
  \begin{equation}\label{eq:boundrearranged}
    \big| \log \Xi_{\cD} - L_{\cD, <k} \big|  = |L_{\cD, \geq k}| \leq \sum_{\Gam \in \cC_{\cD, \geq k}} |\om(\Gam)| \leq n d^{-(C_5+1)} \exp(-g(k)),
  \end{equation}
Using \eqref{eq:boundrearranged} and the definition of $g(k)=(d k -C_1 k^2)\log\ta - (C_5+7) k \log d$, we deduce that
  \begin{align*}
      |L_{\cD, \geq k}| & \leq n d^{-(C_5+1)} \exp(-g(k))\\
      & \leq n d^{-(C_5+1)} \exp(- ((dk-C_1 k^2)\log\ta  + (C_5+7)k \log d))\\
        & = nd^{(C_5+7)k-C_5-1} \ta^{-dk+C_1k^2},
  \end{align*}
  as desired.
\end{proof}

Now we prove \Cref{l:boundingweight}, which gives a general bound on the weight of a polymer that we will use throughout our calculations. We will prove that for any polymer $A \in \cP_{\cD}$, we have
\begin{equation} \label{eq:boundingweight}
    \om(A) \leq \frac{\lambda^{|A|}}{(1+\lambda)^{|N(A)|}} (1+\lambda e^{-\beta})^{|N(A)|}.
\end{equation}

\begin{proof}[Proof of \Cref{l:boundingweight}] \label{proof:boundingweight}
Let $A \subseteq \cD$.
We may rewrite the claimed inequality \eqref{eq:boundingweight} as
\begin{align*}
    \sum_{B \subseteq N(A)} \lam^{|B|} e^{-\beta |E(A,B)|} \le (1 + \lam e^{-\beta})^{|N(A)|}.
\end{align*}
Since every vertex in $B$ contributes at least one edge to $E(A, B)$, we have that $|E(A,B)| \ge |B|$, so that $e^{-\beta |E(A,B)|} \le e^{-\beta |B|}$. Thus,
\begin{align*}
    \sum_{B \subseteq N(A)} \lam^{|B|} e^{-\beta |E(A,B)|} \le \sum_{B \subseteq N(A)} \lam^{|B|} e^{-\beta |B|} = (1 + \lam e^{-\beta})^{|N(A)|},
\end{align*}
as required.
\end{proof}

We now turn to the proof of \Cref{c:KoteckyPreiss}, where we will need the vertex-expansion properties of $G$ as given by \hyperref[def:property-i]{Property I}.

\begin{proof}[Proof of \Cref{c:KoteckyPreiss}]
    Recall that for $A \in \cP_{\cD}$, by definition of a polymer we have $|A| \leq |[A]| \leq 3n/8$. For each $v \in \cD$ and $1 \leq \ell \leq 3n/8$, we denote the set of polymers of size $\ell$ containing $v$ by \[\cP^v_{\cD, \ell} \coloneqq \{A \in \cP^v_{\cD} : |A|=\ell\}.\] 
Given $\ell \in \mathbb{N}$, by \Cref{l:counting2linked} we may bound the number of polymers of size $\ell$ containing a fixed vertex $v$ by
\begin{align} \label{eq:numberpolymers}
    |\cP^v_{\cD, \ell}| \leq (ed^2)^{\ell-1} \leq \exp(3 \ell \log d).
\end{align}

We split the proof in three cases according to the size of $A$. 

\paragraph{Case 1:} $|A| < \sqrt{d}$.

\noindent Recall that in this range, by \hyperref[def:property-i]{Property I} we have $|N(A)|\geq d|A|-C_1|A|^2.$ Combining this with \eqref{eq:numberpolymers} and \Cref{l:boundingweight}, we have
\begin{equation} \label{eq:case1start}
\sum_{\ell=1}^{\sqrt{d}} \sum_{A \in \cP^v_{\cD, \ell}} \om(A) \exp(f(A) +g(A))\\
\le \sum_{\ell=1}^{\sqrt{d}} \exp(3\ell \log d) \lam^{\ell} \ta^{-(d\ell-C_1 \ell^2)} \exp(f(\ell) +g(\ell)).
\end{equation}
Plugging in the definition of $f(\ell)$ and $g(\ell)$ in this range, the right-hand side of \eqref{eq:case1start} is
\begin{align*}
    \sum_{\ell=1}^{\sqrt{d}} \exp(3\ell \log d) \lam^{\ell} \ta^{-(d\ell-C_1 \ell^2)} &\exp\left(2 \ell d^{-(C_5+1)}+ (d \ell -C_1\ell^2)\log\ta - (C_5+7) \ell \log d\right) \\
    &\leq \sum_{\ell \geq 1} \exp\left(2 \ell d^{-(C_5+1)} + \ell \log \lam - (C_5+4) \ell \log d \right) \\
    &\le \sum_{\ell \ge 1} \exp \left( -(C_5+3.5)\ell \log d \right) \\
    &\le \frac{1}{3d^{C_5+3}},
\end{align*}
where the first inequality uses that $\lambda \le \lambda_0$ is bounded.
Thus, \eqref{eq:case1start} becomes
\begin{align} \label{eq:case1result}
    \sum_{\ell=1}^{\sqrt{d}} \sum_{A \in \cP^v_{\cD, \ell}} \om(A) \exp(f(A) +g(A)) \le \frac{1}{3d^{C_5+3}}.
\end{align}

\paragraph{Case 2:} $\sqrt{d} \leq |A| \leq d^{C_3}$.

\noindent In this case, by \hyperref[def:property-i]{Property I} we have that $|N(A)| \geq \frac{d}{C_2}|A|$, and combining this with \eqref{eq:numberpolymers} and \Cref{l:boundingweight}, we have
\begin{equation} \label{eq:case2start}
\sum_{\ell=\sqrt{d}}^{d^{C_3}} \sum_{A \in \cP^v_{\cD, \ell}} \om(A) \exp(f(A) +g(A))\\
\le \sum_{\ell=\sqrt{d}}^{d^{C_3}} \exp(3\ell \log d) \lam^{\ell} \ta^{-\frac{\ell d}{C_2}} \exp(f(\ell) +g(\ell)).
\end{equation}
Plugging in the definition of $f(\ell)$ and $g(\ell)$ in this range, the right-hand side of \eqref{eq:case2start} becomes

\begin{align}
        \sum_{\ell=\sqrt{d}}^{d^{C_3}} \exp(3\ell \log d) &\lam^{\ell}\ta^{-\frac{\ell d}{C_2}} \exp\left(2d^{-(C_5+1)} \ell + \frac{d \ell}{2C_2} \log\ta\right) \nonumber\\
        &\leq  \sum_{\ell \leq d^{C_3}} \exp\left( 3\ell \log d + \ell \log \lam + 2 \ell d^{-(C_5+1)} - \frac{d\ell }{2C_2} \log\ta\right )\nonumber \\
        &\leq \sum_{\ell \leq d^{C_3}} \exp\left(- \frac{d\ell}{4C_2} \log\ta\right) \nonumber\\
        &\leq d^{C_3}\exp\left(-\frac{C_0 \log^{3/2}d}{4C_2(1+\lambda_0)} \right) \label{eq:Case2rootd},
\end{align}
where the second inequality uses that $\lam \le \lam_0$ is bounded, $\log \ta > 0$ and $\log \ta \ge \frac{\alpha}{1+\lambda} \ge \frac{C_0 \log^{3/2} d}{(1+\lambda_0)d}$ to show that the last term is dominating. The last inequality uses the latter bound on $\log \ta$ again. Since the exponential term above decays faster than any polynomial in $d$, combining this with \eqref{eq:case2start} we get that
\begin{align} \label{eq:case2result}
    \sum_{\ell=\sqrt{d}}^{d^{C_3}}\ \sum_{A \in \cP^v_{\cD, \ell}} \om(A) \exp(f(A) +g(A)) &\le \frac{1}{3 d^{C_5+3}}.
\end{align}

\paragraph{Case 3:} $d^{C_3} < |A| \leq \frac{3}{8}n$.

\noindent In this case, $f(A)=\frac{|A|}{d^{C_5+1}}$ and $g(A)=\frac{2|A|}{d^{C_5+1}}$, and we bound
\begin{equation} \label{eq:case3KP}
    \sum_{\ell=d^{C_3}}^{3n/8} \sum_{A \in \cP^v_{\cD, \ell}} \om(A) \exp(f(A) +g(A)) \leq \sum_{\ell=d^{C_3}}^{3n/8} \sum_{A \in \cP^v_{\cD, \ell}} \om(A) \exp\left(3 |A|d^{-(C_5+1)}\right).
\end{equation}
Note that because $|A| \leq |[A]|$, we have
\[
\{A \in \cP_{\cD} : |A| \geq d^{C_3}\} \quad \subseteq \quad  \{A \in \cP_{\cD} : |[A]| \geq d^{C_3}\},
\]
and thus it suffices to bound
\[
\sum_{A \in \cP_{\cD}:\ |[A]|\geq d^{C_3}} \om(A) \exp\left(3 |A|d^{-(C_5+1)}\right).
\]
Due to the upper bound of $|A|\leq \frac{3}{8} n$, by \hyperref[def:property-i]{Property I} we have $|N(A)| \geq \left(1+\frac{C_4}{d^{C_5}}\right) |A|$. Set $\eta\coloneqq \left(1+\frac{C_4}{d^{C_5}}\right)$. Since $N([A])=N(A)$, we may split the sum up as
\begin{equation} \label{eq:beforecontainer}
    \sum_{A \in \cP_{\cD}:\ |[A]|\geq d^{C_3}} \om(A) \exp(3 |A|d^{-(C_5+1)}) \leq \sum_{a \geq d^{C_3}} \exp(3a d^{-(C_5+1)}) \sum_{b \geq \eta a} \ \sum_{A \in \cG_{\cD}(a, b)} \om(A).
\end{equation}
By \Cref{lem:container_main}, there exists a constant $C>0$ such that
\[
\sum_{A \in \cG_{\cD}(a, b)} \om(A) \leq n \exp\left(-\frac{C(b-a) \alpha^2}{\log d}\right).
\]
Since $b-a \geq (\eta-1) a=C_4d^{-C_5}a$, \eqref{eq:beforecontainer} becomes
\begin{align*} 
    \sum_{A \in \cP_{\cD}:\ |A|\geq d^{C_3}} \om(A) \exp \left(3 |A|d^{-(C_5+1)} \right) &\leq n^2 \sum_{a \geq d^{C_3}} \exp\left(3a d^{-(C_5+1)}-\frac{CC_4 a \alpha^2}{ \log d} d^{-C_5}\right).
\end{align*}
Since $\alpha \geq \frac{C_0 \log^{3/2} d}{d^{1/2}}$ and $a \ge d^{C_3}$, we have that
\begin{equation} \label{eq:pluginalpha}
    3a d^{-(C_5+1)}-\frac{CC_4 a \alpha^2}{\log d}d^{-C_5} \le -C'ad^{-(C_5+1)} \le -C'd^{C_3-C_5-1},
\end{equation}
for some constant $C' > 0$, which gives us that
\begin{equation} \label{eq:case3KPresult}
    \sum_{A \in \cP_{\cD}:\ |A|\geq d^{C_3}} \om(A) \exp\left(3 |A|d^{-(C_5+1)}\right) \leq n^2\exp\left( -C'd^{C_3-C_5-1} \right).
\end{equation}
Plugging \eqref{eq:case3KPresult} into \eqref{eq:case3KP}, we get that
\begin{equation}\label{eq:case3result}
    \sum_{\ell=d^{C_3}}^{3n/8}\ \sum_{A \in \cP^v_{\cD, \ell}} \om(A) \exp(f(A) +g(A)) \leq n^2 \exp\left(-C' d^{C_3-C_5-1} \right) \leq \frac{1}{3d^{C_5+3}},
\end{equation}
where the last inequality uses the conditions $C_3>C_5 + 2$ and $\log n=O(d)$ (i.e., the size constraint in terms of the regularity), from \hyperref[def:property-i]{Property I}.

Finally, combining \eqref{eq:case1result}, \eqref{eq:case2result}, and \eqref{eq:case3result}, we get that
\[
\sum_{A \in \cP^v_{\cD}} \om(A) \exp(f(A) + g(A)) \leq \frac{1}{d^{C_5+3}},
\]
as required.
\end{proof}

\noindent It is worth mentioning that the last inequality in \eqref{eq:case3result} is the only place that requires the size constraint of \hyperref[def:property-i]{Property I}, and this will be relevant for \Cref{sec:tori}.

\section{Properties of the measure $\hat{\mu}$} \label{sec:propertiesmu}

Recall that $G$ is a $d$-regular $n$-vertex bipartite graph with bipartition $(\cO, \cE)$ satisfying \hyperref[def:property-i]{Property I}. The main goal of this section is to further describe the structure of sets generated by the measure $\hat{\mu}$, and eventually show that $\hat{\mu}$ is indeed a good approximation of the measure ${\mu}$ of the Ising model in terms of total variation distance. The bounds derived by the verification of the Koteck\'y--Preiss condition will be central in this direction. 

We start by describing some typical structural properties of the defect set.
Recall from \Cref{sec:proof_main} that we sample a set $I$ by first sampling a decorated polymer configuration $\hat{\Theta}= \{(A_1, B_1), ..., (A_s, B_s)\}$ from $\hat{\Omega}_{\cD}$, forming $D \coloneqq \bigcup_{i=1}^s (A_i \cup B_i)$. In the set $I$, we include $D$ and every vertex that is \emph{not} in $\cD \cup N(D)$ independently with probability $\frac{\lambda}{1+\lambda}$.

\begin{lemma} \label{l:Gammasmall}
Let $(I, \cD)$ be drawn according to the measure $\hat{\mu}^*$ defined in \eqref{eq:gettinghatZ}. Then
\[
\mathbb{P}_{\hat{\mu}^*}\left( |I \cap \cD| \geq \frac{n}{d^2} \right) \leq \exp\left(-\frac{n}{d^{C_5+4}}\right).
\]
\end{lemma}
\begin{proof}
    We define a different weight function $\om^{*}$ on $\cP_\cD$, where for $A \in \cP_\cD$,
    \[
    \om^{*}(A) \coloneqq  \om(A) \exp(|A|\, d^{-(C_5+1)}).
    \]
    Analogous to the cluster expansion $\log \Xi_{\cD}$ in $\om(A)$, the cluster expansion in $\om^{*}(A)$ is given by \[\log \Xi^{*}_{\cD}= \sum_{\Gam \in \cC_{\cD}} \om^{*}(\Gam).\] The same calculation shows that \Cref{c:KoteckyPreiss} also holds for $\om^{*}(A)$, and thus the Koteck\'y--Preiss condition is satisfied for $\om^{*}(A)$, $f(A)$ and $g(A)$. Furthermore, for any polymer configuration $\Theta \in \Om_\cD$, we have
    \begin{equation} \label{eq:omegatilde}
        \prod_{A \in \Theta} \om^{*}(A) = \prod_{A \in \Theta} \om(A) \exp(|A| d^{-(C_5+1)}) = \exp(\norm{\Theta} d^{-(C_5+1)}) \prod_{A \in \Theta} \om(A),
    \end{equation}
    where $\norm{\Theta} \coloneqq \sum_{A \in \Theta} |A|$.
    For a polymer configuration $\Theta \in \Om_{\cD}$, we set
    \[
    \nu_{\cD}(\Theta) \coloneqq \frac{\om(\Theta)}{\Xi_{\cD}}.
    \]
    Hence, using \eqref{eq:omegatilde} we get that
    \begin{equation} \label{eq:beforelog}
        \begin{split}
             \frac{\Xi^{*}_\cD}{\Xi_\cD}&=\frac{\sum_{\Theta \in \Om_\cD} \exp\left(\norm{\Theta} d^{-(C_5+1)}\right) \prod_{A \in \Theta} \om(A)}{\Xi_{\cD}}\\
            &= \sum_{\Theta \in \Om_\cD}  \exp\left(\norm{\Theta} d^{-(C_5+1)}\right) \nu_\cD(\Theta)\\
            &= \mathbb{E}_{\nu_\cD} \left[ \exp\left(\norm{\Theta} d^{-(C_5+1)}\right) \right],
        \end{split}
    \end{equation}
    where the expectation is with respect to the probability distribution on $\Om_{\cD}$, with a set $\Theta$ drawn with probability $\nu_{\cD}(\Theta)$. Note that $I \cap \cD$ is indeed in $\Om_{\cD}$ and its distribution under the measure $\hat{\mu}$ is exactly the one given by $\nu_{\cD}$.
 Recall that for a random variable $X$, the $r$th-cumulant generating function of $X$ is
 \[
 K_X(r) \coloneqq \log \Ex{\exp(rX)}.
 \]
 Taking logarithms on both sides of \eqref{eq:beforelog}, since $\Xi_\cD \geq 1$, we conclude that
    \begin{equation}\label{e:tildeXilower}
      \log \Xi^{*}_\cD \geq  K_{|I \cap \cD|}(d^{-(C_5+1)}).
    \end{equation}
    On the other hand, the analogue of \eqref{eq:boundrearranged} holds for $\om^{*}(\Gam)$, and applying this with $k=1$ (noting that $\cC_{k\geq 1} = \cC$) yields

    \begin{align} \label{e:tildeXiupper}
        \log \Xi^{*}_\cD =  \sum_{\Gam \in \cC} \om^{*}(\Gam) \leq n d^{-(C_5 + 1)} \exp(-g(1)) \leq n d^{6} \ta^{C_1-d}. 
    \end{align}
    Hence, from \eqref{e:tildeXilower} and \eqref{e:tildeXiupper} we get that
    \begin{equation} \label{eq:boundcumulant}
        K_{|I \cap \cD|}\left(d^{-(C_5+1)}\right) \leq \log \Xi^{*}_\cD \leq n d^{6} \ta^{C_1-d}.
    \end{equation}
Combining \eqref{eq:boundcumulant} and Markov's inequality \eqref{eq:Markov}, we have that
    \begin{align*}
            \mathbb{P}\left(|I \cap \cD|\geq \frac{n}{d^2}\right) & \leq \exp\left(K_{|I \cap \cD|}\left(d^{-(C_5+1)}\right) - \frac{n}{d^2} d^{-(C_5+1)} \right)\\
            &\leq \exp\left( nd^{6}\ta^{C_1-d} - nd^{-(C_5+3)} \right).
    \end{align*}
Since $\ta \le 1+\lam_0$ and $\log \ta \ge \frac{\alpha}{1+\lambda} \ge \frac{C_0 \log^{3/2}d}{(1+\lambda_0)d}$, we have that
\begin{align*}
    d^{6} \ta^{C_1-d} \le d^{6} (1+\lambda_0)^{C_1} \exp\left(-\frac{C_0 \log^{3/2} d}{1+\lambda_0} \right) \le d^{-(C_5+4)}.
\end{align*}
Therefore,
    \[
    \mathbb{P}\left(|I \cap \cD|\geq \frac{n}{d^2}\right) \leq \exp\left(-\frac{n}{d^{C_5+4}}\right),
    \]
    as claimed.
\end{proof}

Recall that for a set $I \subseteq V(G)$, the minority side $\cM(I)$ is one of the sides $\cM(I) \in \{\cO, \cE\}$ such that $|\cM(I) \cap I| \leq |I\setminus \cM(I)|$.
\begin{lemma} \label{l:minoritydefect}
Let $(I, \cD)$ be drawn according to the measure $\hat{\mu}^*$ defined in \eqref{eq:gettinghatZ}  and let $\cM(I)$ be the \emph{minority} side of $I$. Then
\[
\mathbb{P}_{\hat{\mu}^*}( \cD \neq \cM(I)) =  O\left(\exp\left(-\frac{n}{d^{C_5+4}}\right)\right).
\]
\end{lemma}
\begin{proof}
    By \Cref{l:Gammasmall}, we have
    \begin{equation} \label{eq:failure}
        \mathbb{P}_{\hat{\mu}^*}( \cD \neq \cM(I))\leq \mathbb{P}_{\hat{\mu}^*}\left( \cD \neq \cM(I) \;\middle|\; |I \cap \cD| \leq \frac{n}{d^2}\right) + \exp\left(-\frac{n}{d^{C_5+4}}\right).
    \end{equation}
    In the following, we estimate the probability that $\cD \neq \cM(I)$ under the conditioning that $|I \cap \cD| \leq n/d^2$.
    Since $|I \cap \cM(I)| \leq |I \cap \cD|$, it suffices to show that
    \[
    \mathbb{P}_{\hat{\mu}^*}\left( |I \cap \overline{\cD}| \leq \frac{n}{d^2} \;\middle|\; |I \cap \cD| \leq \frac{n}{d^2} \right) = O\left(\exp\left(-\frac{n}{d^{C_5+4}}\right)\right).
    \]
   Let $X\coloneqq|I \cap \overline{\cD}|$, and note that $X=|D \cap \overline{\cD}|+|I \setminus D|$, where $D \cap \overline{\cD}$ contains the vertices from the decorated polymers on the non-defect side and every vertex of $I \setminus D$ was randomly chosen from $\overline{\cD} \setminus N(D)$ independently with probability $\frac{\lam}{1+\lam}$.
   Hence, $X$ stochastically dominates the random variable $Y \sim \Bin\left(M', \frac{\lam}{1+\lam}\right)$ where
    $M' \coloneqq\frac{n}{2}-|N(D) \cap \overline{\cD}|$.
    Since $\frac{\lam}{1+\lam}\geq \frac{1}{d}$, $M' \geq \frac{n}{2}-d|D \cap \cD|$ and $|D \cap \cD| =|I \cap \cD| \leq \frac{n}{d^2}$, we get that
    \[
    \mathbb{E}\left[Y\right] \geq \frac{\lam}{1+\lam} \cdot \frac{n}{2} \left(1-\frac{2}{d}\right) = \Om\left(\frac{n}{d}\right).
    \]
    Applying a Chernoff bound (\Cref{l:Chernoff}) to $Y$, we obtain that
    \begin{align}
 \mathbb{P}_{\hat{\mu}^*} \left( X \leq \frac{n}{d^2} \;\middle|\; |I \cap \cD| \leq \frac{n}{d^2}\right) &\leq \Pr*{|Y - \mathbb{E}[Y]| \geq \frac{n}{d}} \nonumber\\
        &\leq  2 \exp\left(-\frac{2 n^2}{d^2M'} \right) =  \exp\left(-\Om\left(\frac{n}{d^2}\right)\right). \label{eq:Chernoffapplication}
    \end{align}
    From \eqref{eq:failure} and \eqref{eq:Chernoffapplication} it follows that
    \begin{align*}
    \mathbb{P}_{\hat{\mu}^*}\left( \cD \neq \cM(I) \right)
    &= O\left(\exp\left(-\frac{n}{d^2}\right)\right) + \exp\left(-\frac{n}{d^{C_5+4}}\right)\\
    &= O\left(\exp\left(-\frac{n}{d^{C_5+4}}\right)\right),
    \end{align*}
    as required.
\end{proof}

Now we prove \Cref{lem:approximateZ}, which tells us that $Z_G$ and $\hat{Z}_G \coloneqq (1+ \lam)^{n/2} (\Xi_\cO + \Xi_\cE)$ are close to each other on a logarithmic scale. In particular, this implies that the total variation distance between the two measures $\mu$ and $\hat{\mu}$ is small. We will prove that for any fixed $\lambda_0>0$ there exists a constant $C_0>0$ such that if $\lambda \leq \lambda_0$ and $\lambda p \geq \frac{C_0 \log^{3/2} d}{d^{\lambexp}}$, then
\begin{equation} \label{eq:differencehat}
    \left| \log Z_G(\lambda, \beta)- \log \left( (1+\lambda)^{n/2} (\Xi_{\cO} + \Xi_{\cE})\right)\right| = O\left( \exp\left( - \frac{n}{d^{C_5+4}}\right) \right),
\end{equation}
and
\begin{equation} \label{eq:totalvariation}
    ||\hat{\mu}-\mu||_{TV} = O\left( \exp\left( -\frac{n}{d^{C_5+4}}\right)\right).
\end{equation}

\begin{proof}[Proof of \Cref{lem:approximateZ}]
    We recall the crucial components of the measures $\mu$ and $\hat{\mu}$.
    For each subset $I \subseteq V(G)$, we have
    \[
    \tilde{\om}(I) = \lam^{|I|} e^{-\beta |E(I)|} \quad \text{ and } \quad \hat{\om}(I)= \hat{\om}^*(I, \cO) + \hat{\om}^*(I, \cE)\]
    where $\hat{\om}^*(I, \cD)= \mathbbm{1}_{I \cap \cD \in \Om_{\cD}} \om(I)$ as in \eqref{eq:hatweight}.
    Recall that $Z_G$ and $\hat{Z}_G$ are the partition functions of the corresponding probability measures $\mu$ and $\hat{\mu}$ given by 
    \[
Z_G \coloneqq \sum_{I \subseteq V(G)} \tilde{\om}(I) \quad \text{ and } \quad \hat{Z}_G \coloneqq \sum_{I \subseteq V(G)} \hat{\om}(I)
\]
and that for every subset $I \subseteq V(G)$,
\[
\mathbb{P}_{\mu}(\mathbf{I}=I) = \frac{\tilde{\om}(I)}{Z_G} \quad \text{ and } \quad \mathbb{P}_{\hat{\mu}}(\mathbf{I}=I) = \frac{\hat{\om}(I)}{\hat{Z}_G}.
\]
However, if we compare the definitions of the weights $\tilde{\om}$ and $\hat{\om}$, we see that $\tilde{\om}$ and $\hat{\om}$ will agree for any $I \subseteq V(G)$ such that precisely one of $I \cap \cO \in \Om_{\cO}$ or $I \cap \cE \in \Om_{\cE}$ holds. 
We will say that $I \subseteq V(G)$  is captured by the odd polymer model if $I \cap \cO \in \Om_{\cO}$, and similarly captured by the even polymer model if $I \cap \cE \in \Om_{\cE}$.
The sets which are captured by none of the even or odd  polymer models are exactly the sets in $\cJ$ as in \eqref{eq:nonpolymer}. By \Cref{lem:nonpolymer}, we have
\begin{equation} \label{eq:nonpolymerbound}
    \sum_{I \in \cJ} \tilde{\om}(I) \leq Z_G \cdot \exp(-\Om(n/d)).
\end{equation}
Let $\cB\subseteq 2^{V(G)}$ consist of all sets that are captured by both the even and odd  polymer models. Letting $\cM(I)$ and $ \overline{\cM}(I)$ be the minority and majority side of $I$ (so we have $\{\cM(I), \overline{\cM}(I)\} = \{\cO, \cE\}$), we note that  
\begin{equation*} 
\begin{split}
    \mathbb{P}_{\hat{\mu}^*}(\cD \neq \cM(I)) &= \sum_{I \subseteq V(G)} \hat{\mu}^*(\overline{\cM}(I), I)\\
    &= \sum_{I \in \cB} \frac{\mathbbm{1}_{\{I \cap {\overline{\cM}(I)} \in \Om_{\overline{\cM}(I)}\}} \lam^{|I|} e^{-\beta |E(I)|}}{\hat{Z}_G}\\
    &= \sum_{I \in \cB} \frac{\lam^{|I|} e^{-\beta |E(I)|}}{\hat{Z}_G} = \frac{1}{\hat{Z}_G} \sum_{I \in \cB} \tilde{\om}(I).
\end{split}
\end{equation*}
Using \Cref{l:minoritydefect}, we obtain that
\begin{equation}\label{eq:doublecount}
\sum_{I \in \cB} \tilde{\om}(I) = \hat{Z}_G \cdot \mathbb{P}_{\hat{\mu}^*}(\cD \neq \cM(I)) \leq \hat{Z}_G \cdot O\left(\exp\left(-\frac{n}{d^{C_5+4}}\right)\right).
\end{equation}
However, by definition of the weight $\hat{\om}$, for sets $ I \in \cB$ it holds that $\hat{\om}(I) - \tilde{\om}(I) = \lam^{|I|} e^{-\beta |E(I)|} = \tilde{\om}(I)$.
From the definition of $\cB$ and $\cJ$, we may bound
\[
\hat{Z}_G- \sum_{I \in \cB} \tilde{\om}(I) \leq Z_G \leq \hat{Z}_G + \sum_{I \in \cJ} \tilde{\om}(I).
\]
Plugging in \eqref{eq:nonpolymerbound} and \eqref{eq:doublecount}, we obtain that
\[
\hat{Z}_G - \hat{Z}_G \cdot O\left(\exp\left(-\frac{n}{d^{C_5+4}}\right)\right) \leq Z_G \leq \hat{Z}_G + Z_G \cdot \exp\left(-\Om\left(\frac{n}{d}\right)\right).
\]
In particular,
\begin{equation} \label{e:Z_Zprime}
1-O\left(\exp\left(-\frac{n}{d^{C_5+4}}\right)\right) \leq \frac{Z_G}{\hat{Z}_G} \leq 1+ O\left(\exp\left(-\frac{n}{d}\right)\right)
\end{equation}
and taking logarithms,
\[
\log \left( 1-O\left(\exp\left(-\frac{n}{d^{C_5+4}}\right)\right) \right) \leq \log Z_G - \log \hat{Z}_G \leq \log \left( 1 + O\left(\exp\left(-\frac{n}{d}\right)\right) \right).
\]
This expansion implies \eqref{eq:differencehat}.

Now we also bound the total variation distance between the measures $\hat{\mu}$ and $\mu$.
Since $I \notin \cB$ implies $\hat{\om}(I) = \tilde{\om}(I)$, we have that
\begin{equation}\label{e:TV_bound_prelim}
    \norm{\hat{\mu} - \mu}_{TV} =  \sum_{\substack{I \subseteq V(G) \\ \hat{\mu}(I) > \mu(I)}} (\hat{\mu}(I) - \mu(I)) \leq \mathbb{P}_{\hat{\mu}}(I \in \cB) + \sum_{I \not\in \cB} \left|\frac{\tilde{\om}(I)}{\hat{Z}_G} - \frac{\tilde{\om}(I)}{Z_G}\right|.
\end{equation}
For the right-hand side of \eqref{e:TV_bound_prelim}, we observe that 
\[ \mathbb{P}_{\hat{\mu}}(I \in \cB) = \sum_{I \in \cB} \frac{2\lam^{|I|} \exp(-\beta |E(I)|)}{\hat{Z}_G} = 2 \cdot \mathbb{P}_{\hat{\mu}}(\cD \neq \cM(I))= O\left(\exp\left(-\frac{n}{d^{C_5+4}}\right)\right) \]
and applying \eqref{e:Z_Zprime} we get that
    \[\norm{\hat{\mu} - \mu}_{TV} \leq  \mathbb{P}_{\hat{\mu}}(I \in \cB) + \sum_{I \not\in \cB} \frac{\om(I)}{Z_G} \cdot O\left(\exp\left(-\frac{n}{d^{C_5+4}}\right)\right) = O\left(\exp\left(-\frac{n}{d^{C_5+4}}\right)\right), \]
as desired for \eqref{eq:totalvariation}.
\end{proof}

\section{Concrete examples for graphs with Property I}\label{sec:examples}

\subsection{Cartesian product graphs} \label{sec:product}

Graphs that are commonly studied to determine their number of independent sets include the hypercube, grids, discrete tori, and Hamming graphs. These examples are instances of Cartesian product graphs, namely they are the Cartesian product of edges, paths, cycles, or complete graphs, respectively.
Given a sequence $(H_i)_{i \in [t]}$ of connected graphs, their \emph{Cartesian product} is the graph $G= \square_{i=1}^t H_i$ with vertex set
 \[
  V(G) \coloneqq \prod_{i=1}^t V(H_i) = \{(x_1,\dots, x_t) : x_i \in V(H_i) \text{ for all } i \in [t]\}
\]
and edge set
\[
  E(G) \coloneqq \left\{ \bigl\{ x,y \bigr\} \colon \text{ there exists } i \in [t] \text{ such that } \{x_i,y_i\} \in E(H_i) \text{ and } x_j=y_j \text{ for all }j \neq i \right\}.
\]
The graphs $H_i$ are called the \emph{base graphs} of the Cartesian product graph $G= \square_{i=1}^t H_i$.
Note that if the base graphs $H_i$ are bipartite and $d_i$-regular for all $i \in [t]$, then their Cartesian product is a $d$-regular bipartite graph with $d\coloneqq\sum_{i=1}^t d_i$.

In \cite{CEGK25}, the authors derive asymptotics for the number of independent sets in these graphs. To that end, they prove the following isoperimetric properties for Cartesian product graphs.

\begin{theorem}\label{l:isoperimetry}
Fix integers $s,t \ge 1$ and positive real $k > 0, q \in (0,1]$. Let $(H_i)_{i \in [t]}$ be a sequence of connected, regular, bipartite graphs with $|V(H_i)| \leq s$ for all $i \in [t]$.
  Let $G = \square_{i=1}^t H_i$ with bipartition $(\mathcal{O},\mathcal{E})$.
  Then for any $X \subseteq \mathcal{O}$ or $X \subseteq \mathcal{E}$, the following hold.
  \begin{enumerate}
  \item The maximum codegree of $G$ is at most $s$.
  \item If $|X| \leq t^k$, then $|N(X)| \geq t|X|/c$, for some constant $c$ depending only on $k$.
  \item If $|X| = q|V(G)|/2$, then
    \begin{equation*}
      |N(X)| \geq |X| \left( 1 + \frac{2 \sqrt{2}(1-q)}{s \sqrt{t}}\right).
    \end{equation*}
  \end{enumerate}
\end{theorem}

In particular, our main results apply in this setting and allow us to study the number of independent sets in \textit{percolated} Cartesian product graphs. Letting $G=\square_{i=1}^t H_i$ be as in \Cref{l:isoperimetry} we may then apply \Cref{thm:expansion} if $s$ is a fixed constant by setting
\begin{align*}
C_1=s,\; C_2=\frac{s \cdot c(k)}{2}, \;C_3=k,\; C_4=\frac{1}{s\sqrt{2}}, \text{ and }\; C_5=1/2,
\end{align*}
where $k>0$ is any positive real number and $c(k)$ is the constant given by the second statement in \Cref{l:isoperimetry} depending on $k$.
In particular, the value of $C_5=1/2$ resembles many commonly studied cases showing similar vertex-expansion properties, including Cartesian product graphs with constant-size base graphs, such as the hypercube. Other graphs like the middle layer graph (where $C_5=1$) \cite{Katona_2009, Kruskal_1963} exhibit similar vertex-expansion.

In what follows, we consider more concrete examples, namely the Cartesian product of complete bipartite graphs $K_{s, s}$, the middle layer graph, and even tori.

\subsection{Cartesian products of complete bipartite graphs}

Given $s, t \in \mathbb{N}$ let $G=\square_{i=1}^t K_{s,s}$ be the Cartesian product of complete bipartite graphs and note that $n \coloneqq |V(G)| =(2s)^t$ and $G$ is $d$-regular with $d \coloneqq st$.
By \Cref{thm:expansion}, taking $\lambda=1$ and $k=2$ we have that as $t\to \infty$  (and thus as $d\to \infty$)
\begin{align}
 \Ex{i(G_p)}= 2 \cdot 2^{(2s)^t/2} \exp\left( \frac{(2s)^t}{2^{st+1}} (2-p)^{st}+ L_2 + O\left((2s)^t(st)^4 \left(1-\frac{p}{2} \right)^{3st}\right)\right), \label{Kss}  
\end{align}
where $L_2 =L_{\cD,2} \coloneqq \sum_{\Gam \in \cC_{\cD,2}} \om(\Gam)$. Due to symmetry we may assume $\cD=\cE$ and a direct calculation will give us an explicit expression for $L_2$. 

\smallskip
\noindent\textit{Case 1:} $\Gam=(\{u\},\{u\})$. The calculation is identical to even tori (given in \Cref{sec:tori}) and gives a contribution of
\begin{equation}\label{eqn:compl_cluster2_1}
    -\frac{1}{2}\cdot\frac{(2s)^t}{2}\left(\frac{2-p}{2}\right)^{2st}.
\end{equation}

\noindent\textit{Case 2:} $\Gam=(\{u\},\{v\})$ with $\{u,v\}$ $2$-linked and $u\neq v$. The ways to create such a cluster is to choose some $u$ and then $v\in N^2(u)$. It is easy to see that $|N^2(u)|=(s-1)t+s^2\binom{t}{2}$, by choosing one or two coordinates of $u$ that will be changed, each of which gives $s-1$ or $s^2$ choices respectively. Hence the total contribution of these clusters is
\begin{equation}\label{eqn:compl_cluster2_2}
   -\frac{1}{2}\cdot\frac{(2s)^t}{2}\left((s-1)t+s^2\binom{t}{2}\right)\left(\frac{2-p}{2}\right)^{2st}.
\end{equation}

\noindent\textit{Case 3:} $\Gam=(\{u,v\})$ with $\{u,v\}$ $2$-linked ($u\neq v$) and $v$ has one different coordinate than $u$. Then $\operatorname{codeg}(u,v)=s$ and $|N(\{u,v\})|=2d-s$. Now we will have $2d-2s$ vertices contributing a factor of $(1+e^{-\beta})$ and $s$ vertices contributing a $(1+e^{-2\beta})$ factor. Since for each choice of $u$ we have $st$ choices for $v$ (choosing the coordinate and then moving to another vertex on the same side of that bipartite subgraph), we get a total contribution of
\begin{equation}\label{eqn:compl_cluster2_3}
    \frac{1}{2}\cdot\frac{(2s)^t}{2}(s-1)t\left(\frac{2-p}{2}\right)^{2st-2s}\left(\frac{1+(1-p)^2}{2}\right)^s.
\end{equation}

\noindent\textit{Case 4:} $\Gam=(\{u,v\})$ with $\{u,v\}$ $2$-linked ($u\neq v$) and $v$ has two different coordinates than $u$. For each $u$, we have $s^2\binom{t}{2}$ options for $v$ and $\operatorname{codeg}(u,v)=2$. Thus, as in the respective case for even tori, we get a total contribution of
\begin{equation}\label{eqn:compl_cluster2_4}
    \frac{1}{2}\cdot\frac{(2s)^t}{2}s^2\binom{t}{2}\left(\frac{2-p}{2}\right)^{2st-4}\left(\frac{1+(1-p)^2}{2}\right)^2.
\end{equation}

Now summing up \eqref{eqn:compl_cluster2_1}, \eqref{eqn:compl_cluster2_2}, \eqref{eqn:compl_cluster2_3}, \eqref{eqn:compl_cluster2_4} we get that
\begin{align}
    L_2= &\frac{1}{2}\cdot\frac{(2s)^t}{2}s^2\binom{t}{2}\left(\frac{2-p}{2}\right)^{2st-4}\left(\frac{1+(1-p)^2}{2}\right)^2\nonumber\\
    +&\frac{1}{2}\cdot\frac{(2s)^t}{2}(s-1)t\left(\frac{2-p}{2}\right)^{2st-2s}\left(\frac{1+(1-p)^2}{2}\right)^s\nonumber\\
    -&\frac{1}{2}\cdot\frac{(2s)^t}{2}\left(\frac{2-p}{2}\right)^{2st} \left(1+ (s-1)t+s^2\binom{t}{2}\right).\label{KssL2}
\end{align}
Taking $s=1$ in \eqref{KssL2} and plugging it into  \eqref{Kss} recovers the asymptotic estimate: as $t\to \infty$,
\begin{align*}
\Ex{i(Q^t_p)}= 2 \cdot 2^{2^{t-1}} \exp\left( \frac{1}{2}(2-p)^t+ 2^t\left(\frac{2-p}{2}\right)^{2t} \left(a(p) \binom{t}{2}-\frac{1}{4} \right) + O\left(2^t t^4 \left(1-\frac{p}{2} \right)^{3t}\right)\right),
\end{align*}
with $a(p)=\frac{(1+(1-p)^2)^2}{(2-p)^4}-\frac{1}{4}$, which was 
obtained by Kronenberg and Spinka (\cite[Theorem 1.1]{kronenberg2022independent}).

Note that $\log_2 n = t \log_2 (2s)$.
For any fixed $s$, by \Cref{rem:sharp}, we obtain a sharp asymptotic result with only these two terms for any $p> 2-2(2s)^{1/(3s)}$.

\subsection{The middle layer graph}
Let $G$ be the middle layer graph, that is, the subgraph of the hypercube $Q^{2d-1}$ consisting of all binary vectors of length $2d-1$ containing either $d-1$ or $d$ entries with a one. In particular, two such vectors are connected if they differ in exactly one coordinate. Note that  $n\coloneqq |V(G)| =2 \cdot \binom{2d-1}{d-1}$ and $G$ is $d$-regular. Furthermore, any two vertices that form a $2$-linked set have codegree $1$. Let $\cL_d,\cL_{d-1}$ denote the bipartition classes of $G$. By \Cref{thm:expansion}, taking $\lambda=1$ and $k=2$ we have that as $d\to \infty$,
\[
\Ex{i(G_p)}= 2 \cdot 2^{\binom{2d-1}{d-1}} \exp\left( \frac{\binom    {2d-1}{d-1}}{2^{d}} (2-p)^{d} + L_2 + O\left(\binom{2d-1}{d-1}d^4\left(1-\frac{p}{2}\right)^{3d}\right)\right),
\]
where $L_2 =L_{\cD,2} \coloneqq \sum_{\Gam \in \cC_{\cD,2}} \om(\Gam)$. Due to symmetry, we may assume $\cD=\cL_{d-1}$.
We will explicitly calculate $L_2$. 

\smallskip
\noindent\textit{Case 1:} $\Gam=(\{u\},\{u\})$. The weight of a single vertex $u$ is $\om(\{u\})=\left(\frac{1+e^{-\beta}}{2}\right)^{d}=\left(\frac{2-p}{2}\right)^{d}$. Thus, the contribution is
\begin{equation}\label{eqn:mid_cluster2_1}
    -\frac{1}{2}\binom{2d-1}{d-1}\left(\frac{2-p}{2}\right)^{2d}.
\end{equation}
\noindent\textit{Case 2:} $\Gam=(\{u\},\{v\})$ with $\{u,v\}$ $2$-linked and $u\neq v$. As before, this boils down to computing the size of $N^2(u)$. We have $d$ choices to turn a $0$ into a $1$, and $d-1$ choices to turn a $1$ into a $0$ (as we may not choose the same coordinate), i.e., $|N^2(u)|=(d-1)d$. Hence the total contribution of these clusters is
\begin{equation}\label{eqn:mid_cluster2_2}
    -\frac{1}{2}\sum_{u\in\cE}\sum_{v\in N^2(u)}\om(\{u\})\om(\{v\})=-\frac{1}{2}\binom{2d-1}{d-1}(d-1)d\left(\frac{2-p}{2}\right)^{2d}.
\end{equation}
\noindent\textit{Case 3:} $\Gam=(\{u,v\})$ with $\{u,v\}$ $2$-linked ($u\neq v$). Then $\operatorname{codeg}(u,v)=1$, $|N(\{u,v\})|=2d-1$, and
$$\om(\{u,v\})=2^{-(2d-1)}(1+e^{-\beta})^{2(d-1)}(1+e^{-2\beta}).$$
Hence, the total contribution is
\begin{equation}\label{eqn:mid_cluster2_3}
    \frac{1}{2}\binom{2d-1}{d-1}(d-1)d\left(\frac{2-p}{2}\right)^{2(d-1)}\frac{1+(1-p)^2}{2}.
\end{equation}

Combining \eqref{eqn:mid_cluster2_1}, \eqref{eqn:mid_cluster2_2} and \eqref{eqn:mid_cluster2_3} we get that
$$L_2=\frac{1}{8}\binom{2d-1}{d-1}\left(\frac{2-p}{2}\right)^{2(d-1)}\left((d-1)dp^2-(2-p)^2\right).$$
Note that $\log_2 |V(G)| \leq 2d$. Hence, by \Cref{rem:sharp}, for any $p> 2-2^{1/3} \approx 0.74$ these two terms give a sharp result.

\section{Growing Even Tori} \label{sec:tori}

Let $m\geq 6$ be an even integer and let $G=\square_{i=1}^t C_{m}=\mathbb{Z}_{m}^t$ be the even torus. Note that $G$ is a $d$-regular $n$-vertex graph  with $n \coloneqq m^t$ and $d \coloneqq 2t$ and its maximum codegree is $2$. For a constant $m$, it is easy to verify by \Cref{l:isoperimetry} that $G$ satisfies \hyperref[def:property-i]{Property I} with $C_5=1/2$. Recall that \Cref{thm:expansion} can be applied to any graph with $0<C_5<2$, which translates to $m\sqrt{t}=O(d^{C_5})$ for $0<C_5<2$, i.e., $m=O(t^{c})$ for any constant $0<c<3/2$. However, the size constraint in terms of the regularity  $\log n= O(d)$ in \hyperref[def:property-i]{Property I} translates to $\log m = O(1)$, which prevents $m$ from growing with $t$. Carefully tracking the part of the Koteck\'y--Preiss condition which uses $\log n=O(d)$, one may formulate an appropriate modification of the vertex-isoperimetric conditions which is free of an upper bound on the size of the graph. A similar modification for the hard-core model (i.e., for {\em unpercolated} graphs) can be found in \cite[Theorem 1.1]{CEGK25}.

\begin{definition}[Property II] \label{def:property-ii}
\label{def:propii}
Fix positive real constants $C_1, C_4, C_5$ with $C_5 < 2$. 
Let $G$ be a $d$-regular $n$-vertex bipartite graph  with bipartition $(\cO, \cE)$. We say $G$ satisfies {\bf Property II} if the following hold.
    \begin{itemize}
        \item[\upshape{\textbf{IIa.}}] For any $X \subseteq \cO$ or $X \subseteq \cE$, we have
        \begin{enumerate}
          \item[\upshape{(1)}] $|N(X)| \geq \sqrt{d} |X|$ if $|X| \leq d^3 \log n$,
          \item[\upshape{(2)}] $|N(X)| \geq (1+\frac{C_4}{d^{C_5}})|X|$ if $|X| \leq \frac{3}{8}n$.
      \end{enumerate}
      \item[\upshape{\textbf{IIb.}}] The maximum codegree of $G$ is at most $C_1$.
      \item[\upshape{\textbf{IIc.}}] The number of vertices satisfies $n=\om(d^{6})$.
    \end{itemize}  
\end{definition}

\begin{theorem}\label{thm:numberindsets2}
    \Cref{thm:expansion} holds also  for any $d$-regular $n$-vertex bipartite graph $G$ that satisfies \hyperref[def:property-ii]{Property II}.
\end{theorem}

Indeed, the proof of \Cref{thm:expansion} uses the vertex-isoperimetric inequalities in \hyperref[def:property-i]{Property I} mostly for the verification of the Koteck\'y--Preiss condition. Here we will mention only the main changes necessary for ease of presentation.
 In the proof of \Cref{c:KoteckyPreiss}, one has to adjust the regimes of the cases accordingly. This goes hand in hand with changing the regime for the functions $g, \tilde{g}$. In particular, it suffices to take the first regime to cover any $|A| \leq d/\log \log d$. Note that the codegree condition in \hyperref[def:property-ii]{Property II} implies that any set $A \subseteq \cO$ or $A \subseteq \cE$ automatically satisfies $|N(A)| \geq (d-C_1|A|)|A|$. Thus the proof remains essentially unchanged for small sets.
 For the regime where $d/\log \log d < \ell \leq d^3 \log n$, set $g(\ell)= \frac{\sqrt{d} \ell}{2} \log \ta$ (this corresponds to $C_2=\sqrt{d}$). Repeating the calculation in the second case of the proof of \Cref{c:KoteckyPreiss} for this regime carefully, one finds that expansion of $\sqrt{d}$ is sufficient to prove \eqref{eq:Case2rootd}.
 In the last regime the calculation remains basically unchanged, since the expansion property remained unchanged in this regime. Note that the only one place that used the size constraint in terms of the regularity $\log n=O(d)$ from \hyperref[def:property-i]{Property I} is the last inequality in \eqref{eq:case3result}. Considering sets $A$ of size $a\coloneqq |A| \geq d^3 \log n$ yields the same bound without this size restriction, keeping the logarithmic term when using the bound on $a$ in \Cref{eq:pluginalpha}.

 The only place that uses the vertex-expansion properties in \hyperref[def:property-i]{Property I} is our container lemmas, where two properties are required. The first one requires that every set $X \subseteq N(y)$ for some $y \in V(G)$ with $|X| > d/2$ expands by a factor of $d$, whereas the second one requires for expansion of large sets by a factor of $\left(1+ \frac{C_4}{d^{C_5}} \right)$.
 While the latter requirement remained unchanged, below we prove that the former is indeed implied by the codegree condition in \hyperref[def:property-ii]{Property II}.

 \begin{lemma}
     Let $G$ be a $d$-regular bipartite graph, and let $X \subseteq N(y)$ for some $y \in V(G)$ with $|X|>d/2$. If the maximum codegree of $G$ is at most $C_1$, then $|N(X)| \geq \frac{d}{C_1}|X|$.
 \end{lemma}
 \begin{proof}
     Let $v \in N(X)$ and note that $|N(v) \cap X| \leq C_1$, since otherwise the codegree of $v$ and $y$ would exceed $C_1$. In other words, each vertex $v \in N(X)$ sends at most $C_1$ edges to $X$. We have $|E(X, N(X))| = d |X|$ as $G$ is $d$-regular, and thus $|N(X)|\geq \frac{d}{C_1}|X|$.
 \end{proof}

 Now we prove that the even torus $\mathbb{Z}_m^t$ satisfies \hyperref[def:property-ii]{Property II} when $m=O(t^c)$ for $c<3/2$.

\begin{lemma}
Let $G=\mathbb{Z}_m^t$ be the even torus with $m=O(t^c)$ for some $0<c<3/2$. Then $G$ satisfies \hyperref[def:property-ii]{Property II}.    
\end{lemma}
\begin{proof}
    Recall that $G$ has $n=m^t$ vertices and degree $d=2t$, in particular $n=\om(d^{6})$. Furthermore, note the maximum codegree of $G$ is $2$. 
    
    Let $\{\cO, \cE\}$ denote the bipartition classes of $G$. 
    We may assume by symmetry that $X \subseteq \cE$. First consider $|X| \leq d^3 \log n$ and note that since $d^3 \log n = (2t)^3 t \log m = O(t^4 \log m)$ the set $X$ has polynomial size in $t$. Therefore, the second statement of \Cref{l:isoperimetry} implies that $X$ expands by a factor of $\Omega(d)$, so in particular it expands by at least $\sqrt{d}$.
    
    Now consider $|X| \leq \frac{3}{8}n$ and note that $m \sqrt{t}= O(t^{c+1/2})$. Thus the third statement of \Cref{l:isoperimetry} implies that there are constants $C_4>0$ and $C_5 =c+1/2$ such that $|N(X)| \geq (1+\frac{C_4}{d^{C_5}})|X|$ if $|X| \leq \frac{3}{8}n$. By assumption $c<3/2$, so that $C_5 < 2$, finishing the proof of the claim.
 \end{proof}

Now we proceed to computing the number of independent sets in the percolated even torus. By \Cref{thm:numberindsets2}, taking $\lambda=1$ and $k=2$ in  \Cref{thm:expansion} we have that as $t\to \infty$,
\begin{align}
\Ex{i(G_p)}= 2 \cdot 2^{m^t/2} \exp\left( \frac{m^t}{2^{2t+1}} (2-p)^{2t} + L_2 + O\left(m^t t^4\left(1-\frac{p}{2}\right)^{6t}\right)\right),\label{Zmt}    
\end{align}
where $L_2 =L_{\cD,2} \coloneqq \sum_{\Gam \in \cC_{\cD,2}} \om(\Gam)$. By symmetry, we may assume that $\cD=\cE$.

We explicitly calculate $L_2$.   There are four cases for clusters of size $2$.

\smallskip
\noindent\textit{Case 1:} $\Gam=(\{u\},\{u\})$. The incompatibility graph is a single edge and its Ursell function is $-1/2$, so the contribution of these clusters is
\begin{equation*}
    -\frac{1}{2}\sum_{u\in\cE} \om(\{u\})^2.
\end{equation*}
It is straightforward to compute that for a single vertex $u$, $\om(\{u\})=\left(\frac{1+e^{-\beta}}{2}\right)^d=\left(\frac{2-p}{2}\right)^{2t}$. Thus, the contribution is
\begin{equation}\label{eqn:tori_cluster2_1}
    -\frac{1}{2}\cdot\frac{m^t}{2}\left(\frac{2-p}{2}\right)^{4t}.
\end{equation}
\noindent\textit{Case 2:} $\Gam=(\{u\},\{v\})$ with $\{u,v\}$ $2$-linked and $u\neq v$. The incompatibility graph is a single edge and its Ursell function is $-1/2$. The ways to create such a cluster is to choose some $u$ and then $v\in N^2(u)$. By the coordinate structure of the product graph, it is easy to see that $|N^2(u)|=2t+4\binom{t}{2}=2t^2$, by choosing one or two coordinates of $u$ that will be changed, each of which gives two options. Hence the total contribution of these clusters is
\begin{equation}\label{eqn:tori_cluster2_2}
    -\frac{1}{2}\sum_{u\in\cE}\sum_{v\in N^2(u)}\om(\{u\})\om(\{v\})=-\frac{1}{2}\cdot\frac{m^t}{2}2t^2\left(\frac{2-p}{2}\right)^{4t}.
\end{equation}
\noindent\textit{Case 3:} $\Gam=(\{u,v\})$ with $\{u,v\}$ $2$-linked ($u\neq v$) and $v$ has one different coordinate than $u$. Then $\operatorname{codeg}(u,v)=1$ and
$$\om(\{u,v\})=2^{-(4t-1)}\sum_{B\subseteq N({u,v})}e^{-\beta|E(\{u,v\},B)|}=2^{-(4t-1)}(1+e^{-\beta})^{4t-2}(1+e^{-2\beta}).$$
For the last equality we used that the common neighbor contributes none or two edges, for a factor of $(1+e^{-2\beta})$, while the remaining $4t-2$ neighbors contribute none or one edge, for a factor $(1+e^{-\beta})$. Since the Ursell function is $1$ and for each choice of $u$ we have $2t$ choices for $v$, the total contribution is
\begin{equation}\label{eqn:tori_cluster2_3}
    \frac{1}{2}\cdot\frac{m^t}{2}2t\left(\frac{2-p}{2}\right)^{4t-2}\frac{1+(1-p)^2}{2},
\end{equation}
where the factor $1/2$ in the front is because $\{u,v\}$ is unordered. 

\noindent\textit{Case 4:} $\Gam=(\{u,v\})$ with $\{u,v\}$ $2$-linked ($u\neq v$) and $v$ has two different coordinates than $u$. Then $\operatorname{codeg}(u,v)=2$ and
$$\om(\{u,v\})=2^{-(4t-2)}\sum_{B\subseteq N({u,v})}e^{-\beta|E(\{u,v\},B)|}=2^{-(4t-2)}(1+e^{-\beta})^{4t-4}(1+e^{-2\beta})^2,$$
where the last equality is again implied by the factors contributed by the two common neighbors and the remaining $4t-4$ neighbors. Since the Ursell function is $1$ and for each choice of $u$ we have $4\binom{t}{2}$ choices for $v$, the total contribution is
\begin{equation}\label{eqn:tori_cluster2_4}
    \frac{1}{2}\cdot\frac{m^t}{2}2t(t-1)\left(\frac{2-p}{2}\right)^{4t-4}\left(\frac{1+(1-p)^2}{2}\right)^2,
\end{equation}
where the factor $1/2$ in front is because $\{u,v\}$ is unordered.

Combining \eqref{eqn:tori_cluster2_1}, \eqref{eqn:tori_cluster2_2}, \eqref{eqn:tori_cluster2_3}, and \eqref{eqn:tori_cluster2_4}, we get
\begin{align}
    L_2=\frac{m^t}{4}\left(\frac{2-p}{2}\right)^{4t-4}\Bigg[-(2t^2+1)\left(\frac{2-p}{2}\right)^4&+2t\left(\frac{2-p}{2}\right)^2\frac{1+(1-p)^2}{2}\nonumber\\  
    &+2t(t-1)\left(\frac{1+(1-p)^2}{2}\right)^2\Bigg].\label{ZmtL2}
\end{align}
(We note that there is a discrepancy between our computation for $L_2$ when $p=1$ compared to the corresponding result in \cite{jenssen2023homomorphisms}, and one of the authors agrees with our computation \cite{jenssen2025personal}.)

We conclude this section with some remarks. For the unpercolated case, taking $p=1$ in  \eqref{Zmt} and \eqref{ZmtL2} we have that as $t\to \infty$,
\begin{equation}\label{eqn:tori}
i(\mathbb{Z}_m^t)=2\cdot 2^{m^t/2}\exp\left(\frac{m^t}{2^{2t+1}}+\frac{m^t}{2^{4t+2}}(6t^2-4t-1)+O\left(\frac{m^t t^4}{2^{6n}}\right)\right). 
\end{equation} 
From \Cref{rem:sharp}, two terms suffice for a sharp result as long as $m \le 62$. On the other hand, for $m$ growing with $t$ as $t \rightarrow \infty$, we do not get a sharp $(1+o(1))$ bound on the number of independent sets in the sense of \Cref{rem:sharp}, but we still obtain the estimate
\begin{equation}\label{conj:tori_simple}
   i(\mathbb{Z}_m^t)=2\cdot2^{m^t/2}\exp\left(\frac{m^t}{2^{2t+1}}(1+o_t(1))\right)
\end{equation}
as long as $m=O(t^c)$ for some $0<c<3/2$, as $t\to \infty$. In this form (but for $t$ fixed and $m \rightarrow \infty$), a much stronger result was conjectured by Jenssen and Keevash \cite{jenssen2023homomorphisms} and recently confirmed by Peled and Spinka \cite{peled2020long}. 
In fact, Jenssen and Keevash \cite{jenssen2023homomorphisms} studied homomorphisms from $\mathbb{Z}_m^t$ to any fixed graph $H$ (for fixed $m$ and $t\rightarrow\infty$) equipped with a \textit{fixed} vector of activities on the vertices. This is also done via the cluster expansion method. In the same paper, they conjectured \cite[Section 19]{jenssen2023homomorphisms} that if $t$ is fixed and $m\rightarrow\infty$, then the first term of the, possibly divergent, cluster expansion divided by $m^t$ should give the correct term in the exponential up to a $(1+o_t(1))$ factor. Peled and Spinka \cite{peled2020long} proved this conjecture, with the $o_t(1)$ term being $e^{-\Omega(t)}$, for any target graph with fixed vertex and edge weighting under certain symmetry conditions. On the other hand, our results show that this phenomenon persists, even when $m$ is growing with $t$, for the wider range of parameters satisfying $\lam(1-e^{-\beta})=\tilde{\Om}(t^{-1/2})$ (with $\beta=\infty$ being the hard-core model). Furthermore, our method proves it for any constant number of terms of the cluster expansion.  As there is no clear reason why such a phenomenon should only manifest only for ranges of parameters $m,t,\lambda, \beta$ covered so far, we are led to the following conjecture.


\begin{conjecture}
    For any fixed $k\in\mathbb{N}$ there exist functions $f_k(t), g_k(t)$ with $f_k(t)=\tilde{\Omega}(t^{-1/2})$ and $g_k(t)=o_t(1)$ such that the following holds. For any $\lam>0$, $\beta\in (0,\infty]$ satisfying $\lam(1-e^{-\beta}) \geq f_k(t)$ and for all even $m$,
    $$Z_{\mathbb{Z}_m^t}(\lam,\beta)=2(1+\lam)^{m^t/2}\exp\left(\frac{m^t}{2}\lam\left(\frac{1+\lam e^{-\beta}}{1+\lam}\right)^{2t}+L_2+\dots+L_k(1+g_k(t))\right),$$
    where $L_j\coloneqq\sum_{\Gam \in \cC_{\cE,j}} \om(\Gam).$
\end{conjecture}

\section{Discussion}

In this paper, we extended the work of Kronenberg and Spinka \cite{kronenberg2022independent} on the expected number of independent sets in the percolated hypercube to any percolated graph from a wide class of regular and bipartite graphs with some vertex-expansion properties.
Continuing along these lines, it would be interesting to determine what minimal assumptions are required to guarantee the convergence of the cluster expansion and obtain a result analogous to \Cref{thm:expansion}.

The proof methods, especially the cluster expansion method, rely substantially on the assumption that the graph is bipartite and exhibits some vertex-expansion. In terms of the isoperimetry required, we extended the regime allowing the vertex-expansion of order $(1+1/d^{C_5})$ for any $0<C_5 <2$, in comparison to previous work that requires $0<C_5 \le 1$. 
In particular, this extension of the range of $C_5$ covers examples such as even tori of growing side-length, which exhibit worse expansion than, for example, the hypercube. It would be interesting to find more examples of graphs that satisfy \hyperref[def:property-i]{Property I} with $1<C_5<2$.

All the examples considered in this paper are {\em symmetric}, in the sense that there exists an automorphism mapping the even side $\cE$ to the odd side $\cO$, and vice versa. Hence, the sum over both sides can effectively be replaced by a factor of two in those examples in Sections \ref{sec:examples} and \ref{sec:tori}. On the other hand, our results take into account the possible asymmetry between the bipartition sides.   
It would be interesting to find a natural example of a regular bipartite graph $G$ satisfying \hyperref[def:property-i]{Property I} with {\em asymmetric bipartition sides}, where the clusters on the odd side actually give a different contribution than those on the even side.

The range of $\lam p$ for which our results hold is of special interest. For $0<C_5 \leq 1$, \Cref{thm:expansion} covers any $\lam p \geq \frac{C_0\log^{3/2} d}{d^{1/2}}$, whereas for worse expansion we need to pay a price in the exponent of $d$, more precisely, it allows any $\lam p \geq \frac{C_0\log^{3/2} d}{d^{1-C_5/2}}$ if $1\leq C_5< 2$. 
It is known that for $\lam = O(1/d)$ the typical independent set drawn via the hard-core model in $Q^d$ (i.e., when $p=1$) is quite unstructured \cite{galvin2011threshold, weitz2006counting}. On the contrary, for larger $\lam$, independent sets lie mostly on one side of the bipartition, and thus Galvin \cite{galvin2011threshold} conjectured that the structure of independent sets exhibits a phase transition around $\lambda= \tilde{\Omega}(1/d)$. From a counting point of view (i.e., when $\lam=1$), this becomes a question about $p$ in the percolated graph $G_p$. In recent work on the percolated hypercube $Q^d_p$, Chowdhury,  Ganguly, and Winstein \cite{chowdhury2024gaussian} showed  that for $p>2/3$ their uniform sampler, which is precisely the hard-core model for $\lam=1$, is with high probability very close to the polymer model. In light of Galvin's conjecture \cite{galvin2011threshold} and our main result (\Cref{thm:expansion} for $0<C_5 \leq 1$), one might expect a  phase transition in the structure of independent sets in $G_p$ for some $p$ between $O(1/d)$ and $\Omega\left(\frac{\log^{3/2} d}{d^{1/2}}\right)$, leading to the following question.
\begin{question}
    Let $G$ be a $d$-regular bipartite graph satisfying \hyperref[def:property-i]{Property I} with $0<C_5 \leq 1$.
    Does there exist a phase transition for the typical structure of independent sets in $G_p$? If so, what is the critical probability $p_*$ for it? Is it of the form  $p_*=\tilde{\Theta}(1/d^c)$ for some constant $c>0$?
\end{question}

\section*{Acknowledgments}
We would like to thank the referees for their helpful comments and suggestions. This research was supported in part by the Austrian Science Fund (FWF) [10.55776/F1002].
For open access purposes, the author has applied a CC BY public copyright license to any author accepted manuscript version arising from this submission.

\printbibliography

\newpage

\appendix

\section{Weights of vertex sets via partition functions and entropy} \label{sec:appendix}

In this appendix, we detail some of the main ingredients that go into bounding weights of large polymers, as well as the non-configurations. These are entropy bounds due to Peled and Spinka \cite{peled2020long} who used them to study discrete spin systems on regular bipartite graphs, generalizing previous entropy bounds by Kahn, Galvin, and Tetali \cite{GaKa2004, galvin2004weighted, galvin2006bounding, kahn2001entropy}. We will eventually specialize to the relevant case of the antiferromagnetic Ising model, as was done by Kronenberg and Spinka \cite{kronenberg2022independent}.

First we describe the main lemma of Peled and Spinka, Lemma \ref{lemma:Peled-Spinka} below, which is used to derive Lemmas \ref{lemma:entropy-1}, \ref{lemma:entropy-2}, and \ref{lemma:entropy-3}. For the setup, we are given a $d$-regular bipartite graph $G$ with bipartition $(\cO, \cE)$, and a spin graph $\cS$ with finite vertex set $\mathbb{S}$, positive vertex weights $\lam_i$, and nonnegative edge weights $\lam_{i,j}$. For a subset $I \subseteq \mathbb{S}$, we denote $\lam_I \coloneqq \sum_{i \in I} \lam_i$. 

Fix a vertex subset $S \subseteq V(G)$. A \textit{configuration} of $S$ is any map $f : S \rightarrow \mathbb{S}$. The \textit{internal boundary} of $S$ is $\partial_{\bullet}S \coloneqq N(S^c) \setminus S^c$, i.e., the set of vertices in $S$ that are adjacent to some vertex not in $S$. The \textit{external boundary} of $S$ is $\partial_{\circ}S \coloneqq N(S) \setminus S$, i.e., the set of vertices not in $S$ that are adjacent to some vertex in $S$. We say that $S$ is \textit{boundary-even} if $\partial_{\bullet}S \subseteq \cE$, or equivalently $N(S \cap \cO) \subseteq S$. Likewise, $S$ is \textit{boundary-odd} if $\partial_{\bullet}S \subseteq \cO$, or equivalently $N(S \cap \cE) \subseteq S$. We denote the \textit{edge-boundary} as $\partial S \coloneqq E(S,S^c)$.

For a configuration $f : S \rightarrow \mathbb{S}$, we define its weight as
\begin{align*}
    \tilde{\om}(f) \coloneqq \prod_{v \in S} \lam_{f(v)} \cdot \prod_{\{v,u\} \in E(S)} \lam_{f(v),f(u)}.
\end{align*}
For a family $\cF$ of configurations, we write $\tilde{\om}(\cF) = \sum_{f \in \cF} \tilde{\om}(f)$. For a collection $\Psi$ of maps $\psi : [d] \rightarrow \mathbb{S}$ and a subset $I \subseteq \mathbb{S}$, we define a partition function as follows:
\begin{align*}
    Z(\Psi, I) \coloneqq \sum_{\psi \in \Psi} \left( \prod_{j = 1}^d \lam_{\psi(j)} \right) \left( \sum_{i \in I} \lam_i \prod_{j = 1}^d \lam_{i, \psi(j)} \right)^d.
\end{align*}
The definition of $Z(\Psi, I)$ arises from the proof of Lemma \ref{lemma:Peled-Spinka}, but the following provides some intuition for it from the perspective of graph homomorphisms. Consider the special case $\lam_i = 1$ and $\lam_{i,j} \in \{0,1\}$ for all $i,j \in \mathbb{S}$, that is, the spin graph is some unweighted graph $\cS$. Then
\begin{align*}
    Z(\mathbb{S}^{[d]}, \mathbb{S}) = \sum_{\psi \in \mathbb{S}^{[d]}} \left( \sum_{i \in \mathbb{S}} \mathbbm{1}_{i \sim \psi(j)\, \forall j \in [d]} \right)^d.
\end{align*}
In the inner sum, we assign spin $\psi(j)$ to each vertex $j \in [d]$ on the right side of $K_{1,d}$, and for the single vertex on the left side of $K_{1,d}$, we count the number of ways to assign a spin $i \in \mathbb{S}$ to it such that $i$ and each $\psi(j)$ are adjacent in $\cS$. That is, the inner sum counts the number of homomorphisms from $K_{1,d}$ to $\cS$ whose restriction to the right side is $\psi$. Raising to the power $d$ counts the number of such homomorphisms from $K_{d,d}$ to $\cS$, and summing up over all $\psi \in \mathbb{S}^{[d]}$ means that $Z(\mathbb{S}^{[d]}, \mathbb{S})$ counts the number of homomorphisms from $K_{d,d}$ to $\cS$. In summary, the function $Z(\Psi, I)$ acts as a partial partition function on $K_{d,d}$, where $\Psi$ and $I$ provide information about the two parts of $K_{d,d}$.

The following is the promised entropy lemma of Peled and Spinka \cite{peled2020long}. For a configuration $f : S \rightarrow \mathbb{S}$ and a subset $X \subseteq V(G)$, let $f_X$ denote $f$ restricted to $X \cap S$. Note that if $S$ is boundary-even and $v \in S \cap \cO$, then $N(v) \subseteq S$. 

\begin{lemma}[Peled--Spinka, Lemma 7.3 \cite{peled2020long}] \label{lemma:Peled-Spinka}
Let $S \subseteq V(G)$ be finite and boundary-even, and let $(\mathbb{S}_u : u \in \partial_{\bullet}S)$ be a collection of subsets of $\mathbb{S}$. Let $\cF$ be a family of configurations $f : S \rightarrow \mathbb{S}$ such that $f(u) \in \mathbb{S}_u$ for all $u \in \partial_{\bullet}S$, and consider $f \in \cF$ chosen with probability proportional to its weight $\tilde{\om}(f)$. To each vertex $v \in S \cap \cO$, we assign a random variable $X_v$ on $\cF$ that is measurable with respect to the random variable $f_{N(v)}$. Then
\begin{align*}
    \tilde{\om}(\cF) \le \prod_{v \in S \cap \cO} \prod_{x} \left[ \frac{Z(\Psi_{v,x}, I_{v,x})}{\mathbb{P}(X_v = x)} \right]^{\frac{1}{d} \mathbb{P}(X_v = x)} \cdot \prod_{u \in \partial_{\bullet}S} (\lam_{\mathbb{S}_u})^{\frac{1}{d}|\partial u \cap \partial S|},
\end{align*}
where the second product is over the support of $X_v$, $\Psi_{v,x} = \{ \psi \in \mathbb{S}^{N(v)} : \mathbb{P}(f_{N(v)} = \psi \mid X_v = x) > 0 \}$, and $I_{v,x} = \{ i \in \mathbb{S} : \mathbb{P}(f(v) = i \mid X_v = x) > 0 \}$.
\end{lemma}

Now we specialize Lemma \ref{lemma:Peled-Spinka} to our setting of the antiferromagnetic Ising model. We take the spin set to be $\mathbb{S} = \{0,1\}$, vertex activity parameters to be $\lam_1 = \lam$ and $\lam_0 = 1$, and edge activity parameters to be $\lam_{0,0} = \lam_{0,1} = \lam_{1,0} = 1$ and $\lam_{1,1} = e^{-\beta}$, where $\beta > 0$. 

In each of the special cases below, we start with a \textit{boundary-odd} vertex set $T \subseteq V(G)$, and then we take $S = T \cup N(T)$ and $\mathbb{S}_u = \{0\}$ for all $u \in \partial_{\bullet}S$. Note that $\partial_{\bullet}S \subseteq N(T) \setminus T \subseteq \cE$, so in particular $S$ is boundary-even, and also note that $S \cap \cO = T \cap \cO$. Configurations $f : S \rightarrow \mathbb{S}$ which satisfy $f(u) \in \mathbb{S}_u$ (i.e., $f(u) = 0$) for all $u \in \partial_{\bullet}S$ can naturally be identified with subsets of $T$, namely their supports $F = \{v \in T : f(v) = 1\}$. Thus, we can view our family $\cF$, consisting of configurations on $S$ satisfying the appropriate boundary condition, as some collection of subsets of $T$. Likewise, we can view our family $\Psi$ of maps from $[d]$ to $\mathbb{S}$ as a collection of subsets $\psi$ of $[d]$. With these identifications, for $F \subseteq T$ we write
\begin{align*}
    \tilde{\om}(F) \coloneqq \lam^{|F|} e^{-\beta |E(F)|},
\end{align*}
and for $\Psi$ a collection of subsets of $[d]$ we write
\begin{align*}
    Z(\Psi) \coloneqq Z(\Psi, \{0,1\}) = \sum_{\psi \in \Psi} \lam^{|\psi|} \left( 1 + \lam e^{-\beta |\psi|} \right)^d.
\end{align*}
We will implicitly use the fact that $Z(\Psi, \{0\}), Z(\Psi, \{1\})$ are both at most $Z(\Psi, \{0,1\})$ in the proofs below. Finally, note that $\lam_{\mathbb{S}_u} = \lam_0 = 1$ for all $u \in \partial_{\bullet}S$, so the rightmost product in Lemma \ref{lemma:Peled-Spinka} is always $1$ in this setting.

\begin{lemma}[Kronenberg--Spinka, Lemma 4.11] \label{lemma:entropy-1-restatement}
Let $T \subseteq V(G)$ be boundary-odd, and let $\cF$ be a family of subsets of $T$. Then
\begin{align*}
    \tilde{\om}(\cF) \le \prod_{v \in T \cap \cO} Z(\Psi_v)^{\frac{1}{d}},
\end{align*}
where $\Psi_v = \{F \cap N(v) : F \in \cF\}$.
\end{lemma}

\begin{proof}
We apply Lemma \ref{lemma:Peled-Spinka} by taking $S = T \cup N(T)$, $\mathbb{S}_u = \{0\}$ for all $u \in \partial_{\bullet}S$, and the random variable $X_v$ to be some constant, say $0$, for all $v \in S \cap \cO = T \cap \cO$. Note that $\Psi_v = \Psi_{v,0} = \{ F \cap N(v) : F \in \cF\} = \Psi_{v,0}$, that $I_{v,0} \subseteq \{0,1\}$, and that $\mathbb{P}(X_v = 0) = 1$. From Lemma \ref{lemma:Peled-Spinka}, we deduce the stated inequality.
\end{proof}

For the next two lemmas, similar to entropy conventions, we use the convention that $(1/p)^p$ equals $1$ whenever $p = 0$. This is just an alternative to restricting products to the support of the associated random variables, as was done in Lemma \ref{lemma:Peled-Spinka}.

\begin{lemma}[Kronenberg--Spinka, Lemma 4.12] \label{lemma:entropy-2-restatement}
Let $T \subseteq V(G)$ be boundary-odd, and let $\cF$ be a family of subsets of $T$, such that every $F \in \cF$ contains no isolated vertex in $F \cap \cO$. Then
\begin{align*}
    \tilde{\om}(\cF) \le \prod_{v \in T \cap \cO} Z(\Psi_v)^{\frac{p_v}{d}} \left( \frac{1}{p_v} \right)^{\frac{p_v}{d}} \left( \frac{1}{1 - p_v} \right)^{\frac{1-p_v}{d}},
\end{align*}
where $\Psi_v = \{F \cap N(v) : F \in \cF, |F \cap N(v)| > 0\}$ and $p_v = \mathbb{P}(|F \cap N(v)| > 0)$ when $F$ is randomly chosen from $\cF$ according to $\tilde{\om}$.
\end{lemma}

\begin{proof}
Like in the previous proof, we apply \ref{lemma:Peled-Spinka} with $S = T \cup N(T)$ and $\mathbb{S}_u = \{0\}$ for all $u \in \partial_{\bullet}S$, but now we take the random variable $X_v$ to be $\mathbbm{1}_{|F \cap N(v)| = 0}$ for all $v \in T \cap \cO$. We have
\begin{align*}
    \Psi_v = \Psi_{v,0} = \{ F \cap N(v) : F \in \cF, |F \cap N(v)| > 0 \}.
\end{align*}
On the other hand, assuming that the event $X_v = 1$ occurs with positive probability, we have 
\begin{align*}
    \Psi_{v,1} &= \{ F \cap N(v) : F \in \cF, |F \cap N(v)| = 0 \} = \{\emptyset\}, \\
    I_{v,1} &= \{|F \cap \{v\}| : F \in \cF, |F \cap N(v)| = 0\} = \{0\},
\end{align*}
where the last equality is because $|F \cap N(v)| = 0$ and $v$ not being isolated in $F$ implies that $v \notin F$. Thus, $Z(\Psi_{v,1}, I_{v,1}) = Z(\{\emptyset\}, \{0\}) = 1$ whenever $\mathbb{P}(X_v = 1) > 0$. Writing $p_v = \mathbb{P}(X_v = 0) = \mathbb{P}(|F \cap N(v)| > 0)$, Lemma \ref{lemma:Peled-Spinka} implies the stated inequality.
\end{proof}

\begin{lemma}[Kronenberg--Spinka, Lemma 4.13] \label{lemma:entropy-3-restatement}
Let $T \subseteq V(G)$ be boundary-odd, let $\cF$ be a family of subsets of $T$, and let $s > 0$. Then
\begin{align*}
    \tilde{\om}(\cF) \le \prod_{v \in T \cap \cO} Z(\Psi_v)^{\frac{p_v}{d}} Z(\Psi_v')^{\frac{p_v'}{d}} (1 + \lam)^{1 - p_v - p_v'} \left( \frac{1}{p_v} \right)^{\frac{p_v}{d}} \left( \frac{1}{p_v'} \right)^{\frac{p_v'}{d}} \left( \frac{1}{1 - p_v - p_v'} \right)^{\frac{1-p_v-p_v'}{d}},
\end{align*}
where $\Psi_v = \{F \cap N(v) : F \in \cF, 1 \le |F \cap N(v)| \le s\}$, $\Psi_v' = \{F \cap N(v) : F \in \cF, |F \cap N(v)| > s\}$, $p_v = \mathbb{P}(1 \le |F \cap N(v)| \le s)$, and $p_v' = \mathbb{P}(|F \cap N(v)| > s)$, when $F$ is randomly chosen from $\cF$ according to $\tilde{\om}$.
\end{lemma}

\begin{proof}
Like in the previous proofs, we apply \ref{lemma:Peled-Spinka} with $S = T \cup N(T)$ and $\mathbb{S}_u = \{0\}$ for all $u \in \partial_{\bullet}S$, but we take the random variable $X_v$ to be $(\mathbbm{1}_{|F \cap N(v)| = 0}, \mathbbm{1}_{|F \cap N(v)| \le s})$ for all $v \in T \cap \cO$. We have
\begin{align*}
    \Psi_v &= \Psi_{v,(0,1)} = \{ F \cap N(v) : F \in \cF, 1 \le |F \cap N(v)| \le s \}, \\
    \Psi_v' &= \Psi_{v,(0,0)} = \{ F \cap N(v) : F \in \cF, |F \cap N(v)| > s \}.
\end{align*}
On the other hand, assuming that the event $X_v = (1,1)$ occurs with positive probability, we have
\begin{align*}
    \Psi_{v,(1,1)} &= \{ F \cap N(v) : F \in \cF, |F \cap N(v)| = 0 \} = \{\emptyset\},
\end{align*}
and $I_{v,(1,1)} \subseteq \{0,1\}$, so that $Z(\Psi_{v,(1,1)}, I_{v,(1,1)}) \le Z(\{\emptyset\}, \{0,1\}) = (1 + \lam)^d$. Writing $p_v = \mathbb{P}(X_v = (0,1)) = \mathbb{P}(1 \le |F \cap N(v)| \le s)$ and $p_v' = \mathbb{P}(X_v = (0,0)) = \mathbb{P}(|F \cap N(v)| > s)$, Lemma \ref{lemma:Peled-Spinka} implies the stated inequality.
\end{proof}

\section{Container lemma for the hard-core model} \label{sec:container-hard-core}

In this appendix, we outline how to modify the proofs of Corollaries \ref{cor:hard-core-corollary-1} and \ref{cor:hard-core-corollary-2} in the hard-core model, i.e., $\beta = \infty$, so that we do not need the assumption that $\lambda$ is allowed to grow arbitrarily as $d \rightarrow \infty$. Recall that these two corollaries were special cases of our main container lemma (\Cref{lem:container_main}). The following are the precise statements we want to show true.

\begin{lemma} \label{lem:hard-core-corollary-1-modified}
For any fixed constants $C_2, C_4, C_5 > 0$ with $C_5 \le 1$, there exist constants $C_0, C > 0$ such that the following holds. Suppose that $\lam \ge \frac{C_0 \log^{3/2} d}{d^{1/2}}$, that $|N(X)| \ge \frac{d^2}{2C_2}$ for all $X \subseteq N(y)$ with $y \in \overline{\cD}$ and $|X| > d/2$, and that $a, b \in \mathbb{N}$ with $b \ge \left( 1 + \frac{C_4}{d^{C_5}} \right)a$. Then
\begin{align*}
    \sum_{A \in \cG_\cD(a,b)} \lam^{|A|} \le |\cD|(1+\lam)^b \exp \left( -\frac{C(b - a)\log^2d}{d} \right).
\end{align*}
\end{lemma}

\begin{lemma} \label{lem:hard-core-corollary-2-modified}
For any fixed constants $C_2, C_4, C_5 > 0$ with $1 \le C_5 < 2$, there exist constants $C_0, C > 0$ such that the following holds. Suppose that $\lam \ge \frac{C_0 \log^{3/2} d}{d^{1-C_5/2}}$, that $|N(X)| \ge \frac{d^2}{2C_2}$ for all $X \subseteq N(y)$ with $y \in \overline{\cD}$ and $|X| > d/2$, and that $a,b \in \mathbb{N}$ with $b \ge \left( 1 + \frac{C_4}{d^{C_5}} \right)a$. Then
\begin{align*}
    \sum_{A \in \cG_\cD(a,b)} \lam^{|A|} \le |\cD|(1+\lam)^b \exp \left( -\frac{C(b - a)\log^2d}{d^{2-C_5}} \right).
\end{align*}
\end{lemma}

The proofs are mostly the same as the proof of \Cref{lem:container_main}, and we will only highlight where modifications are needed. Recall that, after combining with \Cref{lem:galvin_approx_family}, the proof of \Cref{lem:container_main} reduced to proving \Cref{lemma:total-weight-given-approximation}, which bounds the total weight of sets in $\cG_\cD(a,b)$ with a given approximation. In similar fashion, the proofs of Lemmas \ref{lem:hard-core-corollary-1-modified} and \ref{lem:hard-core-corollary-2-modified} will reduce to the following version of \Cref{lemma:total-weight-given-approximation} in the hard-core model. Recall that in the hard-core model, the weight of a set $A \in \cG_\cD(a,b)$ is given by $\om(A) \coloneqq (1+\lambda)^{-b} \lambda^{|A|}$.

\begin{lemma} \label{lemma:total-weight-given-approximation-modified}
For any fixed constants $C_4, C_5 > 0$, there exist constants $C_0, C'' > 0$ such that the following holds. Suppose that $\lam \ge \frac{C_0 \log d}{d}$, that $a,b \in \mathbb{N}$ with $b \ge \left(1 + \frac{C_4}{d^{C_5}}\right)a$, and that $1 \le \psi \le d/2$. Then for any $\psi$-approximation $(F,H)$, we have
\begin{align*}
    \sum_{A \in \cG_\cD(a,b) :\ A \approx (F,H)} \om(A) \le \exp\left( -C''(b-a)\min\left\{\overline{\alpha},\frac{\overline{\alpha}^2}{\log d}\right\} \right),
\end{align*}
\end{lemma}

In the hard-core model, $\overline{\alpha} \coloneqq \log(1+\lambda)$, which can be viewed as taking $\beta \rightarrow \infty$ in the Ising model definition of $\overline{\alpha}$. Originally in \Cref{lemma:total-weight-given-approximation}, we only included the $\overline{a}^2$ term in the exponential of \Cref{lemma:total-weight-given-approximation-modified} because it is always the minimum when $\lambda$ is bounded. In fact, when $\lambda$ grows at most polynomially fast in $d$. Since in the setting of \Cref{lemma:total-weight-given-approximation-modified} we make no assumptions on the growth of $\lambda$, we have to include both terms. 

Combining Lemmas \ref{lem:galvin_approx_family} and \ref{lemma:total-weight-given-approximation-modified}, we have that
\begin{align*}
    \sum_{A \in \cG_{\cD}(a,b)} \om(A) \le |\cD| \exp \left( \frac{C'b\log^2 d}{d^2} + \frac{C'(b - a)\log^2 d}{d} -C''(b-a)\min\left\{\overline{\alpha},\frac{\overline{\alpha}^2}{\log d}\right\} \right).
\end{align*}

When $\overline{\alpha}\leq \log d$, then the second term of the minimum is used. It is immediate to see that it dominates the first two summands of the exponential when $\overline{\alpha} \ge \frac{C_0' \log^{3/2} d}{d^{1/2}}$ (if $0 < C_5 \le 1$) or $\overline{\alpha} \ge \frac{C_0' \log^{3/2} d}{d^{1-C_5/2}}$ (if $1 \le C_5 < 2$). These lower bounds on $\overline{\alpha}$ hold because, when $d$ is sufficiently large, $\lambda \ge \frac{C_0 \log^{3/2} d}{d^{1/2}}$ implies that $\overline{\alpha} = \log(1 + \lambda) \ge \frac{C_0 \log^{3/2} d}{2d^{1/2}}$, and likewise $\lambda \ge \frac{C_0 \log^{3/2} d}{d^{1-C_5/2}}$ implies that $\overline{\alpha} = \log(1 + \lambda) \ge \frac{C_0 \log^{3/2} d}{2d^{1-C_5/2}}$. Plugging in the lower bounds on $\overline{\alpha}$, we deduce the statements of Lemmas \ref{lem:hard-core-corollary-1-modified} and \ref{lem:hard-core-corollary-2-modified}. When $\overline{\alpha} > \log d$, then $(b-a)\overline{\alpha}\geq (b-a)\log d$ and the lemmas follow trivially.

Recall that we bounded weights of polymers with a given approximation using entropy bounds from \Cref{sec:entropy_tools}, namely Lemmas \ref{lemma:entropy-1} and \ref{lemma:entropy-2}. These lemmas give upper bounds on the weights $\tilde{\omega}(I)$ of vertex sets $I \subseteq V(G)$ in terms of products of partition functions $Z(\Psi)$ on certain collections $\Psi$ of subsets of $[d]$. In the hard-core case, the function $Z(\Psi)$ is simply given by 
\begin{align*}
    Z(\Psi) = \sum_{\psi \in \Psi} \lambda^{|\psi|}.
\end{align*}
To apply  Lemmas \ref{lemma:entropy-1} and \ref{lemma:entropy-2}, we also required an effective general upper bound on $Z(\Psi)$ given by \Cref{lemma:bound-on-Z}, and this depended on the parameters $\overline{\alpha}$ and $\ell_{\Psi} = |\{ i \in [d] : i \notin \psi \text{ for all } \psi \in \Psi \}|$. We needed the bounds $\lambda \le \lambda_0$ and $\lambda(1 - e^{-\beta})^2 \ge \frac{C_0 \log d}{d}$ to be able to apply \Cref{lemma:bound-on-Z}. However, in the hard-core setting, the corresponding upper bound on $Z(\Psi)$ is simply
\begin{align} \label{upper-bound-Z-hard-core}
    Z(\Psi) \le (1+\lambda)^{d - \ell_{\psi}} = (1+\lambda)^d \exp(-\overline{\alpha} \ell_{\Psi}),
\end{align}
and this is straightforward from definitions and does not require any assumptions on $\lambda$ other than positivity.

Now, using (\ref{upper-bound-Z-hard-core}) in place of \Cref{lemma:bound-on-Z} and mostly following the same proofs, one may derive the following versions of Lemmas \ref{lemma:b-f-small} and \ref{lemma:b-f-large} for the hard-core setting.

\begin{lemma} \label{lemma:b-f-small-modified}
Suppose that $a, b \in \mathbb{N}$ with $b \ge a$, and that $1 \le \psi \le d/2$. Then for any $\psi$-approximation $(F,H)$,
\begin{align*}
    \sum_{A \in \cG_\cD(a,b) :\ A \approx (F,H)} \om(A) \le \binom{2bd}{b - |F|} \exp{\left( -(b - a) \overline{\alpha} \right)}.
\end{align*}
\end{lemma}

\begin{lemma} \label{lemma:b-f-large-modified}
For any fixed $C > 0$, there exists $C_0 > 0$ such that the following holds. Suppose that $\lam \ge \frac{C_0 \log d}{d}$, that $a, b \in \mathbb{N}$ with $b \ge a$, and that $1 \le \psi \le d/2$. Then for any $\psi$-approximation $(F,H)$,
\begin{align*}
    \sum_{A \in \cG_\cD(a,b) :\ A \approx (F,H)} \om(A) \le \exp{\left(-\frac{1}{2} (b - |F|) \overline{\alpha} + \frac{b}{d^C} \right)}.
\end{align*}
\end{lemma}

We note that \Cref{lemma:b-f-small-modified} can also be proved by elementary means, by repeating part of the argument in the proof of \Cref{lemma:b-f-small}. Indeed, the inequality in the above statement can be rearranged to
\begin{align*}
    \sum_{A \in \cG_\cD(a,b) :\ A \approx (F,H)} \lambda^{|A|} \le \binom{2bd}{b - |F|} (1+\lambda)^a.
\end{align*}
Suppose that $A \in \cG_\cD(a,b)$ and $A \approx (F,H)$. We view $A$ as some subset of $[A]$ (where $|[A]| = a$), and we bound the number of possible choices of $[A]$. Since $F \subseteq N(A) \subseteq N(H)$ and $N(A)$ determines $[A]$, the number of possible choices of $[A]$ is at most the number of subsets of $N(H)$ containing $F$. Using that  $|N(H)| \le 2bd$ as observed in \eqref{NH-upper-bound}, the latter number is at most $\binom{2bd}{b - |F|}$, and the above inequality follows.

For proving \Cref{lemma:b-f-large-modified}, a minor modification is needed in the proof of Lemma \ref{lemma:b-f-large}. Namely, we implicitly used that $\lambda \le \lambda_0$ together with $\lambda(1 - e^{-\beta}) \ge \frac{C_0 \log d}{d}$ in order to obtain and apply the inequality $\overline{\alpha} \ge \frac{C_0 \log d}{(1+\lambda_0)d}$ in that proof. However, in the hard-core setting, we may alternatively use that $\lambda \ge \frac{C_0 \log d}{d}$ implies that $\overline{\alpha} = \log (1 + \lambda) \ge \frac{C_0 \log d}{2d}$ when $d$ is sufficiently large. 

From here, the proof of \Cref{lemma:total-weight-given-approximation-modified} from Lemmas \ref{lemma:b-f-small-modified} and \ref{lemma:b-f-large-modified} is the same as the proof of \Cref{lemma:total-weight-given-approximation} given at the end of \Cref{sec:bound_given_approx}, just modified so that both of the upper bounds (\ref{eq:unbounded-case-1}) and (\ref{eq:unbounded-case-2}) are taken into account.

\end{document}